\documentclass[11pt,letterpaper]{article}
\usepackage{booktabs}
\usepackage[letterpaper,margin=0.98in]{geometry}

\usepackage{bm}
\usepackage{authblk}
\usepackage{mlmodern}
\usepackage{inconsolata}

\usepackage{mdframed}

\usepackage{dsfont}

\usepackage{amsfonts,nicefrac}
\usepackage[utf8]{inputenc}
\usepackage[dvipsnames]{xcolor}
\usepackage{graphicx}
\usepackage{multirow}
\usepackage{amsmath}
\usepackage{amsthm}
\usepackage{amssymb}
\usepackage{comment}
\usepackage{cancel}

\usepackage[shortlabels]{enumitem}
\usepackage[style=alphabetic, maxbibnames=15, giveninits=true, maxcitenames=15, natbib=true,
  maxalphanames=10, backend=biber, sorting=nty, backref=true, uniquename=false]{biblatex}
\DefineBibliographyStrings{english}{backrefpage = {page},backrefpages = {pages},}
\DeclareNameAlias{sortname}{given-family}
\renewbibmacro{in:}{\ifentrytype{article}{}{\printtext{\bibstring{in}\intitlepunct}}}

\usepackage{aliascnt}
\usepackage{hyperref}
\hypersetup{
  colorlinks   = true, urlcolor     = blue, linkcolor    = blue, citecolor   = red }
\usepackage[capitalise]{cleveref}
\newtheorem{theorem}{Theorem}[section]

\newaliascnt{lemma}{theorem}
\newtheorem{lemma}[lemma]{Lemma}
\aliascntresetthe{lemma}
\newaliascnt{proposition}{theorem}
\newtheorem{proposition}[proposition]{Proposition}
\aliascntresetthe{proposition}

\newtheorem{assumption}{Assumption}[section]
\theoremstyle{remark}
\newtheorem{remark}{Remark}[section]

\renewcommand{\epsilon}{\varepsilon}

\usepackage{eqparbox}
\usepackage{algorithm,algorithmic}

\usepackage{tikz}

\crefname{equation}{}{}

\floatstyle{ruled}
\newfloat{subroutine}{htbp}{lok}
\floatname{subroutine}{Subroutine}
\crefname{subroutine}{subroutine}{subroutines}
\Crefname{subroutine}{Subroutine}{Subroutines}

\title{Anchored Extra-Proximal Methods: Optimal Higher-Order Methods for Monotone Inclusion Problems \vspace{3mm}

\footnotetext{The authors are listed in alphabetical order.
}}
\usepackage{times}

\renewcommand*{\Affilfont}{\normalsize}

\renewcommand*{\Affilfont}{\normalsize}

\makeatletter
\renewcommand\AB@affilsepx{\quad \protect\Affilfont}
\renewcommand\AB@affilsep{\protect\Affilfont}
\makeatother

\author[$\dagger$]{Ruichen Jiang}
\author[$\ddagger$]{TaeHo Yoon}

\affil[$\dagger$]{Google Research}
\affil[$\ddagger$]{Johns Hopkins University}

\date{}

\usepackage{amsmath}
\usepackage{amssymb}
\usepackage{mathtools}
\usepackage{subcaption}
\usepackage{makecell,threeparttable}

\newcommand\Vector[1]{\bm{#1}}

\newcommand\vc{{\Vector{c}}}
\newcommand\vd{{\Vector{d}}}

\newcommand\vg{{\Vector{g}}}
\newcommand\vh{{\Vector{h}}}

\newcommand\vq{{\Vector{q}}}
\newcommand\vr{{\Vector{r}}}

\newcommand\vv{{\Vector{v}}}
\newcommand\vw{{\Vector{w}}}
\newcommand\vx{{\Vector{x}}}
\newcommand\vy{{\Vector{y}}}
\newcommand\vz{{\Vector{z}}}

\newcommand\MATRIX[1]{\mathbf{#1}}

\newcommand\mI{{\MATRIX{I}}}

\newcommand\bigO{\mathcal{O}}

\newcommand\op{{\mathrm{op}}}
\newcommand{\reals}{\mathbb{R}}

\DeclareMathOperator{\dom}{dom\,}

\newenvironment{bluechanges}{\color{blue}}{}

\newcommand{\mypara}[1]{\noindent\textbf{#1.}\ }

\newcommand{\norm}[1]{\left\|#1\right\|}
\newcommand{\inprod}[2]{\left\langle#1,#2\right\rangle}

\DeclareMathOperator{\Lip}{Lip}

\makeatletter
\def\define@double@struck#1{%
  \expandafter\newcommand\csname o#1\endcsname{\mathds{#1}}%
}
\@tfor\double@struck@letter:=ABCDEFGHIJKLMNOPQRSTUVWXYZ\do{%
  \expandafter\define@double@struck\double@struck@letter
}
\makeatother
\newcommand{\res}{\mathrm{res}}
\newcommand{\gra}{\mathrm{gra}\,}

\newif\ifarxivchanges
\arxivchangesfalse
\newcommand{\arxivchangecolor}{\ifarxivchanges\color{blue}\hypersetup{citecolor=blue}\fi}
\newenvironment{arxivchange}{\begingroup\arxivchangecolor}{\endgroup}
\newcommand{\arxivchanged}[1]{\ifarxivchanges\textcolor{blue}{#1}\else#1\fi}
\newcommand{\arxivchangecaption}{\ifarxivchanges\captionsetup{labelfont={bf,color=blue},textfont={color=blue}}\fi}
\begin{document}

\pagenumbering{gobble}
\maketitle

\begin{abstract}
We characterize, up to logarithmic factors, the optimal higher-order oracle complexity of finding approximate solutions to composite monotone inclusion problems, formed by the sum of a smooth single-valued monotone operator and a maximal monotone set-valued operator.
We introduce
the Anchored Extra-Proximal (AEP) framework, which combines an anchored extrapolation step with an inexact anchored proximal update satisfying a relative-error condition. This framework recovers the composite Fast Extragradient method in the first-order setting and yields natural higher-order extensions through Taylor approximations of the smooth operator at the extrapolated point.
For every $p\geq 2$, assuming that the $(p-1)$th derivative of the single-valued operator is $L_p$-Lipschitz continuous and the initial distance to a solution is at most $R$, we combine AEP with a bisection line search to obtain a $p$th-order method that attains a tangent residual of at most $\varepsilon$ in $\widetilde{\mathcal O}((L_p R^{p}/\varepsilon)^{\frac{2}{3p-1}})$ oracle calls.
We complement this result with a worst-case lower bound of $\Omega((L_p R^p/\varepsilon)^{\frac{2}{3p-1}})$ for every deterministic algorithm in the $p$th-order oracle model, without imposing any prescribed update structure. For the subclass of smooth nonexpansive fixed-point problems, we further establish a lower bound of $\Omega(\min\{R/\varepsilon, {(L_p R^{p}/\varepsilon)}^{\frac{2}{3p-1}}\})$. By choosing between AEP and Halpern iteration, we match this lower bound up to logarithmic factors, thereby also settling the deterministic oracle complexity of smooth nonexpansive fixed-point problems.

\end{abstract}

\newpage
\pagenumbering{arabic}

\section{Introduction}

A broad class of optimization problems, including convex-concave min-max
optimization, monotone variational inequalities, and nonexpansive fixed-point problems, can be formulated as finding a zero of the sum of two maximal monotone operators~\cite{Facchinei2003}. Specifically, {we consider the problem}
\begin{equation}\label{eq:monotone}
    {\text{find }\vz \in \reals^d \text{ such that} \,\, 0 \in \oF(\vz) + \oH(\vz),}
\end{equation}
where $\oF\colon \reals^d \rightarrow \reals^d$ is a {continuous}
monotone operator and $\oH\colon \reals^d \rightrightarrows \reals^d$ is a set-valued maximal monotone operator. Throughout the paper, we assume that Problem~\eqref{eq:monotone}
admits a solution and fix an arbitrary solution $\vz^*$.  In this paper, we are interested in finding approximate solutions to Problem~\eqref{eq:monotone} with a small \emph{tangent residual}~\cite{cai2022finite,cai2023accelerated}, which measures the distance from the origin to the set $\oF(\vz) + \oH(\vz)$. It reduces to the familiar gradient norm in the setting of unconstrained minimization or min-max optimization where $\oH = 0$, and serves as an upper bound on several commonly used optimality measures, including the duality gap, as we review in Appendix~\ref{subsec:optimality}.

Various first-order iterative methods, which rely only on evaluations of $\oF$ and the resolvent of $\oH$, have been proposed for solving the monotone inclusion problem~\eqref{eq:monotone} and its special cases of min-max optimization and variational inequalities. Prominent examples include the Extragradient
method~\cite{Korpelevich1976,Nemirovski2004} and Popov's method~\cite{Popov1980}, which was later popularized as the Optimistic Gradient method~\cite{Rakhlin2013optimization,Daskalakis2018}. Although these classical methods achieve the optimal complexity of $\bigO(\varepsilon^{-1})$ for the weaker \emph{duality gap} criterion, their complexity with respect to the tangent residual is only $\bigO(\varepsilon^{-2})$~\cite{golowich2020last,golowich2020tight,gorbunov2022last,gorbunov2022extragradient,cai2022finite}, which is suboptimal. The optimal oracle complexity for reducing the tangent residual was established relatively recently. For unconstrained convex-concave min-max optimization (a special case of Problem~\eqref{eq:monotone}), \citet{yoon2021accelerated} proposed the Extra Anchored Gradient (EAG) method, which combines extragradient steps with anchoring and achieves a complexity of $\bigO(\varepsilon^{-1})$ in terms of the gradient/operator norm. They also established a matching lower bound of $\Omega(\varepsilon^{-1})$ for deterministic first-order algorithms, thereby proving that EAG is optimal up to a constant factor. Subsequently, \citet{lee2021fast} introduced the Fast Extragradient (FEG) method, improving the constant in the complexity bound and extending the same rate to structured nonmonotone settings, including negatively comonotone min-max problems. \citet{cai2024accelerated} further extended EAG and FEG to the {composite} monotone inclusion problem of the form \eqref{eq:monotone} and to comonotone inclusion problems.

Although the first-order oracle complexity is understood up to constant factors, the optimal oracle complexity of second- and higher-order methods remains open.
Second- and higher-order counterparts of the extragradient and optimistic gradient methods have been proposed in the literature. For any $p\geq 2$, assuming the $(p-1)$th derivative of $\oF$ is Lipschitz continuous, these $p$th-order methods
attain a complexity of $\bigO(\varepsilon^{-2/(p+1)})$ for an averaged iterate under the {dual gap} criterion~\cite{monteiro2010complexity,monteiro2012iteration,bullins2022higher,adil2022optimal,lin2025perseus,jiang2024generalized} and a complexity of $\bigO(\varepsilon^{-2/p})$ for the best iterate under the tangent-residual criterion~\cite{monteiro2010complexity,monteiro2012iteration,lin2025perseus}. Most recently, for convex-concave min-max optimization, Minimax-AIPE and its $p$th-order extension~\citep{chen25solving,chen2026solving} use multi-loop accelerated proximal schemes with high-order regularized subproblem solvers, achieving complexities of $\tilde{\bigO}(\varepsilon^{-4/7})$ for the duality gap and $\tilde{\bigO}(\varepsilon^{-4/(3p+1)})$ for the tangent residual, respectively. For smooth monotone variational inequalities, \citet{chen2026halpern} proposed a two-loop method consisting of an outer inexact Halpern iteration and an inner resolvent solver based on the restarted Newton proximal extragradient (NPE) method for $p=2$ or an anchored tensor method for general $p$, attaining $\tilde{\bigO}(\varepsilon^{-1/p})$ complexity for the proximal residual, {which was the best known upper bound prior to the present work.  However, this rate still left a polynomial gap to the lower bound $\Omega(\varepsilon^{-2/(3p-1)})$ established in \cite{chen2026solving} for a restricted class of $p$th-order tensor algorithms.}

This gap raises the following fundamental question:
\begin{quote}
  \textit{What is the optimal oracle complexity of solving the monotone inclusion problem in~\eqref{eq:monotone} using second-order and higher-order information?}
\end{quote}

\paragraph{Contributions.}
We settle the optimal deterministic $p$th-order oracle complexity for the monotone inclusion problems in \eqref{eq:monotone} and for the subclass of smooth nonexpansive fixed-point problems, up to logarithmic factors.  Our main results are as follows.
\begin{itemize}
  \item \textbf{Monotone inclusion problems.} We introduce a general \emph{Anchored Extra-Proximal} (AEP) framework and derive from it a
  concrete $p$th-order method instantiating it, for each $p\geq 2$. Assuming that the $(p-1)$th derivative of $\oF$ is $L_p$-Lipschitz continuous and the initial distance to a solution is at most $R$, the resulting method returns a point $\vz$ with tangent residual at most $\varepsilon$ using ${\bigO}\bigl((L_p R^{p}/\varepsilon)^{\frac{2}{3p-1}} \log (L_p R^{p}/\varepsilon)\bigr)$ $p$th-order oracle calls. Under the same assumptions, we prove that any deterministic algorithm in the $p$th-order oracle model requires $\Omega\bigl((L_p R^{p}/\varepsilon)^{\frac{2}{3p-1}}\bigr)$ oracle calls in the worst case to find a point with tangent residual at most $\varepsilon$.

  \item \textbf{Nonexpansive fixed-point problems.} For nonexpansive operators whose \((p-1)\)th derivative is \(L_p\)-Lipschitz, assuming that the initial distance to a fixed point is at most $R$, we establish a worst-case lower bound of
$\Omega\!\left(\min\{R/\varepsilon, {(L_p R^{p}/\varepsilon)}^{\frac{2}{3p-1}}\}\right)$ oracle calls to attain a fixed-point residual of at most $\varepsilon$. By choosing between AEP and Halpern iteration, we attain this rate up to logarithmic factors.
\end{itemize}

Beyond these complexity guarantees, our framework provides a direct and unified algorithmic interpretation. In contrast to the nested-loop constructions in~\cite{chen25solving,chen2026solving,chen2026halpern}, AEP can be viewed as a natural higher-order generalization of the FEG method~\cite{lee2021fast,cai2024accelerated}. Motivated by the proximal-point perspective of~\cite{monteiro2010complexity,monteiro2012iteration,Mokhtari2020a,Mokhtari2020,yoon2025accelerated,jiang2025generalized}, we begin with an ideal anchored proximal point method {with time-varying step sizes} and show that anchoring yields a sharper convergence guarantee in terms of the tangent residual than the classical proximal point method~\cite{Martinet1970,Rockafellar1976}.
We then introduce the AEP framework, which couples an anchored extrapolation step with an inexact anchored proximal update satisfying a hybrid proximal extragradient (HPE)-type relative-error condition.

This framework recovers first-order anchor acceleration and naturally extends it to higher orders. In particular,
{the FEG method with optimal $\bigO(\varepsilon^{-1})$ complexity is a first-order, unconstrained special case of AEP, where the inexact proximal update direction is taken as $\oF$ evaluated at the extrapolated point.
When we have access to higher-order derivatives of $\oF$, we instead use the $(p-1)$th-order Taylor model centered at the extrapolated point as the inexact proximal update, which is the $p$th-order AEP method.}
The main technical challenge of AEP is to select a step size {small enough so that}
the Taylor-model update satisfies the required relative-error condition and simultaneously {large enough to retain sufficient} progress needed for acceleration. We resolve this through a bisection-based line search that terminates after $\bigO(\log(1/\varepsilon))$ trials per iteration.

To establish our lower bounds for both problem classes, we adapt the hard construction of~\citet{jang2026higher} for contractive fixed-point problems to
nonexpansive fixed-point and monotone inclusion settings of our interest,
and develop a corresponding resisting-oracle argument.
{Notably, our lower bound for monotone inclusions applies to unconstrained problems, and to every deterministic algorithm under the $p$th-order oracle model;
hence it is more general than the lower bound from \cite{chen2026solving},
which constructs a constrained min-max problem and proves the lower bound over the algorithm class using specific types of tensor steps.}

\paragraph{Concurrent Work.}
While this manuscript was being completed, we became aware of two concurrent and independent works establishing optimal $p$th-order oracle complexity results for smooth monotone variational inequalities. \citet{zhang2026matching} obtain an upper bound of $\tilde{\bigO}(\varepsilon^{-2/(3p-1)})$ and a matching lower bound up to logarithmic factors, while \citet{zhang2026optimal} achieve the logarithm-free bound $\bigO(\varepsilon^{-2/(3p-1)})$. The former uses a multi-hierarchy acceleration scheme built around inexact extrapolated Halpern iteration, whereas the latter uses a double-loop method combining inexact accelerated Halpern iteration with an anchored tensor inner solver.

In contrast, our work
focuses on a clear conceptual connection to the existing anchor acceleration literature.
Indeed, our AEP framework becomes the FEG method in the first-order case, and naturally extends it to higher orders through Taylor-model approximations.
This direct construction is substantially simpler than the multi-hierarchy acceleration scheme in~\cite{zhang2026matching}. Compared with~\cite{zhang2026optimal}, our method selects step sizes adaptively through line search rather than following a prescribed schedule and, for unconstrained second-order problems, requires only a linear-system solve at each trial rather than a cubic-regularized Newton-type subproblem.
We also establish lower bounds for any deterministic $p$th-order algorithms and characterize the optimal complexity of smooth nonexpansive fixed-point problems up to logarithmic factors. Appendix~\ref{appen:concurrent_work} details the algorithmic differences, logarithmic factors, and lower-bound constructions.

\section{Preliminaries}

In this paper, we measure approximate optimality using the \emph{tangent residual}~\cite{cai2022finite,cai2023accelerated}, defined by
\begin{equation}\label{eq:tangent_residual}
  \res^{\mathrm{tan}}(\vz)
  := \inf_{\vv \in \oH(\vz)} \|\oF(\vz)+\vv\|
  = \mathrm{dist}\bigl(0,\oF(\vz)+\oH(\vz)\bigr).
\end{equation}
One appealing property is that any $\vv \in \oH(\vz)$ provides the computable
upper bound $\res^{\mathrm{tan}}(\vz) \leq \|\oF(\vz)+\vv\|$.
The tangent residual also controls other standard optimality measures for Problem~\eqref{eq:monotone} and its special cases, including forward-backward and proximal residuals, as well as the strong gap for variational inequalities on bounded domains. We collect the definitions and comparisons in Appendix~\ref{subsec:optimality}. We next describe three important special cases of Problem~\eqref{eq:monotone} before reviewing the anchoring mechanism underlying our algorithms.

\subsection{Special Cases of Monotone Inclusions}\label{subsec:special_cases}

\mypara{Min-Max Optimization} Consider the composite min-max problem
\begin{equation}\label{eq:min-max}
  \min_{\vx \in \reals^m}\,\max_{\vy \in \reals^n}\; \ell(\vx,\vy) \coloneq f(\vx, \vy) + h_1(\vx) - h_2(\vy),
\end{equation}
where $f\colon \reals^m \times \reals^n \to \reals$ is smooth, and $h_1$ and $h_2$ are proper, closed, convex functions. The functions $h_1$ and $h_2$ can, for example, encode convex constraints through indicator functions or nonsmooth regularizers.
When $f$ is convex in $\vx$ and concave in $\vy$, Problem~\eqref{eq:min-max} is an instance of the monotone inclusion problem in~\eqref{eq:monotone}. Indeed, with $\vz = (\vx,\vy)$, define
$\oF(\vz) =
(\nabla_{\vx}f(\vx,\vy),-\nabla_{\vy}f(\vx,\vy))$ and
$\oH(\vz)=\partial h_1(\vx)\times\partial h_2(\vy)$,
where $\partial h$ denotes the subdifferential of a convex function $h$. Then $(\vx^*,\vy^*)$ is a saddle point of Problem~\eqref{eq:min-max} if and only if $\vz^*=(\vx^*,\vy^*)$ satisfies $0 \in \oF(\vz^*)+\oH(\vz^*)$.

\mypara{Variational Inequalities} A composite (Stampacchia) variational inequality is given by
\begin{equation}\label{eq:VI}
  \text{find }\vz^* \in \dom h \text{ such that} \,\, \langle \oF(\vz^*), \vz^* - \vz\rangle + h(\vz^*) - h(\vz)  \leq 0 \quad \forall \vz \in \dom h,
\end{equation}
where $\oF \colon \reals^d \to \reals^d$ is smooth and $h$ is a proper, closed, convex function.
Composite variational inequalities encompass convex-concave min-max optimization and arise more broadly in equilibrium problems~\cite{Facchinei2003}.
If $\oF$ is monotone, then
Problem~\eqref{eq:VI} is equivalent to the monotone inclusion $0 \in \oF(\vz^*) + \partial h(\vz^*)$, and hence is an instance of Problem~\eqref{eq:monotone} with $\oH = \partial h$.

\mypara{Fixed-Point Problems} Given an operator $\oT\colon \reals^d \to \reals^d$,
the fixed-point problem is formulated as
\begin{equation}\label{eq:fixed-point}
    \text{find }\vz \in \reals^d \text{ such that} \,\, \vz = \oT(\vz) .
\end{equation}
Suppose that $\oT$ is nonexpansive, meaning that $\|\oT(\vz)-\oT(\vz')\| \leq \|\vz-\vz'\| $ for all $\vz,\vz' \in \reals^d$.
Then $\oF := \oI - \oT$ is continuous and monotone
and
Problem~\cref{eq:fixed-point} is a special case of Problem~\eqref{eq:monotone} with $\oH \equiv 0$. In this case, the tangent residual coincides with the fixed-point residual: $\res^{\mathrm{tan}}(\vz) = \|\vz - \oT(\vz)\|$. {Conversely, Problem~\eqref{eq:monotone} can be reformulated as a fixed-point problem using the reflected resolvent of $\oF+\oH$, as discussed in the next subsection.}

\subsection{Halpern Iteration and Anchoring}
\label{subsec:halpern}
A key algorithmic idea underlying our methods is \emph{anchoring}, which dates back to~\citet{halpernFixedPointsNonexpanding1967} and has more recently been developed as an acceleration mechanism for fixed-point problems~\cite{SabachShtern2017_first,liederConvergenceRateHalperniteration2021,contrerasOptimalErrorBounds2023}. Given a nonexpansive $\oT$ with a fixed point $\vz^*$, the optimal Halpern iteration starts from $\vz_0$ and computes
\begin{equation}\label{eq:halpern}
  \vz_{k+1} = \beta_k \vz_0 + (1-\beta_k)\oT(\vz_k),
  \qquad \beta_k = \frac{1}{k+2}.
\end{equation}
Each update retracts the fixed-point step $\oT(\vz_k)$ to the initial point $\vz_0$, called the \emph{anchor}, whose weight decreases over time. Its fixed-point residual satisfies
$\|\vz_k-\oT(\vz_k)\|\leq 2\|\vz_0-\vz^*\|/(k+1)$
\cite{liederConvergenceRateHalperniteration2021},
which is exactly worst-case optimal among deterministic
methods using evaluations of $\oT$~\cite{park2022exact}.

To connect this iteration to Problem~\eqref{eq:monotone}, fix $a>0$ and consider the \emph{reflected resolvent}
\begin{equation}\label{eq:resolvents}
  \oT := 2\oJ_{a(\oF+\oH)}-\oI,
  \qquad \oJ_{a(\oF+\oH)} := \bigl(\oI+a(\oF+\oH)\bigr)^{-1}.
\end{equation}
This operator is nonexpansive, and its fixed points are precisely the solutions
of~\eqref{eq:monotone}.  More generally,
the mapping $\oM\mapsto 2(\oI+a\oM)^{-1}-\oI$ gives a one-to-one correspondence between maximal monotone operators and nonexpansive maps~\cite{bauschkeConvexAnalysisMonotone2017}. Under this correspondence, Halpern iteration is equivalent to the accelerated proximal point method of~\citet{kim2021accelerated}, as shown in~\cite{ryuLargescaleConvexOptimization2022,contrerasOptimalErrorBounds2023}. However, each iteration requires evaluating the full resolvent of $\oF+\oH$, which generally involves solving an implicit subproblem.

The challenge is therefore to retain the acceleration from anchoring using only evaluations of $\oF$ and the resolvent of $\oH$. In this direction, anchoring was applied to min-max optimization and monotone inclusions~\cite{RyuYuanYin2019_ode,diakonikolas2020halpern}, and notably, \citet{diakonikolas2020halpern} first achieved near-optimal $\tilde\bigO (\epsilon^{-1})$ complexity with respect to the proximal residual. \citet{yoon2021accelerated} then combined anchoring with extragradient steps to establish the optimal $\bigO(\epsilon^{-1})$ complexity in the unconstrained setting, which was subsequently extended to weakly nonmonotone operators \citep{lee2021fast,cai2024accelerated}.
Further related work is reviewed in the appendix.

\section{Anchored Extra-Proximal Method}\label{sec:aep}

In this section, we develop a general Anchored Extra-Proximal (AEP) framework underlying our second- and higher-order methods. We first study an idealized anchored proximal point method (PPM) in Section~\ref{subsec:anchored-ppm} and establish its convergence in terms of the tangent residual. Because each iteration requires evaluating the resolvent of the full operator $\oF+\oH$, this idealized method is not generally implementable under our oracle model. We therefore introduce the AEP framework in Section~\ref{subsec:aep}, replacing the exact implicit proximal-point update with an HPE-type relative-error condition.
We show that this relaxation preserves the key convergence guarantee while allowing the exact update to be replaced by a suitable approximation. Finally, we recover FEG as a special case of AEP and recover its known convergence guarantee.

\subsection{Proximal Point Method with Anchoring}\label{subsec:anchored-ppm}
To better motivate our framework, we first study an anchored variant of the PPM, which may also be of independent interest. Starting from an initial point $\vz_0$, at iteration $k\geq 0$, anchored PPM computes
\begin{equation}\label{eq:anchored_PPM}
  \vz_{k+1} = \oJ_{a_k(\oF+\oH)}\left( \frac{a_k}{A_k+a_k} \vz_0 + \frac{A_k}{A_k + a_k} \vz_k\right).
\end{equation}
Here, $\oJ_{a(\oF+\oH)}$ is the resolvent defined in \eqref{eq:resolvents}, $a_k>0$ is the step size, and $A_k := \sum_{i=0}^{k-1} a_i$ with $A_0 = 0$ is the cumulative step size. Equivalently, $\vz_{k+1}$ is the unique solution to the monotone inclusion subproblem
\begin{equation}\label{eq:anchored_PPM_inclusion}
  0 \in  a_k \oF(\vz_{k+1}) + a_k \oH(\vz_{k+1}) + \vz_{k+1}  - \left( \frac{a_k}{A_{k}+a_k} \vz_0 + \frac{A_k}{A_k+a_k}\vz_k \right).
\end{equation}
Compared with the standard PPM~\cite{Martinet1970,Rockafellar1976}, the key difference is that the resolvent is evaluated at a convex combination of the anchor $\vz_0$ and the current iterate $\vz_k$, with weights determined by the step sizes. The following proposition establishes a last-iterate convergence guarantee in terms of the tangent residual. We defer its proof to Appendix~\ref{appen:anchored_PPM}.

\begin{proposition}\label{prop:anchored_PPM}
  Let $\{\vz_k\}$ be generated by the anchored PPM in \eqref{eq:anchored_PPM}. Then for every $k \geq 0$, we have
  \begin{equation*}
    \res^{\mathrm{tan}}(\vz_{k+1}) \leq \frac{2\|\vz_0-\vz^*\|}{A_{k+1}}.
  \end{equation*}
\end{proposition}
Proposition~\ref{prop:anchored_PPM} bounds the tangent residual in terms of the cumulative step size $A_{k+1} = \sum_{i=0}^k a_i$. Because this idealized method imposes no upper bound on the step sizes, its convergence rate can formally be made arbitrarily fast by choosing sufficiently large step sizes. This observation does not, however, directly yield a practical algorithm: each iteration requires solving exactly the monotone inclusion subproblem in~\eqref{eq:anchored_PPM_inclusion},
which generally has no closed-form solution and is not directly available under our oracle model. This motivates the relative-error criterion developed in the next subsection, which allows the proximal subproblem to be solved inexactly while preserving the essential convergence guarantee of Proposition~\ref{prop:anchored_PPM}.

\mypara{Connection with Halpern Iteration}
Consider the constant-step size setting $a_k\equiv a$, for which $A_k=ak$. Define $\oT:=\oJ_{a(\oF+\oH)}$ and $\vy_k:=\frac{1}{k+1}\vz_0+\frac{k}{k+1}\vz_k$. Then $\vy_0=\vz_0$ and $\vz_{k+1}=\oT\vy_k$, so
$\vy_{k+1}
  =
  \frac{1}{k+2}\vy_0
  +
  \frac{k+1}{k+2}\oT\vy_{k}$.
Thus, the auxiliary sequence $\{\vy_k\}$ is precisely the Halpern iteration in \eqref{eq:halpern} applied to the resolvent $\oT$, rather than the reflected resolvent
associated with accelerated PPM in
Section~\ref{subsec:halpern}.

\begin{remark}[Time-Varying Step Sizes]
  An important feature of Proposition~\ref{prop:anchored_PPM} is that it allows an arbitrary sequence of positive, time-varying step sizes $\{a_k\}$. This flexibility will be crucial for obtaining the accelerated convergence rates of our second- and higher-order methods. Indeed, varying $a_k$ changes the resolvent $\oJ_{a_k(\oF+\oH)}$ from one iteration to the next. Consequently, Proposition~\ref{prop:anchored_PPM} is not a direct consequence of standard analyses of Halpern iteration, which consider repeated application of a fixed nonexpansive operator. This feature also distinguishes our method from the accelerated PPM in~\cite{kim2021accelerated}, which uses a constant proximal step size. At a high level, it likewise distinguishes our approach from that in~\cite{chen2026halpern}, whose outer iteration seeks to approximate Halpern iteration associated with a fixed, large step size.
\end{remark}

\subsection{Anchored Extra-Proximal Framework}\label{subsec:aep}

We are now ready to introduce the AEP framework, where each iteration consists of two principal updates. First, we perform an anchored extrapolation step to construct an intermediate point $\vz_{k+\frac{1}{2}}$. We then compute an inexact solution $\vz_{k+1}$ of the anchored proximal subproblem in~\eqref{eq:anchored_PPM_inclusion}. The term ``extra'' refers to this extrapolation step preceding the proximal update. In the same spirit as HPE~\cite{monteiro2010complexity}, the inexact solution is required to satisfy a relative-error criterion that controls its tangent residual by the distance $\|\vz_{k+1}-\vz_{k+\frac{1}{2}}\|$. More precisely, each iteration proceeds as follows.

\begin{enumerate}[(i)]\setlength{\itemsep}{0pt}
  \item At iteration $k$, we maintain a vector $\vv_k\in\oH(\vz_k)$. We first check whether $ \|\oF(\vz_k)+\vv_k\|\leq\varepsilon$.
  If so, the algorithm terminates and returns $\vz_k$; otherwise, it proceeds to the next step.

  \item Given a step size $a_k>0$ and the current cumulative step size $A_k$, we compute the extrapolated point
  \begin{equation}\label{eq:AEP_k+1/2}
    \vz_{k+\frac{1}{2}}
    =
    \frac{a_k}{A_k+a_k}\vz_0
    +
    \frac{A_k}{A_k+a_k}
    \bigl(\vz_k-a_k(\oF(\vz_k)+\vv_k)\bigr).
  \end{equation}

  \item We then find the next iterate $\vz_{k+1}$ and a residual $\vr_{k+1}$ satisfying
  \begin{align}
    \vr_{k+1}
    &\in
    a_k\oF(\vz_{k+1})
    +
    a_k\oH(\vz_{k+1})
    +
    \vz_{k+1}
    -
    \left(
      \frac{a_k}{A_k+a_k}\vz_0
      +
      \frac{A_k}{A_k+a_k}\vz_k
    \right),
    \label{eq:AEP_k+1}
    \\
    \|\vr_{k+1}\|
    &\leq
    \rho\|\vz_{k+1}-\vz_{k+\frac{1}{2}}\|,
    \label{eq:AEP_error_condition}
  \end{align}
  where $\rho\in[0,1]$ is a prescribed constant. In particular, $\vz_{k+1}$ can be viewed as an approximate solution to the anchored proximal subproblem in \eqref{eq:anchored_PPM_inclusion}, and Condition \eqref{eq:AEP_error_condition} imposes a relative-error criterion so that the tangent residual of the subproblem at $\vz_{k+1}$ is at most $\rho\|\vz_{k+1}-\vz_{k+\frac{1}{2}}\|$.

  \item Finally, we update the cumulative step size and the vector associated with $\oH$ according to
  \begin{equation}\label{eq:v_plus_AEP}
    A_{k+1}=A_k+a_k,
    \quad
    \vv_{k+1}
    =
    \frac{1}{a_k}
    \left(
      \frac{a_k}{A_k+a_k}\vz_0
      +
      \frac{A_k}{A_k+a_k}\vz_k
      -
      a_k\oF(\vz_{k+1})
      -
      \vz_{k+1}
      +
      \vr_{k+1}
    \right).
  \end{equation}
  Note that the residual inclusion~\eqref{eq:AEP_k+1} ensures that $\vv_{k+1}\in\oH(\vz_{k+1})$.
\end{enumerate}

\begin{arxivchange}
\begin{algorithm}[!t]\arxivchangecolor\arxivchangecaption\small
    \caption{Anchored Extra-Proximal Method}\label{alg:aep}
    \begin{algorithmic}[1]
        \STATE \textbf{Input:} initial point $\vz_0\in \dom\oH$, parameter $\rho \in[0,1]$, required accuracy $\varepsilon$
        \STATE \textbf{Initialize:} set $\vv_0 \in \oH(\vz_0)$ and $A_0 \leftarrow 0$
        \FOR{iteration $k=0,1,\ldots$}
            \IF{$\|\oF(\vz_k) + \vv_k\| \leq \varepsilon$}
            \STATE Return $\vz_k$
            \ELSE
            \STATE Choose step size $a_k$ and compute $\vz_{k+\frac{1}{2}} = \frac{a_k}{A_{k}+a_k} \vz_0 + \frac{A_k}{A_k+a_k} (\vz_k - a_k (\oF(\vz_k) +\vv_k))$
            \STATE Find $\vz_{k+1}$ with $\vr_{k+1}$ such that \vspace{-1em}
            \begin{equation*}\vspace{-1em}
              \vr_{k+1} \in  a_k \oF(\vz_{k+1}) + a_k \oH(\vz_{k+1}) + \vz_{k+1}  - \Bigl( \frac{a_k}{A_{k}+a_k} \vz_0 + \frac{A_k}{A_k+a_k}\vz_k \Bigr), \quad \|\vr_{k+1}\| \leq \rho \|\vz_{k+1} - \vz_{k+\frac{1}{2}}\|.
            \end{equation*}
            \STATE Update \vspace{-1em}
            \begin{equation*}\vspace{-1em}
               A_{k+1} = A_k + a_k, \quad \vv_{k+1} = \frac{1}{a_k} \left(\vr_{k+1} - a_k \oF(\vz_{k+1}) - \vz_{k+1} + \left( \frac{a_k}{A_{k}+a_k} \vz_0 + \frac{A_k}{A_k+a_k}\vz_k \right) \right).
            \end{equation*}
            \ENDIF
        \ENDFOR
    \end{algorithmic}
\end{algorithm}
\end{arxivchange}

\begin{arxivchange}
The complete procedure is summarized in Algorithm~\ref{alg:aep}. Our main result for the framework shows that, despite computing the proximal step only inexactly, AEP retains the same tangent-residual guarantee as the ideal anchored PPM in Proposition~\ref{prop:anchored_PPM}.
The convergence analysis relies on the following potential descent lemma.
\end{arxivchange}

\begin{arxivchange}
\begin{lemma}\label{lem:Lyapunov_AEP}
Let $\{\vz_k\}_{k\geq0}$ be generated by AEP in
Algorithm~\ref{alg:aep} with $\rho\in[0,1]$. Set $V_0:=0$ and, for every
$k\geq1$, define
\begin{equation}\label{eq:Lyapunov}
  V_k
  :=A_k\langle \oF(\vz_k)+\vv_k,\vz_k-\vz_0\rangle
  +\frac{A_k^2}{2}\|\oF(\vz_k)+\vv_k\|^2.
\end{equation}
Then, for every $k\geq0$,
\begin{equation}\label{eq:potential_diff}
  V_{k+1}
  \leq V_k
  -\frac{(1-\rho^2)A_{k+1}^2}{2a_k^2}
   \|\vz_{k+1}-\vz_{k+\frac{1}{2}}\|^2.
\end{equation}
\end{lemma}
\end{arxivchange}

\begin{arxivchange}
This potential shares the same two-term structure as those used to analyze anchored first-order methods in \cite{yoon2021accelerated,lee2021fast}, but one key difference is that the weights depend on the cumulative step size $A_k$.
We defer the complete proof to Appendix~\ref{appen:aep}. Using \cref{lem:Lyapunov_AEP}, we obtain our main convergence result for AEP.
\end{arxivchange}

\begin{theorem}\label{thm:aep}
  Let $\{\vz_k\}_{k\geq0}$ be generated by AEP in Algorithm~\ref{alg:aep} with $\rho\in[0,1]$. Then, for every $k\geq0$,
  \begin{equation*}
    \res^{\mathrm{tan}}(\vz_{k+1})
    \leq
    \|\oF(\vz_{k+1})+\vv_{k+1}\|
    \leq
    \frac{2\|\vz_0-\vz^*\|}{A_{k+1}}.
  \end{equation*}
\end{theorem}

To obtain a concrete algorithm from this abstract framework, it remains to specify how to construct $\vz_{k+1}$ and $\vr_{k+1}$ satisfying~\eqref{eq:AEP_k+1}--\eqref{eq:AEP_error_condition}, as well as how to select the step size $a_k$. Our key idea is to replace $\oF$ in the anchored proximal subproblem by a more tractable surrogate operator $\oP_k$. We then compute $\vz_{k+1}$ by solving
\begin{equation}\label{eq:surrogate_inclusion}
  0
  \in
  a_k\oP_k(\vz_{k+1})
  +
  a_k\oH(\vz_{k+1})
  +
  \vz_{k+1}
  -
  \left(
    \frac{a_k}{A_k+a_k}\vz_0
    +
    \frac{A_k}{A_k+a_k}\vz_k
  \right).
\end{equation}
In this case, the residual in~\eqref{eq:AEP_k+1} becomes $\vr_{k+1}=a_k\bigl(\oF(\vz_{k+1})-\oP_k(\vz_{k+1})\bigr)$. Consequently, the relative-error condition in \eqref{eq:AEP_error_condition} becomes
\begin{equation}\label{eq:surrogate_error_condition}
  a_k
  \left\|
    \oF(\vz_{k+1})-\oP_k(\vz_{k+1})
  \right\|
  \leq
  \rho\|\vz_{k+1}-\vz_{k+\frac{1}{2}}\|,
\end{equation}
which depends directly on the approximation error of the surrogate model. Choosing $\oP_k$ as the $(p-1)$th-order Taylor model of $\oF$ around the intermediate point $\vz_{k+\frac{1}{2}}$ yields our $p$th-order method. We further develop a line-search procedure that selects $a_k$ so that~\eqref{eq:surrogate_error_condition} is satisfied while ensuring that the step size remains sufficiently large to achieve fast convergence.

\begin{arxivchange}
\mypara{Comparison with HPE}
The classical HPE framework of~\citet{solodov1999hybrid,monteiro2010complexity} can similarly be viewed as an inexact approximation of the standard PPM based on a relative-error criterion and an extragradient correction. At iteration $k$, it first finds an intermediate point $\vz_{k+\frac{1}{2}}$ and a residual $\vr_{k+\frac{1}{2}}$ satisfying
\begin{equation*}
  \vr_{k+\frac{1}{2}}
  \in
  a_k\oF(\vz_{k+\frac{1}{2}})
  +
  a_k\oH(\vz_{k+\frac{1}{2}})
  +
  \vz_{k+\frac{1}{2}}
  -
  \vz_k,
  \qquad
  \|\vr_{k+\frac{1}{2}}\|
  \leq
  \rho\|\vz_{k+\frac{1}{2}}-\vz_k\|.
\end{equation*}
It then defines
$
  \vv_{k+\frac{1}{2}}
  =
  \frac{1}{a_k}
  \left(
    \vz_k
    -
    a_k\oF(\vz_{k+\frac{1}{2}})
    -
    \vz_{k+\frac{1}{2}}
    +
    \vr_{k+\frac{1}{2}}
  \right)
  \in
  \oH(\vz_{k+\frac{1}{2}})
$
and performs the extragradient update
\begin{equation*}
  \vz_{k+1}
  =
  \vz_k
  -
  a_k
  \left(
    \oF(\vz_{k+\frac{1}{2}})
    +
    \vv_{k+\frac{1}{2}}
  \right).
\end{equation*}
\end{arxivchange}

\begin{arxivchange}
There are two main differences between the two frameworks. First, classical HPE computes an inexact proximal point and then performs an extragradient update using the operator at that point. AEP instead first constructs the anchored extrapolation point $\vz_{k+\frac{1}{2}}$ and then computes an inexact anchored proximal point $\vz_{k+1}$; the abstract AEP framework does not require an operator evaluation at the extrapolated point. Second, AEP incorporates the anchor $\vz_0$ into both updates.
\end{arxivchange}

\begin{arxivchange}
\mypara{Fast Extragradient as First-Order AEP}
As a first application of the AEP framework, we show that the composite FEG method of~\cite{cai2024accelerated} can be recovered as a special case. When the operator $\oF$ is $L_1$-Lipschitz continuous,  a natural choice of the surrogate operator in~\eqref{eq:surrogate_inclusion} is $\oP_k(\vz):=\oF(\vz_{k+\frac{1}{2}})$. Indeed, by Lipschitz continuity the relative-error condition~\eqref{eq:surrogate_error_condition} holds whenever $a_kL_1\leq\rho$. To recover FEG, we set $\rho=1$ and choose the constant step size $a_k\equiv\eta:=\frac{1}{L_1}$. Since $A_k=k\eta$, Algorithm~\ref{alg:aep} coincides
with the composite FEG method of~\cite{cai2024accelerated}
in the monotone setting, corresponding to a comonotonicity
parameter of zero. The explicit updates are
\begin{align}
    \vz_{k+\frac{1}{2}}
    &=
    \frac{1}{k+1}\vz_0
    +
    \frac{k}{k+1}
    \bigl(
      \vz_k-\eta(\oF(\vz_k)+\vv_k)
    \bigr),
    \notag\\
    \vz_{k+1}
    &=
    (\oI+\eta\oH)^{-1}
    \left(
      \frac{1}{k+1}\vz_0
      +
      \frac{k}{k+1}\vz_k
      -
      \eta\oF(\vz_{k+\frac{1}{2}})
    \right),
    \notag\\
    \vv_{k+1}
    &=
    \frac{1}{\eta}
    \left(
      \frac{1}{k+1}\vz_0
      +
      \frac{k}{k+1}\vz_k
      -
      \eta\oF(\vz_{k+\frac{1}{2}})
      -
      \vz_{k+1}
    \right)
    \in\oH(\vz_{k+1}).
  \label{eq:composite_FEG}
\end{align}
As an immediate consequence of Theorem~\ref{thm:aep}, we obtain
\begin{equation*}
    \res^{\mathrm{tan}}(\vz_{k+1})
    \leq
    \|\oF(\vz_{k+1})+\vv_{k+1}\|
    \leq
    \frac{2L_1\|\vz_0-\vz^*\|}{k+1},
  \end{equation*}
  which recovers the convergence guarantee in~\cite[Theorem~4.1]{cai2024accelerated} for this setting.
\end{arxivchange}

\section{Anchored Extra-Proximal Newton Method}\label{sec:AEP_Newton}

In this section, we instantiate the AEP framework in Algorithm~\ref{alg:aep} using a second-order oracle, where we have access to both $\oF$ and its Jacobian $D\oF$. We call the resulting method Anchored Extra-Proximal Newton (AEP Newton). Our analysis relies on the following smoothness assumption, which is  standard in the literature.

\begin{assumption}\label{assum:Lips_Jacobian}
  The operator $\oF$ has an $L_2$-Lipschitz continuous Jacobian, i.e., $\|D\oF(\vz)-D\oF(\vz')\|_{\op}\leq L_2\|\vz-\vz'\|$
  for all $\vz,\vz'\in\reals^d$.
\end{assumption}

Under Assumption~\ref{assum:Lips_Jacobian}, the standard Taylor remainder bound gives
\begin{equation}\label{eq:second_order_error}
  \left\|
    \oF(\vz)
    -
    \oF(\vz_{k+\frac{1}{2}})
    -
    D\oF(\vz_{k+\frac{1}{2}})
    (\vz-\vz_{k+\frac{1}{2}})
  \right\|
  \leq
  \frac{L_2}{2}
  \|\vz-\vz_{k+\frac{1}{2}}\|^2.
\end{equation}
This naturally suggests using the affine Taylor model $\oP_k(\vz)
  :=
  \oF(\vz_{k+\frac{1}{2}})
  +
  D\oF(\vz_{k+\frac{1}{2}})
  (\vz-\vz_{k+\frac{1}{2}})$
as the surrogate operator in~\eqref{eq:surrogate_inclusion}. With this choice, the surrogate subproblem for computing $\vz_{k+1}$ becomes
\begin{equation}\label{eq:AEP_newton_subproblem}
    0 \in
    a_k
    \left(
      \oF(\vz_{k+\frac{1}{2}})
      +
      D\oF(\vz_{k+\frac{1}{2}})
      (\vz_{k+1}-\vz_{k+\frac{1}{2}})
    \right)
    +
    a_k\oH(\vz_{k+1})
    +
    \vz_{k+1}-
    \frac{a_k}{A_k+a_k}\vz_0
    -
    \frac{A_k}{A_k+a_k}\vz_k,
\end{equation}
and the relative-error condition
in~\eqref{eq:surrogate_error_condition} reduces to
\begin{equation}\label{eq:AEP_newton_error_condition}
  a_k
  \left\|
    \oF(\vz_{k+1})
    -
    \oF(\vz_{k+\frac{1}{2}})
    -
    D\oF(\vz_{k+\frac{1}{2}})
    (\vz_{k+1}-\vz_{k+\frac{1}{2}})
  \right\|
  \leq
  \rho
  \|\vz_{k+1}-\vz_{k+\frac{1}{2}}\|.
\end{equation}
Note that the subproblem in \eqref{eq:AEP_newton_subproblem} has a unique solution. Indeed, since $\oF$ is differentiable
and monotone, the model $\oP_k$ is an affine monotone operator on $\reals^d$,
and $\oP_k+\oH$ is maximal monotone.
By Minty's theorem, $\oI+a_k(\oP_k+\oH)$ has a single-valued inverse defined everywhere, which guarantees
existence and uniqueness of a solution.

\begin{algorithm}[!t]\small
  \caption{Anchored Extra-Proximal Newton Method}
  \label{alg:aep_newton}
  \begin{algorithmic}[1]
    \STATE \textbf{Input:} initial point $\vz_0\in\dom\oH$, target
    accuracy $\varepsilon>0$, and Jacobian Lipschitz constant $L_2>0$
    \STATE \textbf{Initialize:} choose $\vv_0\in\oH(\vz_0)$ and set $A_0=0$
    \FOR{$k=0,1,\ldots$}
      \IF{$\|\oF(\vz_k)+\vv_k\|\leq\varepsilon$}
        \STATE Return $\vz_k$
      \ENDIF
      \STATE Apply the bisection line search in
      Subroutine~\ref{alg:second_order_ls} to select a step size $a_k$
      \STATE Compute $\vz_{k+\frac{1}{2}} = \frac{a_k}{A_{k}+a_k} \vz_0 + \frac{A_k}{A_k+a_k} (\vz_k - a_k (\oF(\vz_k) +\vv_k))$
      \STATE Find $\vz_{k+1}$ such that\vspace{-1em} $$\vspace{-1em}
      0 \in
    a_k
    \bigl(
      \oF(\vz_{k+\frac{1}{2}})
      +
      D\oF(\vz_{k+\frac{1}{2}})
      (\vz_{k+1}-\vz_{k+\frac{1}{2}})
    \bigr)
    +
    a_k\oH(\vz_{k+1})
    +
    \vz_{k+1}-
    \frac{a_k}{A_k+a_k}\vz_0
    -
    \frac{A_k}{A_k+a_k}\vz_k$$
    \STATE Update \vspace{-1em}
    \begin{equation*}\vspace{-1em}
               A_{k+1} = A_k + a_k, \quad \vv_{k+1} = \frac{1}{a_k} \bigl(  \frac{a_k}{A_{k}+a_k} \vz_0 + \frac{A_k}{A_k+a_k}\vz_k - \vz_{k+1} \bigr) - \bigl(
      \oF(\vz_{k+\frac{1}{2}})
      +
      D\oF(\vz_{k+\frac{1}{2}})
      (\vz_{k+1}-\vz_{k+\frac{1}{2}})
    \bigr).
    \end{equation*}
    \ENDFOR
  \end{algorithmic}
\end{algorithm}

The surrogate subproblem~\eqref{eq:AEP_newton_subproblem} is more structured than the exact anchored proximal subproblem~\eqref{eq:anchored_PPM_inclusion}, because the nonlinear operator $\oF$ has been replaced by an affine approximation. In particular, when $\oH\equiv0$, the subproblem reduces to a linear system, whose solution can be expressed as
\begin{equation*}
  \vz_{k+1}
  =
  \vz_{k+\frac{1}{2}}
  -
  a_k
  \bigl(
    \oI+a_kD\oF(\vz_{k+\frac{1}{2}})
  \bigr)^{-1}
  \left(
    \oF(\vz_{k+\frac{1}{2}})
    -
    \frac{A_k}{A_k+a_k}\oF(\vz_k)
  \right).
\end{equation*}
Newton-type proximal subproblems of the form~\eqref{eq:AEP_newton_subproblem} also appear frequently in second-order methods for monotone inclusions and variational inequalities; see, for example,~\cite{monteiro2012iteration,jiang2025generalized}.

We next explain the mechanism through which the second-order information improves the convergence rate. Combining~\eqref{eq:second_order_error} with~\eqref{eq:AEP_newton_error_condition}, we see that the relative-error condition is guaranteed whenever
\begin{equation}\label{eq:aep_newton_stepsize}
  a_k
  \leq
  \frac{2\rho}
  {L_2\|\vz_{k+1}-\vz_{k+\frac{1}{2}}\|}.
\end{equation}
Thus, as the displacement between the intermediate and next iterates decreases, the Taylor model permits increasingly large step sizes. Since Theorem~\ref{thm:aep} bounds the tangent residual in terms of $1/A_{k+1}$, larger step sizes allow the cumulative step size $A_{k+1}$ to grow more rapidly and thereby lead to faster convergence.

To make this intuition quantitative, fix $\rho\in(0,1)$ and suppose that each
step size can be chosen within a constant factor of the largest value certified
by~\eqref{eq:aep_newton_stepsize}; specifically, suppose that
\begin{equation}\label{eq:stepsize_lb}
  a_i \geq \frac{\sqrt{2}\rho}
  {L_2\|\vz_{i+1}-\vz_{i+\frac{1}{2}}\|}, \quad \forall i \geq 0.
\end{equation}
Together with~\eqref{eq:aep_newton_stepsize}, this places $a_i$ within a factor of
$\sqrt{2}$ of the upper bound. The line search in
Section~\ref{subsec:ls} will enforce precisely such a condition unless it has
already found an $\varepsilon$-accurate point. The key ingredient in translating
this step size guarantee into a convergence rate is the following stability
estimate, whose proof is deferred to Appendix~\ref{appen:stability}.

\begin{lemma}\label{lem:stability}
  Let $\{\vz_k\}_{k\geq0}$ be generated by AEP in Algorithm~\ref{alg:aep} with $\rho\in[0,1)$. Then, for every $k\geq0$,
  \begin{equation}\label{eq:aep_stability}
    \sum_{i=0}^k
    \frac{A_{i+1}^2}{a_i^2}
    \|\vz_{i+1}-\vz_{i+\frac{1}{2}}\|^2
    \leq
    \frac{1}{1-\rho^2}
    \|\vz_0-\vz^*\|^2.
  \end{equation}
\end{lemma}
Since \eqref{eq:stepsize_lb} implies
$\|\vz_{i+1}-\vz_{i+\frac{1}{2}}\|
\geq \sqrt{2}\rho/(L_2a_i)$, substituting this lower bound into
Lemma~\ref{lem:stability} yields
$\sum_{i=0}^k \frac{A_{i+1}^2}{a_i^4}
  \leq
  \frac{L_2^2}{2\rho^2(1-\rho^2)}
    \|\vz_0-\vz^*\|^2$.
Moreover, H\"older's inequality gives
$
  \left(\sum_{i=0}^k a_i\right)^{4/5}
  \left(\sum_{i=0}^k \frac{A_{i+1}^2}{a_i^4}\right)^{1/5}
  \geq
  \sum_{i=0}^k A_{i+1}^{2/5}$.
Since $A_{k+1}=\sum_{i=0}^k a_i$, combining the last two inequalities gives
\begin{equation*}
  A_{k+1}^{4/5}
  \geq
  \left(
    \frac{L_2^2\|\vz_0-\vz^*\|^2}
    {2\rho^2(1-\rho^2)}
  \right)^{-1/5}
  \sum_{i=0}^k A_{i+1}^{2/5}.
\end{equation*}
The sequence-growth lemma of~\cite[Lemma~12]{bubeck2019near} therefore implies
$A_{k+1}=\Omega((k+1)^{5/2})$. The residual bound in
Theorem~\ref{thm:aep} then yields the
$\bigO(k^{-5/2})$ rate established formally in
Theorem~\ref{thm:second_order_complexity}.

The prescription in~\eqref{eq:stepsize_lb}, however, is implicit and therefore cannot be used directly. The difficulty is that the step size $a_k$ and the iterates appearing on its right-hand side are mutually dependent.
In particular, $a_k$ determines the anchoring weights and the intermediate point $\vz_{k+\frac{1}{2}}$, which is also the point at which the Jacobian is evaluated. Meanwhile, $\vz_{k+1}$ is obtained by solving~\eqref{eq:AEP_newton_subproblem}, whose coefficients likewise depend on $a_k$.
Thus, the displacement needed to select $a_k$ in \eqref{eq:aep_newton_stepsize} is only known after the corresponding trial subproblem has been solved.
Similar step size coupling arises in higher-order HPE and tensor methods for minimization; see~\cite{monteiro2012iteration,monteiro2013accelerated,jiang2021optimal,bubeck2019near,gasnikov2019optimal}. To resolve it, we develop in Section~\ref{subsec:ls} a bisection-based line-search procedure that jointly determines the step size and the corresponding iterates while enforcing the relative-error condition~\eqref{eq:AEP_newton_error_condition}.

\subsection{Bisection Line Search}\label{subsec:ls}

We now develop a bisection line-search scheme for selecting the step size in the AEP Newton method. The procedure has two possible outcomes: it either finds an admissible step whose size is within a constant factor of the ideal scale in~\eqref{eq:aep_newton_stepsize}, or it directly returns a point with tangent residual at most $\varepsilon$. For concreteness, throughout this subsection, we set
$
  \rho=\frac{1}{\sqrt{2}}
$
in the AEP relative-error condition.

\begin{arxivchange}
Fix a target accuracy $\varepsilon>0$ and consider an outer iteration $k$.
For a trial step size $a>0$,
let $\vz_{k+\frac{1}{2}}(a)$ be defined by~\eqref{eq:AEP_k+1/2} with
$a_k=a$, and define the affine model
\begin{equation}\label{eq:trial_model}
  \oP_{k,a}(\vz)
  :=\oF(\vz_{k+\frac{1}{2}}(a))
  +D\oF(\vz_{k+\frac{1}{2}}(a))
  \bigl(\vz-\vz_{k+\frac{1}{2}}(a)\bigr).
\end{equation}
The trial point $\vz_{k+1}(a)$ solves
\begin{equation}\label{eq:trial_subproblem}
  0\in a\oP_{k,a}(\vz_{k+1}(a))+a\oH(\vz_{k+1}(a))
  +\vz_{k+1}(a)
  -\left(\frac{a}{A_k+a}\vz_0+\frac{A_k}{A_k+a}\vz_k\right).
\end{equation}
Its associated vector is
\begin{equation}\label{eq:trial_v}
  \vv_{k+1}(a)
  :=\frac{1}{a}\left(\frac{a}{A_k+a}\vz_0
  +\frac{A_k}{A_k+a}\vz_k-\vz_{k+1}(a)\right)
  -\oP_{k,a}(\vz_{k+1}(a))\in\oH(\vz_{k+1}(a)),
\end{equation}
where the inclusion follows from~\eqref{eq:trial_subproblem}.
We measure the quality of a trial step using the merit function
\begin{equation}\label{eq:ls_merit}
  \phi_k(a)
  :=
  L_2a
  \|\vz_{k+1}(a)-\vz_{k+\frac{1}{2}}(a)\|,
  \qquad
  \phi_k(0):=0.
\end{equation}
By the Taylor remainder bound~\eqref{eq:second_order_error} and the
definition of $\phi_k$ in~\eqref{eq:ls_merit},
\begin{equation}\label{eq:trial_error_merit}
  a\|\oF(\vz_{k+1}(a))-\oP_{k,a}(\vz_{k+1}(a))\|
  \leq\frac{\phi_k(a)}{2}
  \|\vz_{k+1}(a)-\vz_{k+\frac{1}{2}}(a)\|.
\end{equation}
Hence, every trial with $\phi_k(a)\leq\sqrt{2}$ satisfies the
relative-error condition in~\eqref{eq:surrogate_error_condition} with
$\rho=1/\sqrt{2}$. Moreover, the lower bound in~\eqref{eq:stepsize_lb} is
equivalent to $\phi_k(a)\geq1$, which ensures sufficiently large steps for acceleration. A bisection trial is therefore accepted
precisely when
\begin{equation}\label{eq:ls_acceptance_window}
  1 \leq \phi_k(a)\leq\sqrt{2}.
\end{equation}
\end{arxivchange}

The line search begins by evaluating the following computable upper trial value:
\begin{equation}\label{eq:ls_cap}
  \bar a_k
  :=
  \max\Biggl\{
    \frac{2A_k\|\oF(\vz_k)+\vv_k\|}{\varepsilon},
    \sqrt{\frac{2(1+\sqrt{2})}{L_2\varepsilon}}
  \Biggr\}.
\end{equation}
The following lemma shows that if $\phi_k(\bar a_k)\leq\sqrt{2}$, then its associated point already meets the target accuracy. We defer the proof to Appendix~\ref{appen:ls_residual}.

\begin{lemma}\label{lem:ls_residual}
  If $\phi_k(\bar a_k)\leq\sqrt{2}$, then
  $
    \res^{\mathrm{tan}}(\vz_{k+1}(\bar a_k))
    \leq
    \varepsilon$.
\end{lemma}

\begin{arxivchange}
Accordingly, if $\phi_k(\bar a_k)\leq\sqrt{2}$, we terminate the entire outer method and return the corresponding trial point. Otherwise, it holds that $\phi_k(0)=0<1<\sqrt{2}<\phi_k(\bar a_k)$,
and we therefore use $[0,\bar a_k]$ as the initial bisection interval. Starting from $a^-=0$ and $a^+=\bar a_k$, the bisection maintains the invariants
\[
  \phi_k(a^-)<1
  \qquad\text{and}\qquad
  \phi_k(a^+)>\sqrt{2}.
\]
At each step, it evaluates the midpoint
$a=\frac{a^-+a^+}{2}$.
If $\phi_k(a)<1$, the lower endpoint is replaced by~$a$; if
$\phi_k(a)>\sqrt{2}$, the upper endpoint is replaced by~$a$; otherwise, the trial is accepted. Thus, every unsuccessful bisection step preserves the two invariants and halves the length of the search interval.
The complete procedure is summarized in Subroutine~\ref{alg:second_order_ls}.
\end{arxivchange}

\begin{arxivchange}
\begin{subroutine}[!t]\arxivchangecolor\arxivchangecaption
  \caption{Bisection Line Search at Iteration $k$}
  \label{alg:second_order_ls}
  \begin{algorithmic}[1]
    \STATE \textbf{Input:}
    $(\vz_0,\vz_k,\vv_k,A_k,\varepsilon,L_2)$
    \STATE Compute $\bar a_k$ according to~\eqref{eq:ls_cap}
    \STATE Compute the trial
    $(\vz_{k+\frac{1}{2}}(\bar a_k),
      \vz_{k+1}(\bar a_k),
      \vv_{k+1}(\bar a_k))$
    \IF{$\phi_k(\bar a_k)\leq\sqrt{2}$}
      \STATE \textbf{Terminate} the outer method and return
      $(\vz_{k+1}(\bar a_k),\vv_{k+1}(\bar a_k))$
    \ENDIF
    \STATE Set $a^-=0$ and $a^+=\bar a_k$
    \WHILE{true}
      \STATE Set $a=(a^-+a^+)/2$
      \STATE Compute the trial
      $(\vz_{k+\frac{1}{2}}(a),
        \vz_{k+1}(a),
        \vv_{k+1}(a))$
      \IF{$\phi_k(a)<1$}
        \STATE Set $a^-=a$
      \ELSIF{$\phi_k(a)>\sqrt{2}$}
        \STATE Set $a^+=a$
      \ELSE
        \STATE \textbf{Accept} $a_k=a$ and return
        $\vz_{k+\frac{1}{2}}(a)$,
        $\vz_{k+1}(a)$, and $\vv_{k+1}(a)$
        as the next outer iterate
      \ENDIF
    \ENDWHILE
  \end{algorithmic}
\end{subroutine}
\end{arxivchange}

Note that this procedure is guaranteed to terminate if $\phi_k$ is continuous, which is established later in Lemma~\ref{lem:second_order_merit_regular}. Moreover, the number of trials before termination can be controlled by the Lipschitz modulus of $\phi_k$, as formalized in the following lemma.

\begin{lemma}[Bisection Complexity]
  \label{lem:ls_bisection_complexity}
  Let $L_k:=\Lip(\phi_k;[0,\bar a_k])$
be the Lipschitz constant of $\phi_k$ on $[0,\bar a_k]$. Then the bisection in
  Subroutine~\ref{alg:second_order_ls} returns an accepted trial after at most
  $1+\bigl\lceil\log_2\bigl(\bar a_kL_k/(\sqrt{2}-1)\bigr)\bigr\rceil$
  bisection trials.
\end{lemma}

In the next section, we will provide an upper bound on the Lipschitz constant
$\Lip(\phi_k;[0,\bar a_k])$ and thus a concrete upper bound on the line-search
complexity.

\subsection{Complexity Analysis}\label{subsec:second_order_complexity}

We now state the outer iteration and line-search complexities of
Algorithm~\ref{alg:aep_newton}. Fix a solution $\vz^*$. For clarity, an outer iteration refers to one accepted AEP Newton
update, whereas a trial evaluation computes
$\vz_{k+\frac{1}{2}}(a)$, $\vz_{k+1}(a)$, and $\vv_{k+1}(a)$
using the update formulas in Algorithm~\ref{alg:aep_newton}
with $a_k=a$.

Unless the upper trial already returns an $\varepsilon$-accurate point, every
accepted update satisfies the lower bound in~\eqref{eq:stepsize_lb} with $\rho=1/\sqrt{2}$.
Combining this bound with Lemma~\ref{lem:stability} shows that
$A_T=\Omega(T^{5/2}/(L_2\|\vz_0-\vz^*\|))$. The residual estimate in
Theorem~\ref{thm:aep} then gives
$\res^{\mathrm{tan}}(\vz_T)=\bigO(L_2\|\vz_0-\vz^*\|^2/T^{5/2})$,
which yields the following theorem. The full proof is given in
Appendix~\ref{appen:second_order_complexity}.
\begin{theorem}[Outer Iteration Complexity]
  \label{thm:second_order_complexity}
  Algorithm~\ref{alg:aep_newton} returns a point $\widehat\vz$ satisfying
  $\res^{\mathrm{tan}}(\widehat\vz)\leq\varepsilon$ after at most
  $T:=\bigl\lceil(16L_2\|\vz_0-\vz^*\|^2/\varepsilon)^{{2}/{5}}\bigr\rceil$
  outer iterations.
\end{theorem}

To turn the generic bisection bound in Lemma~\ref{lem:ls_bisection_complexity} into a
concrete per-iteration bound, it remains to control two quantities: the
Lipschitz modulus of $\phi_k$ and the initial interval length $\bar a_k$.
The following lemma provides both estimates; its proof is deferred to
Appendix~\ref{appen:second_order_merit_regular}.
\begin{lemma}[Lipschitzness of the Merit Function]
  \label{lem:second_order_merit_regular}
  At every outer iteration before termination,
  \begin{equation}\label{eq:merit_lipschitz_bound}
    \Lip(\phi_k;[0,\bar a_k])
    \leq
    29L_2\|\vz_0-\vz^*\|
    \bigl(1+L_2\|\vz_0-\vz^*\|\bar a_k\bigr)^2.
  \end{equation}
  Moreover, whenever the bisection phase is executed,
  $\bar a_k\leq\max\{4\|\vz_0-\vz^*\|/\varepsilon,
  \sqrt{2(1+\sqrt{2})/(L_2\varepsilon)}\}$.
\end{lemma}

Combining the estimates in Lemma~\ref{lem:second_order_merit_regular} with
Lemma~\ref{lem:ls_bisection_complexity} gives the following line-search
complexity.
\begin{theorem}[Line-Search Complexity]
  \label{thm:second_order_ls_complexity}
  At every outer iteration before termination,
  Subroutine~\ref{alg:second_order_ls} returns after
  $\bigO\bigl(1+\log(1+L_2\|\vz_0-\vz^*\|^2/\varepsilon)\bigr)$
  trial evaluations, including the initial trial at $a=\bar a_k$.
\end{theorem}

Each trial requires one second-order oracle call at the intermediate point and
one solution of the affine-model subproblem~\eqref{eq:trial_subproblem}.
Hence,
as a direct corollary of Theorems~\ref{thm:second_order_complexity}
and~\ref{thm:second_order_ls_complexity}, Algorithm~\ref{alg:aep_newton}
returns $\widehat\vz$ with $\res^{\mathrm{tan}}(\widehat\vz)\leq\varepsilon$
using
\begin{equation*}
  \widetilde{\bigO}\left(
    \left(\frac{L_2\|\vz_0-\vz^*\|^2}{\varepsilon}\right)^{2/5}
  \right)
\end{equation*}
second-order oracle calls.

In terms of the tangent residual, this
improves the classical pointwise $\bigO(\varepsilon^{-1})$ guarantee for
NPE-type methods for monotone inclusions~\cite{monteiro2010complexity,monteiro2012iteration}
and the recent $\widetilde{\bigO}(\varepsilon^{-1/2})$ bound for smooth
monotone variational inequalities obtained through an inexact Halpern
scheme~\cite{chen2026halpern}. As discussed in
Appendix~\ref{subsec:optimality}, the tangent residual also controls standard
gaps on bounded constraint sets. Through this relationship, our
result improves the dependence on $\varepsilon$ in the classical ergodic
$\bigO(\varepsilon^{-2/3})$ NPE guarantee and in the
$\widetilde{\bigO}(\varepsilon^{-4/7})$ second-order bounds for
convex-concave min-max problems~\cite{chen25solving,chen2026solving}, while
applying to the more general composite inclusion~\eqref{eq:monotone}.
Finally, our upper bound matches the deterministic second-order oracle lower
bound proved in Section~\ref{sec:lower_bound}, up to the logarithmic cost of
the line search.

\section{Anchored Extra-Proximal Tensor Method}
\label{sec:AEP_tensor}

We now extend the construction in Section~\ref{sec:AEP_Newton} to a
$p$th-order oracle returning
$\oF,D\oF,\ldots,D^{p-1}\oF$, where $p\geq2$. {Using a regularized Taylor approximation as the surrogate operator, we obtain the AEP Tensor method.}
We assume the following standard smoothness condition, analogous to Assumption~\ref{assum:Lips_Jacobian}.

\begin{assumption}\label{assum:Lips_higher_derivative}
  The $(p-1)$th derivative of $\oF$ is $L_p$-Lipschitz continuous, i.e., we have $\|D^{p-1}\oF(\vz)-D^{p-1}\oF(\vz')\|_{\op} \leq L_p\|\vz-\vz'\|$ for all $\vz,\vz'\in\reals^d$,
  where $\|\cdot\|_{\op}$ denotes the induced multilinear operator norm.
\end{assumption}

Define the $(p-1)$th-order Taylor approximation as $\oT^{(p-1)}(\vz';\vz)
  :=
  \oF(\vz)+\sum_{i=1}^{p-1}\frac{1}{i!}
  D^i\oF(\vz)[\vz'-\vz]^i$.
  To instantiate our AEP framework, we choose the surrogate operator in \eqref{eq:surrogate_inclusion} as a \emph{regularized} Taylor model
\begin{equation}\label{eq:surrogate-p}
  \oP_k(\vz) := \oT^{(p-1)}(\vz;\vz_{k+\frac{1}{2}}) + \frac{L_p}{(p-1)!}
    \|\vz-\vz_{k+\frac{1}{2}}\|^{p-1}(\vz-\vz_{k+\frac{1}{2}}).
\end{equation}
The reason for the additional regularization term in \eqref{eq:surrogate-p} is that, unlike the affine model in Section~\ref{sec:AEP_Newton}, a higher-order Taylor approximation need not be monotone.
Hence, we apply regularization to ensure that $\oP_k$ is maximal monotone, as shown in Lemma~\ref{lem:regularized_taylor}, so that the resulting subproblem is guaranteed to have a unique solution.

Using Assumption~\ref{assum:Lips_higher_derivative}, the approximation error in \eqref{eq:surrogate_error_condition} can be bounded by
$\left\|
    \oF(\vz_{k+1})-\oP_k(\vz_{k+1})
  \right\| \leq \frac{(p+1)L_p}{p!} \|\vz_{k+1} - \vz_{k+\frac{1}{2}}\|^p$.
Thus, the relative-error condition in \eqref{eq:surrogate_error_condition} permits step sizes scaling as
$\|\vz_{k+1}-\vz_{k+\frac{1}{2}}\|^{-(p-1)}$. As the displacement between $\vz_{k+1}$ and $\vz_{k+\frac{1}{2}}$ decreases, this scale grows faster than that in the AEP Newton method, leading to a faster convergence rate.
Since this displacement depends on the trial step size, we use a
bisection scheme to select sufficiently large steps satisfying
the AEP relative-error condition with $\rho=1/\sqrt{2}$, or return
an $\varepsilon$-accurate point.
The complete AEP Tensor method appears in Algorithm~\ref{alg:aep_tensor}, with the line-search construction and analysis in Appendix~\ref{appen:tensor_complexity}.

We now state the resulting complexity guarantees.  Here and below, constants
hidden by $\bigO_p$ may depend on the fixed order $p$, but not on
$L_p$, $\varepsilon$, or the initial distance.

{\begin{theorem}[Complexity of AEP Tensor Method]
  \label{thm:higher_order_complexity}
  Under Assumption~\ref{assum:Lips_higher_derivative},
  Algorithm~\ref{alg:aep_tensor} returns a point $\widehat\vz$ satisfying
  $\res^{\mathrm{tan}}(\widehat\vz)\leq\varepsilon$ after at most
  $\bigO_p\bigl(1+(L_p\|\vz_0-\vz^*\|^p/\varepsilon)^{\frac{2}{3p-1}}\bigr)$
  outer iterations, with at most
  $\bigO_p\bigl(1+\log(1+L_p\|\vz_0-\vz^*\|^p/\varepsilon)\bigr)$
  trial evaluations per iteration.
\end{theorem}}

Each trial uses one $p$th-order oracle call and solves one regularized
Taylor-model inclusion subproblem. {Theorem~\ref{thm:higher_order_complexity}
therefore gives the total oracle complexity}
\begin{equation}\label{eq:tensor_oracle_complexity}
  \widetilde{\bigO}_p\left(
    \left(
      \frac{L_p\|\vz_0-\vz^*\|^p}{\varepsilon}
    \right)^{2/(3p-1)}
  \right).
\end{equation}
The accuracy exponent $2/(3p-1)$ decreases with $p$ and, up to the logarithmic cost of
the line search, matches the lower bound in
Theorem~\ref{thm:higher_order_lower_bound}.  In particular, the matching
lower bound holds for every deterministic $p$th-order method, rather than
only for methods constrained to take a prescribed tensor step.

\section{Lower Bounds}\label{sec:lower_bound}

In this section, we show that the exponent in \eqref{eq:tensor_oracle_complexity} cannot be improved using a deterministic $p$th-order method for every $p\ge 2$, even in the unconstrained setting where $\oH=0$.
Our worst-case operator $\oF$ in Section~\ref{sec:lower-bound-construction} is inspired by the construction of \cite{jang2026higher}.
Based on the connection, we additionally present a refined lower bound for nonexpansive fixed-point problems in Section~\ref{sec:fixed-point-lower-bound-discussion}.

For $p\geq2$, a deterministic $p$th-order method for unconstrained monotone inclusion~\eqref{eq:monotone} with $\oH=0$ chooses its $t$th query $\vz_t$ as an arbitrary deterministic function of $\vz_0$ and the previous oracle outputs $\{D^j\oF(\vz_s):0\leq j\leq p-1,\ 0\leq s<t\}$, where $D^0\oF=\oF$.
In particular, a second-order method has access to operator values and Jacobians, consistently with Assumption~\ref{assum:Lips_Jacobian}.
The same concept applies to fixed-point problems~\eqref{eq:fixed-point}, with oracle outputs $D^j \oF$ replaced with $D^j \oT$.

\subsection{Lower Bound for Monotone Inclusions}
\label{sec:lower-bound-construction}

In this section, we prove the following main result:

\begin{theorem}

\label{thm:higher_order_lower_bound}
Fix an integer $p\geq2$.  There is a constant $C_p>0$, depending only
on $p$, such that the following holds.
Let $L_p>0$, $R>0$, and let $N,n$ be positive integers such that $n\geq2N+1$.
Fix a starting point $\vz_0\in\mathbb{R}^n$.
For every deterministic $p$th-order method initialized at $\vz_0$,
there exists a $(p-1)$-times continuously differentiable and monotone operator $\oF\colon \mathbb{R}^n\to\mathbb{R}^n$ with a unique zero $\vz^*$ satisfying $\norm{\vz_0-\vz^*} \le R$ such that $D^{p-1}\oF$ is $L_p$-Lipschitz, for which the output of the algorithm after $N$ oracle queries satisfies
\begin{equation}
\label{eq:higher_order_lower_bound}
    \res^{\mathrm{tan}}(\vz_N)
    =\|\oF(\vz_N)\|
    \geq C_p\frac{L_p R^p}{(N+1)^{(3p-1)/2}}.
\end{equation}

\end{theorem}

The construction of $\oF$ is inspired by the worst-case fixed-point operator from \cite[Theorem~3.1]{jang2026higher} for \textit{contractive} fixed-point problems with $q \in (0,1)$.
However, unlike their results, we use a \textit{nonexpansive} operator $\oG\colon \reals^n \to \reals^n$ of the form
\begin{equation}\label{eq:def_G}
    \oG(\vx)
    =c\vv_1-\varphi_a(\langle\vv_{N+1},\vx\rangle)\vv_1
    +\sum_{i=1}^{N}\varphi_a(\langle\vv_i,\vx\rangle)\vv_{i+1} ,
\end{equation}
and we take $\oF$ as a rescaled and translated version of $\oI - \oG$, adjusted to satisfy the assumptions of \cref{thm:higher_order_lower_bound}.
Here, $\vv_1, \dots, \vv_{N+1}$ are orthonormal vectors in $\reals^n$, chosen adaptively against the deterministic algorithm's outputs, and $\varphi_a \colon \reals \to \reals$ is a $C^\infty$ function such that $0 \le \varphi_a' \le 1$.
Note that with this construction, the Jacobian of $\oG$ is given by
\begin{equation*}
    D\oG(\vx) = -\varphi_a' (\inprod{\vv_{N+1}}{\vx}) \vv_1 \vv_{N+1}^\top + \sum_{i=1}^N \varphi_a' (\inprod{\vv_i}{\vx}) \vv_{i+1} \vv_i^\top
\end{equation*}
and, when restricted to $\mathrm{span}\{\vv_1, \dots, \vv_{N+1}\}$, it is represented by the matrix
\begin{align*}
    \begin{bmatrix}
        0 & 0 & \cdots & -\varphi_a' (\inprod{\vv_{N+1}}{\vx}) \\
        \varphi_a' (\inprod{\vv_1}{\vx}) & 0 & \cdots & 0\\
        \vdots & \ddots & \ddots & \vdots \\
        0 & \cdots & \varphi_a' (\inprod{\vv_N}{\vx}) & 0
    \end{bmatrix} .
\end{align*}
Notably, when $\varphi_a$ is the identity function $\varphi_a (s) = s$ with constant derivative $\varphi_a' \equiv 1$, this agrees with the matrix for the affine worst-case nonexpansive operator for first-order fixed-point algorithms presented in \cite[Theorem~4.1]{park2022exact}.
On the other hand, to handle algorithms that use higher-order derivative information, we require $\varphi_a$ to be identically zero near $0$, so that derivatives of $\oG$ of any order, computed at the algorithm's iterates, will not depend on directions $\vv_i$ that are not yet revealed.
A smoothed version of the soft-thresholding function, as characterized in the following lemma, will be suitable for this purpose.

\newsavebox{\PhiPropertiesBox}

\begin{lemma}
\label{lem:phi-a-b}
Let $a>0$. There exists a $C^\infty$ function
$\varphi_a \colon \mathbb{R}\to\mathbb{R}$ and
$\frac{a}{2} \le b \le a$ such that for any given
$r \in \mathbb{N}$, there exists a constant $B_r$
depending only on $r$ such that
\par
\begingroup
\def\PhiFigureAlignment{r}
\setbox\PhiPropertiesBox=\vbox{
    \hsize=0.62\linewidth
    \linewidth=\hsize
    \noindent
    \begin{enumerate}[(i)]
        \item $\varphi_a^{(k)}(0) = 0$ for $k=0,\dots,r$,
        \item $0 \le \varphi_a'(s) \le 1$ for all $s\in \mathbb{R}$,
        \item $\varphi_a^{(r)}$ is ${B_r}/{a^r}$-Lipschitz continuous,
        \item $\varphi_a(s) = \begin{cases}
            s-b & \text{if } s \ge a, \\
            s+b & \text{if } s \le -a,
        \end{cases}$
        \item $\varphi_a(s) \ge s-b$ for all $s \in \mathbb{R}$.
    \end{enumerate}
}
\pgfmathsetlengthmacro{\PhiUnit}{(\ht\PhiPropertiesBox+\dp\PhiPropertiesBox)/5.3}
\def\PhiHalfHeight{2.65}
\noindent
\usebox{\PhiPropertiesBox}\hfill
\makebox[0.35\linewidth][\PhiFigureAlignment]{\raisebox{-\dp\PhiPropertiesBox}{\begin{tikzpicture}[
    x=\PhiUnit,
    y=\PhiUnit,
    baseline=(current bounding box.south),
    font=\scriptsize,
    line cap=round,
    line join=round,
    every node/.style={text=black,inner sep=1.5pt}
]
    \path[use as bounding box]
    (-3.6,-\PhiHalfHeight) rectangle (3.6,\PhiHalfHeight);

    \begin{scope}[overlay]
        \draw[->,black!50]
            (-3.2,0) -- (3.25,0) node[right] {$s$};
        \draw[->,black!50]
            (0,-2.4) -- (0,2.35) node[above] {$\varphi_a(s)$};

        \draw[densely dashed,gray!65]
            (0,-0.75) -- (1,0.25);
        \draw[densely dashed,gray!65]
            (-1,-0.25) -- (0,0.75);

        \draw[gray!70]
            (-1,-0.055) -- (-1,0.055) node[above] {$-a$};
        \draw[gray!70]
            (1,0.055) -- (1,-0.055) node[below] {$a$};

        \node[below left,inner sep=2pt] at (0,0) {$0$};
        \node[right,inner sep=2pt] at (0,0.75) {$b$};
        \node[left,inner sep=2pt] at (0,-0.75) {$-b$};

        \foreach \side in {-1,1} {
            \draw[blue!65!black,thick,domain=0:1,samples=61]
                plot ({\side*\x},{\side*\x^10*(15*\x^2-36*\x+22)/4})
                -- ({3*\side},{2.25*\side});
        }

    \end{scope}
\end{tikzpicture}}}
\par
\endgroup
\end{lemma}

Given $\varphi_a$ with the above properties, the operator $\oG$ defined in \eqref{eq:def_G} has the following properties.

\begin{lemma}\label{lem:G_operator}
    Let $r, \varphi_a, a, b$ be as in Lemma~\ref{lem:phi-a-b}.
Let $n, N$ be positive integers such that $n \ge N+1$ and let $\vv_1, \dots, \vv_{N+1}$ be an orthonormal set of vectors in $\reals^n$.
Fix $M\ge a$, let $c = 2M + (N-1)b$ and define $\oG$ as in \eqref{eq:def_G}.
Then the following holds:
\begin{enumerate}[label=(\alph*)]
    \item $\oG$ is a nonexpansive operator, i.e., $\|\oG(\vx) - \oG(\vy)\| \leq \|\vx-\vy\|$ for any $\vx,\vy \in \reals^n$, and $D^{r} \oG$ is ${B_r}/{a^r}$-Lipschitz continuous.
    \item $\oG$ has a unique fixed point $\vx^* = \sum_{i=1}^{N+1} s_i \vv_i$, where $s_i = M + (N+1 - i) b$.
    \item For any $\vx\in\reals^n$ such that
    $\inprod{\vx}{\vv_{N+1}}=0$, we have $\norm{\vx-\oG(\vx)} \geq \frac{2M-b}{\sqrt{N+1}}$.
\end{enumerate}
\end{lemma}

Using this $\oG$, we can construct a worst-case monotone operator $\oF$ that will be the key ingredient for \cref{thm:higher_order_lower_bound}. Here we provide a short proof sketch, and the full proof is presented in Appendix~\ref{subsec:proof_monotone_lower_bound}.

\begin{proof}[Proof sketch for \cref{thm:higher_order_lower_bound}]
We define $\oF(\vz) = L_p \sigma^p (\oI - \oG) \left(\frac{\vz - \vz_0}{\sigma} \right)$, where the scaling factor $\sigma$ and the parameters of $\oG$ are chosen so that $\oF$ satisfies the assumptions prescribed in \cref{thm:higher_order_lower_bound}.
In particular, we take $r=p-1$, $a = B_r^{1/r}$ and $M = a(N+1)$.
Since $\oH=0$, the tangent residual is $\norm{\oF(\vz)}$,
which is a scaled fixed-point residual of $\oG$.

Given a deterministic $p$th-order method, we use a resisting-oracle
argument to choose the orthonormal directions
$\vv_1,\dots,\vv_{N+1}$ successively based on the method's
query points $\vz_0,\dots,\vz_{N-1}$ and final output $\vz_N$.
Specifically, we choose
$\vv_{i+1}
\perp \operatorname{span}
\{\vz_1-\vz_0,\dots,\vz_i-\vz_0,\vv_1,\dots,\vv_i\}$ for $i=0,\dots,N$.
The dimension condition $n\ge 2N+1$ guarantees that each such
direction exists.
Moreover, at previous query points $\vz_0, \dots, \vz_{i-1}$, the terms involving the newly chosen $\vv_{i+1}$ and their derivatives through order $p-1$
vanish by \cref{lem:phi-a-b}(i).
Thus, this choice leave all previous oracle responses and the method's history unchanged.

In particular, $\vx_N:=\frac{\vz_N-\vz_0}{\sigma}$ is orthogonal to
$\vv_{N+1}$, so \cref{lem:G_operator}(c) applies.
With our parameter choices, $\norm{\vx_N - \oG(\vx_N)} = \Omega_p((N+1)^{1/2})$, $\sigma=\Theta_p(R/(N+1)^{3/2})$ and $\norm{\vz_0 - \vz^*} \le R$.
Therefore,
$
    \res^{\mathrm{tan}}(\vz_N)
    = L_p\sigma^p \norm{\vx_N-\oG(\vx_N)}
    \ge C_p \frac{L_pR^p}{(N+1)^{(3p-1)/2}},
$
where $C_p>0$ depends only on $p$.
\end{proof}

Theorem~\ref{thm:higher_order_lower_bound} implies that achieving
$\res^{\mathrm{tan}}(\vz_N)\leq\varepsilon$ requires $N = \Omega_p\left(\left( \frac{L_p\|\vz_0-\vz^*\|^p}{\varepsilon} \right)^{2/(3p-1)} \right)$
oracle accesses.
This matches the outer complexity of the AEP Tensor method in
Theorem~\ref{thm:higher_order_complexity}, and the line-search bound in the same theorem leaves only a logarithmic gap between our upper and lower oracle bounds.

\subsection{A Refined Lower Bound for Nonexpansive Fixed-Point Problems and Discussion}
\label{sec:fixed-point-lower-bound-discussion}

\begin{arxivchange}
Focusing on the class of nonexpansive fixed-point problems, and by considering an averaged version of $\oG$ defined in \cref{lem:G_operator}, we can prove the following result. The proof is given below.
\end{arxivchange}

\begin{theorem}
\label{thm:fixed_point_lower_bound}
Under the same setup as Theorem~\ref{thm:higher_order_lower_bound},
for every deterministic $p$th-order method initialized at $\vy_0$, there exists a $(p-1)$-times continuously differentiable and nonexpansive operator $\oT\colon \mathbb{R}^n\to\mathbb{R}^n$ with a fixed point $\vy^*$ satisfying $\norm{\vy_0-\vy^*} \le R$ such that $D^{p-1}\oT$ is $L_p$-Lipschitz, for which the output of the algorithm after $N$ oracle queries satisfies
\begin{equation}
\label{eq:higher_order_lower_bound_fixed_point}
    \norm{\vy_N - \oT(\vy_N)}
    \geq \min\left\{ \frac{R}{2(N+1)} , C_p\frac{L_p R^p}{(N+1)^{(3p-1)/2}} \right\}.
\end{equation}
\end{theorem}

\begin{arxivchange}
\begin{proof}

We set $r = p-1$, $a = \frac{R}{2(N+1)^{3/2}}$ and $M = a(N+1)$.
For the operator $\oG$ from \cref{lem:G_operator}, we have $\norm{\vx^*} \le R$ by the same derivation as in \eqref{eqn:worst-case-operator-fixed-point-norm-bound}.
Now, with $\lambda = \min\left\{1, L_p a^{p-1} / B_{p-1}\right\}$, define:
\begin{equation}\label{eq:worst_case_T}
    \oT(\vy) = \vy_0 + (1-\lambda) (\vy - \vy_0) + \lambda \oG(\vy - \vy_0).
\end{equation}
Because $0 < \lambda \le 1$ and $\oG$ is nonexpansive, $\oT$ is also nonexpansive.
Furthermore, because $D^{p-1}\oG$ is $\frac{B_{p-1}}{a^{p-1}}$-Lipschitz by \cref{lem:G_operator}(a), we have
\begin{equation*}
    \norm{D^{p-1} \oT(\vy) - D^{p-1} \oT(\vy')}_\mathrm{op}
    \le \lambda\frac{B_{p-1}}{a^{p-1}} \norm{(\vy - \vy_0) - (\vy' - \vy_0)} \le L_p \norm{\vy - \vy'} .
\end{equation*}
The unique fixed point of $\oT$ is clearly $\vy^* = \vx^* + \vy_0$ by \cref{lem:G_operator}(b), and $\norm{\vy_0 - \vy^*} = \norm{\vx^*} \le R$.
Finally, let $\vx_t = \vy_t - \vy_0$ for $t=0,\dots,N$ and select $\vv_i$ as in \eqref{eqn:vi-perpendicular-to-history}.
Then, again using \cref{lem:G_operator}(c) and $b \le a \le M$, we obtain $\norm{\vx_N - \oG(\vx_N)} \ge \frac{2M-b}{\sqrt{N+1}} \ge \frac{M}{\sqrt{N+1}} = \frac{R}{2(N+1)}$, which then implies
\begin{align*}
    \norm{\vy_N - \oT(\vy_N)} & = \lambda \norm{\vx_N - \oG(\vx_N)} \ge \frac{\lambda R}{2(N+1)} \\
    & \ge \min\left\{ \frac{R}{2(N+1)} , \frac{L_p a^{p-1} R}{2B_{p-1} (N+1)} \right\} \\
    & = \min \left\{ \frac{R}{2(N+1)}, \frac{L_p R^{p}}{2^p B_{p-1} (N+1)^{{(3p-1)}/{2}}} \right\} .
\end{align*}
\end{proof}
\end{arxivchange}

\paragraph{Interpretation of \cref{thm:fixed_point_lower_bound}.}
In the regime $N = o(R^{2/3})$ where $\frac{R}{2(N+1)} \ll C_p \frac{L_p R^p}{(N+1)^{(3p-1)/2}}$, we have $\norm{\vy_N - \oT(\vy_N)} \ge \frac{R}{2(N+1)}$.
The lower bound matches the order of the convergence rate achievable by a first-order algorithm \citep{liederConvergenceRateHalperniteration2021}.
Therefore, when the number of iterations is not large enough compared to the initial distance to the fixed point, second- or higher-order derivative information does not necessarily accelerate the convergence of nonexpansive fixed-point algorithms.
This aligns with the analogous results of \cite{jang2026higher} for contractive fixed-point algorithms,
where they selected $R$ as a function of $N$ to show the impossibility of acceleration beyond what is achievable by first-order algorithms.
On the other hand, our lower bound treats $N$ and $R$ as independent parameters, and does not rule out the acceleration via higher-order oracles in the regime $N = \omega(R^{2/3})$ where the lower bound becomes $\norm{\vy_N - \oT(\vy_N)} \ge C_p \frac{L_p R^p}{(N+1)^{(3p-1)/2}}$.
Indeed, this lower bound is matched up to a logarithmic factor by Theorem~\ref{thm:higher_order_complexity}, because given a nonexpansive operator $\oT$ with Lipschitz continuous $D^{p-1}\oT$, we can apply \cref{alg:aep_tensor} with $\oF = \oI - \oT$ and $\oH = 0$.

\section{Conclusion}
In this work, we extend anchor acceleration for first-order monotone inclusions to second- and higher-order tensor methods.
Our proposed AEP framework takes a principled approach based on the anchored proximal point method, an extension of the optimal Halpern iteration using a resolvent, and approximates each proximal step using the higher-order derivatives of the operator.
As the order $p$ of the available derivative information grows, increasingly aggressive progress can be made at each iteration, leading to faster convergence rates.
In particular, AEP attains a complexity of $\tilde{\bigO}\bigl(\varepsilon^{-2/(3p-1)}\bigr)$ for finding a point $\vz$ satisfying $\res^{\mathrm{tan}}(\vz)\leq\varepsilon$.
We further show that this rate is optimal up to a logarithmic factor by establishing a matching lower bound.
We also establish a refined lower bound for smooth nonexpansive fixed-point problems, matched up to logarithmic factors by choosing between AEP and Halpern iteration.

Our results suggest several natural directions for future work.
It would be interesting to investigate whether, paralleling the development of the literature on first-order acceleration for monotone inclusions, our AEP framework/method can be extended to weakly nonmonotone inclusions, reformulations using Nesterov-type momentum, stochastic settings, or alternative acceleration mechanisms.
Another promising direction is to develop quasi-Newton variants of AEP that retain the benefits of higher-order acceleration while avoiding the explicit use of higher-order derivatives.
Broadly, our results suggest that anchoring is not merely a first-order acceleration mechanism, but a general principle that systematically translates richer oracle information into faster methods for monotone inclusion problems.

\newpage
\section*{Statement of AI Usage and History}

We use generative AI (ChatGPT 5.6~Sol, ChatGPT~6~Astra and Codex) for testing hypotheses, discovering and writing parts of the proofs, and drafting sections of the paper.
Specifically, we have not used AI to develop the main conceptual framework of AEP in Section~\ref{sec:aep} and its iteration complexity analysis; we did use AI for developing the initial versions of the bisection line search in Section~\ref{subsec:ls} and the lower bounds in Section~\ref{sec:lower_bound}.
The main theoretical results, including the second-order complexity bound in \cref{thm:second_order_ls_complexity} and the lower bound in Section~\ref{sec:lower-bound-construction}, were obtained with the aid of AI tools on or before August~31, and were initially drafted no later than September~9, earlier than the appearance of the concurrent work \citep{zhang2026matching} on September~{14} and \citep{zhang2026optimal} on September~20.
Since then, the authors have reviewed all AI-assisted proofs and writing, checked their correctness, and refined them, but have not made substantive changes to the essence of the results, except for the addition of the complexity bound for higher orders $p>2$ in \cref{sec:AEP_tensor} and the lower bound for nonexpansive fixed-point problems in \cref{sec:fixed-point-lower-bound-discussion}.
We take full responsibility for the final contents of this work.

\newpage
\appendix

\makeatletter
\addtocontents{atoc}{}

\let\orig@addcontentsline\addcontentsline
\renewcommand{\addcontentsline}[3]{\edef\@tempa{#1}\edef\@tempb{toc}\ifx\@tempa\@tempb
\orig@addcontentsline{atoc}{#2}{#3}\else
\orig@addcontentsline{#1}{#2}{#3}\fi
}
\makeatother

\clearpage
\section*{Contents of Appendix}
\setcounter{tocdepth}{3}
\makeatletter
\@starttoc{atoc}
\makeatother
\clearpage

\section{Optimality Measures}\label{subsec:optimality}
In this section, we review several other optimality measures for Problem~\eqref{eq:monotone} and its special cases.

\mypara{Gap Functions} For the min-max problem~\eqref{eq:min-max}, recall that $\ell(\vx,\vy):=f(\vx,\vy)+h_1(\vx)-h_2(\vy)$. When $\dom h_1$ and $\dom h_2$ are compact, the duality gap~\cite{Nemirovski2004} at $(\hat{\vx},\hat{\vy}) \in \dom h_1 \times \dom h_2$ is
\begin{equation*}
  \mathrm{gap}(\hat{\vx},\hat{\vy})
  := \max_{\vy \in \dom h_2} \ell(\hat{\vx},\vy)
  - \min_{\vx \in \dom h_1} \ell(\vx,\hat{\vy}).
\end{equation*}
{This gap is always nonnegative and vanishes if and only if $(\hat{\vx},\hat{\vy})$ satisfies the saddle-point conditions.}

For the composite variational inequality~\eqref{eq:VI} with compact $\dom h$, two related gap functions are commonly used~\cite{Facchinei2003}. The strong gap, also called the primal or Stampacchia gap, is
\begin{equation*}
  \mathrm{gap}_{\mathrm{s}}(\hat{\vz})
  := \sup_{\vz \in \dom h}
  \bigl\{\langle \oF(\hat{\vz}),\hat{\vz}-\vz\rangle
  +h(\hat{\vz})-h(\vz)\bigr\},
\end{equation*}
whereas the weak gap, also called the dual or Minty gap, is
\begin{equation*}
  \mathrm{gap}_{\mathrm{w}}(\hat{\vz})
  := \sup_{\vz \in \dom h}
  \bigl\{\langle \oF(\vz),\hat{\vz}-\vz\rangle
  +h(\hat{\vz})-h(\vz)\bigr\}.
\end{equation*}
If $\oF$ is monotone, then $\mathrm{gap}_{\mathrm{w}}(\hat{\vz}) \leq \mathrm{gap}_{\mathrm{s}}(\hat{\vz})$ for every $\hat{\vz} \in \dom h$. Moreover, $\mathrm{gap}_{\mathrm{s}}(\hat{\vz})=0$ if and only if $\hat{\vz}$ solves~\eqref{eq:VI}; when $\oF$ is also continuous, the same equivalence holds for $\mathrm{gap}_{\mathrm{w}}$.

\mypara{$\varepsilon$-Enlargement-Based Residual} For a maximal monotone operator {$\oH$} and $\varepsilon\geq 0$, its $\varepsilon$-enlargement~\cite{burachik1997enlargement} is
\begin{equation*}
  {\oH}^\varepsilon(\vz)
  := \left\{
  \vv :
  \langle \vv-\vv',\vz-\vz'\rangle \geq -\varepsilon
  \quad \forall (\vz',\vv') \in \gra {\oH}
  \right\}.
\end{equation*}
Given $\rho,\varepsilon\geq 0$, a point $\vz$ is a $(\rho,\varepsilon)$-weak solution of~\eqref{eq:monotone} if there exists $\vv \in (\oF+\oH)^\varepsilon(\vz)$ such that $\|\vv\|\leq\rho$. It is a $(\rho,\varepsilon)$-strong solution if there exists $\vv \in \oF(\vz)+\oH^\varepsilon(\vz)$ such that $\|\vv\|\leq\rho$. When $\oH=N_{\mathcal{X}}$ is the normal-cone operator of a closed convex set $\mathcal{X}$, these notions recover the weak and strong approximate solutions of Monteiro--Svaiter~\cite{monteiro2010complexity}.

\mypara{Resolvent-Based Residuals} For any $\alpha>0$, the \emph{forward-backward residual}~\cite{diakonikolas2020halpern,yoon2025accelerated} is
\begin{equation*}
  \res^{\mathrm{FB}}_{\alpha}(\vz)
  := \frac{1}{\alpha}
  \bigl\|\vz-\oJ_{\alpha\oH}\bigl(\vz-\alpha\oF(\vz)\bigr)\bigr\|,
\end{equation*}
where $\oJ_{\alpha\oH}:=(\oI+\alpha\oH)^{-1}$ denotes the resolvent of $\alpha\oH$. When $\oH=N_{\mathcal X}$, this quantity is commonly
called the \emph{scaled natural residual}~\cite{Facchinei2003}. Similarly, the \emph{proximal residual}~\cite{chen2026halpern} is defined as
\begin{equation*}
  \res^{\mathrm{prox}}_{\alpha}(\vz)
  := \frac{1}{\alpha}
  \bigl\|\vz-\oJ_{\alpha(\oF+\oH)}(\vz)\bigr\|.
\end{equation*}

Following similar observations in~\citet{cai2023accelerated}, the next lemma formalizes how a bound on the tangent residual controls these alternative optimality measures. For completeness, we include its proof below.
\begin{lemma}\label{lem:approx_solutions}
  The following statements are true:
    \begin{enumerate}[(i)]
        \item For the convex-concave min-max problem~\eqref{eq:min-max}, let $\hat{\vz}:=(\hat{\vx},\hat{\vy})$ and define
        \begin{equation*}
          D
          := \max_{(\vx,\vy) \in \dom h_1 \times \dom h_2}
          \|\hat{\vz}-(\vx,\vy)\|.
        \end{equation*}
        Then $\mathrm{gap}(\hat{\vx},\hat{\vy}) \leq D\,\res^{\mathrm{tan}}(\hat{\vz})$. Likewise, for the composite variational inequality~\eqref{eq:VI} and any $\hat{\vz}\in\dom h$, define
        \begin{equation*}
          D_{\mathrm{VI}}
          := \max_{\vz \in \dom h}\|\hat{\vz}-\vz\|.
        \end{equation*}
        Then $\mathrm{gap}_{\mathrm{w}}(\hat{\vz}) \leq \mathrm{gap}_{\mathrm{s}}(\hat{\vz}) \leq D_{\mathrm{VI}}\,\res^{\mathrm{tan}}(\hat{\vz})$.
        \item If $\res^{\mathrm{tan}}(\vz) \leq \rho$, then $\vz$ is a $(\rho,0)$-weak solution and $(\rho,0)$-strong solution in the sense of \cite{monteiro2010complexity}.
        \item For any $\alpha > 0$, it holds that $\res^{\mathrm{FB}}_{\alpha}(\vz) \leq  \res^{\mathrm{tan}}(\vz)$ and $\res^{\mathrm{prox}}_{\alpha}(\vz) \leq \res^{\mathrm{tan}}(\vz)$.
    \end{enumerate}

\end{lemma}
\begin{proof}

\noindent\emph{(i)} For the min-max problem, fix any $\vv=(\vv_1,\vv_2)\in\oH(\hat{\vz})=\partial h_1(\hat{\vx})\times\partial h_2(\hat{\vy})$ and $\vz=(\vx,\vy)\in\dom h_1\times\dom h_2$.
By convexity of $f(\cdot,\hat{\vy})$ and the subgradient inequalities for $h_1$, we have
\begin{equation*}
  f(\hat{\vx},\hat{\vy}) - f(\vx,\hat{\vy}) \leq \langle \nabla_{\vx}f(\hat{\vx},\hat{\vy}), \hat{\vx} - \vx \rangle, \quad h_1(\hat{\vx}) - h_1(\vx) \leq \langle \vv_1, \hat{\vx} -   \vx\rangle.
\end{equation*}
Using the same argument, we also have
\begin{equation*}
  f(\hat{\vx},{\vy}) - f(\hat{\vx},\hat{\vy}) \leq \langle -\nabla_{\vy}f(\hat{\vx},\hat{\vy}), \hat{\vy} - {\vy} \rangle, \quad h_2(\hat{\vy}) - h_2(\vy) \leq \langle \vv_2, \hat{\vy} -   \vy\rangle.
\end{equation*}
By summing all the inequalities together, we obtain that
\begin{equation*}
  \ell(\hat{\vx},\vy)-\ell(\vx,\hat{\vy})
  \leq \langle \oF(\hat{\vz})+\vv,\hat{\vz}-\vz\rangle
  \leq D\,\|\oF(\hat{\vz})+\vv\|.
\end{equation*}
Taking the supremum over $(\vx,\vy)$ and then the infimum over $\vv\in\oH(\hat{\vz})$ proves the duality-gap bound.
For the composite variational inequality, fix any $\vv\in\partial h(\hat{\vz})$.
The subgradient inequality gives, for every $\vz\in\dom h$,
\begin{equation*}
  \langle\oF(\hat{\vz}),\hat{\vz}-\vz\rangle+h(\hat{\vz})-h(\vz)
  \leq \langle\oF(\hat{\vz})+\vv,\hat{\vz}-\vz\rangle
  \leq D_{\mathrm{VI}}\,\|\oF(\hat{\vz})+\vv\|.
\end{equation*}
Taking the supremum over $\vz\in\dom h$ and then the infimum over $\vv\in\partial h(\hat{\vz})$ yields the strong-gap bound.
Monotonicity of $\oF$ gives $\langle\oF(\vz),\hat{\vz}-\vz\rangle\leq\langle\oF(\hat{\vz}),\hat{\vz}-\vz\rangle$, and hence $\mathrm{gap}_{\mathrm{w}}(\hat{\vz})\leq\mathrm{gap}_{\mathrm{s}}(\hat{\vz})$.

\noindent\emph{(ii)} Since $\oH$ is maximal monotone, its nonempty values are closed and convex.
Thus, if $\res^{\mathrm{tan}}(\vz)\leq\rho$, the infimum defining the tangent residual is attained: there exists $\vv\in\oH(\vz)$ such that $\vw:=\oF(\vz)+\vv$ satisfies $\|\vw\|=\res^{\mathrm{tan}}(\vz)\leq\rho$.
Monotonicity implies $\oH(\vz)\subseteq\oH^0(\vz)$ and $(\oF+\oH)(\vz)\subseteq(\oF+\oH)^0(\vz)$.
Consequently, $\vw\in\oF(\vz)+\oH^0(\vz)$ and $\vw\in(\oF+\oH)^0(\vz)$, proving both claims.

\noindent\emph{(iii)} Fix $\alpha>0$ and any $\vv\in\oH(\vz)$.
The resolvent identities give
\begin{equation*}
  \vz=\oJ_{\alpha\oH}(\vz+\alpha\vv)
  =\oJ_{\alpha(\oF+\oH)}(\vz+\alpha(\oF(\vz)+\vv)).
\end{equation*}
Under our standing assumptions, both $\oH$ and $\oF+\oH$ are maximal monotone, so their resolvents are nonexpansive. Therefore,
\begin{equation*}
  \res^{\mathrm{FB}}_{\alpha}(\vz)
  =\frac{1}{\alpha}\bigl\|\oJ_{\alpha\oH}(\vz+\alpha\vv)-\oJ_{\alpha\oH}(\vz-\alpha\oF(\vz))\bigr\|
  \leq\|\oF(\vz)+\vv\|,
\end{equation*}
and similarly,
\begin{equation*}
  \res^{\mathrm{prox}}_{\alpha}(\vz)
  =\frac{1}{\alpha}\bigl\|\oJ_{\alpha(\oF+\oH)}(\vz+\alpha(\oF(\vz)+\vv))-\oJ_{\alpha(\oF+\oH)}(\vz)\bigr\|
  \leq\|\oF(\vz)+\vv\|.
\end{equation*}
Taking the infimum over $\vv\in\oH(\vz)$ proves both residual bounds.
\end{proof}

\section{Additional Related Work}

\mypara{First-Order Acceleration: Anchoring and Beyond}
While the idea of retracting iterates to the initial point dates back to the work of \citet{halpernFixedPointsNonexpanding1967}, it recently regained popularity as an acceleration mechanism for fixed-point problems \citep{SabachShtern2017_first,liederConvergenceRateHalperniteration2021,contrerasOptimalErrorBounds2023}.
Halpern's mechanism, also called anchoring, was then applied to min-max optimization and monotone inclusion \citep{RyuYuanYin2019_ode,diakonikolas2020halpern}, and notably, \citet{diakonikolas2020halpern} first achieved near-optimal $\tilde\bigO (\epsilon^{-1})$ complexity with respect to the proximal residual.
\citet{yoon2021accelerated} then combined anchoring with extragradient to establish the optimal $\bigO(\epsilon^{-1})$ complexity in the unconstrained setting, which was subsequently extended to weakly nonmonotone operators \citep{lee2021fast,cai2024accelerated}, single-call algorithms \citep{Tran-DinhLuo2021_halperntype,cai2023accelerated}, continuous-time analysis \citep{SuhParkRyu2023_continuoustime} or algorithms using generalized anchor points \citep{alcalaMovingAnchorExtragradient2023,botExtragradientMethodFlexible2026}.
This acceleration phenomenon was further developed and interpreted as a vanishing dampening mechanism similar to Nesterov momentum \citep{tran-dinhHalpernsFixedpointIterations2024,botFastOptimisticGradient2025,sedlmayerFastOptimisticMethod2023,tran-dinhExtragradienttypeMethods2024}.
There also exists a family of distinct acceleration mechanisms for monotone inclusion and fixed-point problems that operate under a predetermined iteration budget \citep{yoonOptimalAccelerationMinimax2024,yoonHinvarianceTheoryComplete2026,yoon2026theory}.
In this work, we primarily focus on the anchor acceleration mechanism and its higher-order extensions.

\mypara{Second-Order and Higher-Order Methods}
To our knowledge, the first global iteration-complexity result for a second-order method for monotone variational inequalities was obtained by~\citet{nesterov2006cubicVI}. They proposed a dual Newton method based on cubic regularization~\cite{nesterov2006cubic} and established a complexity of $\bigO(\varepsilon^{-1})$ in terms of the restricted weak gap.
The first second-order complexity bounds improving over first-order methods were later obtained by~\citet{monteiro2010complexity}, who analyzed a variant of the Newton HPE method of~\citet{solodov1999hybrid} for smooth monotone equations. Using the $\epsilon$-enlargement-based residual, their method finds an $(\varepsilon,\varepsilon)$-weak solution at an averaged iterate in $\bigO(\varepsilon^{-2/3})$ iterations and an $(\varepsilon,0)$-strong solution at the best iterate in $\bigO(\varepsilon^{-1})$ iterations. These results were subsequently extended to composite monotone inclusions of the form~\eqref{eq:monotone} in~\cite{monteiro2012iteration}. Motivated in part by the nontrivial line search required by NPE, subsequent works developed line-search-free second-order methods with the same $\bigO(\varepsilon^{-2/3})$ ergodic iteration complexity, both for monotone variational inequalities~\cite{alves2024search} and for convex-concave min-max problems~\cite{jiang2024adaptive,lin2026explicit}.
For monotone variational inequalities, higher-order variants of Mirror Prox~\cite{bullins2022higher,adil2022optimal} and the higher-order dual-extrapolation method \cite{lin2025perseus} find an averaged $\varepsilon$-weak solution in $\bigO(\varepsilon^{-2/(p+1)})$ iterations. The latter additionally obtains a complexity of $\bigO(\varepsilon^{-2/p})$ for finding an $\varepsilon$-strong solution at the best iterate under the more general Minty condition. For convex-concave min-max problems, a $p$th-order generalized optimistic method is proposed in \cite{jiang2024generalized}, with the same $\bigO(\varepsilon^{-2/(p+1)})$ iteration complexity for the restricted primal--dual gap.

In the setting of smooth convex minimization, \citet{monteiro2013accelerated} introduced the accelerated NPE framework and obtained a second-order complexity of $\tilde{\bigO}(\varepsilon^{-2/7})$. Subsequent works extended this framework to higher-order methods, attaining a complexity of $\tilde{\bigO}(\varepsilon^{-2/(3p+1)})$~\cite{bubeck2019near,jiang2021optimal,gasnikov2019optimal}. The additional logarithmic factor was later removed by~\cite{carmon2022optimal,kovalev2022first}, and matching lower bounds were established by~\cite{arjevani2019oracle}, showing that the resulting dependence on $\varepsilon$ is optimal.

\subsection{Comparison with Concurrent Work}\label{appen:concurrent_work}

The concurrent works~\cite{zhang2026matching,zhang2026optimal} and our method attain the same optimal accuracy exponent $2/(3p-1)$ through different algorithmic constructions. \citet{zhang2026matching} use a multi-hierarchy acceleration scheme built around inexact extrapolated Halpern iteration. Building on~\cite{chen2026halpern}, \citet{zhang2026optimal} combine an inexact accelerated Halpern iteration as the outer loop with the anchored tensor method as the inner solver. Our AEP framework instead uses a relative-error condition for inexact anchored proximal updates in composite monotone inclusions. It recovers FEG as a first-order, unconstrained special case and yields natural higher-order extensions by directly substituting Taylor-model approximations. Relative to~\cite{zhang2026optimal}, our method has two computational differences: it selects step sizes adaptively through line search rather than following a prescribed outer-step schedule, and, for unconstrained second-order problems, each line-search trial reduces to a linear-system solve rather than a cubic-regularized Newton-type subproblem.

The upper bounds also differ in their logarithmic factors. The gap between the upper and lower bounds in~\cite{zhang2026matching} is $\bigO\!\left(\log^{6(p-1)}\left(L_p\norm{\vz_0-\vz^*}^p/\varepsilon\right)\right)$. Our upper bound reduces this gap to a single logarithmic factor, $\bigO\!\left(\log\left(L_p\norm{\vz_0-\vz^*}^p/\varepsilon\right)\right)$, arising from the bisection line search. It does not attain the logarithm-free complexity of~\cite{zhang2026optimal}.

On the lower-bound side, we adapt the fixed-point construction of~\citet{jang2026higher} and develop a resisting-oracle argument to establish a matching lower bound for every deterministic algorithm in the $p$th-order oracle model, including on unconstrained problems. \citet{zhang2026matching} use a qualitatively similar but independently derived, distinct operator. Our results additionally cover smooth nonexpansive fixed-point problems, for which we establish a refined lower bound and attain it up to logarithmic factors by choosing between AEP and Halpern iteration.

\section{Proofs for the Anchored Extra-Proximal Framework}

\subsection{Proof of Proposition~\ref{prop:anchored_PPM}}\label{appen:anchored_PPM}

For every $k\geq0$, define
\begin{equation}\label{eq:def_v_+1}
  \vv_{k+1} := \frac{1}{a_k} \left(\frac{a_k}{A_{k}+a_k} \vz_0 + \frac{A_k}{A_k+a_k}\vz_k - \vz_{k+1} - a_k \oF(\vz_{k+1})  \right) \in \oH(\vz_{k+1}),
\end{equation}
where the inclusion follows from~\eqref{eq:anchored_PPM_inclusion}.
We prove the stronger bound
\begin{equation*}
  \res^{\mathrm{tan}}(\vz_{k+1}) \leq \left\|\oF(\vz_{k+1}) + \vv_{k+1}\right\| \leq \frac{2\|\vz_0-\vz^*\|}{A_{k+1}}.
\end{equation*}

We begin by showing that the anchored PPM admits a nonincreasing Lyapunov
function in Lemma~\ref{lem:Lyapunov_PPM}. This descent property will then yield the tangent-residual bound in
Proposition~\ref{prop:anchored_PPM}.

\begin{lemma}\label{lem:Lyapunov_PPM}
Let $\{\vz_k\}$ be generated by the anchored PPM in~\eqref{eq:anchored_PPM},
and let $\vv_{k+1}\in\oH(\vz_{k+1})$ be defined by~\eqref{eq:def_v_+1}.
Set $V_0:=0$ and, for every $k\geq1$, define
\begin{equation}\label{eq:Lyapunov_PPM}
  V_k
  :=A_k\langle \oF(\vz_k)+\vv_k,\vz_k-\vz_0\rangle
  +\frac{A_k^2}{2}\|\oF(\vz_k)+\vv_k\|^2.
\end{equation}
Then $V_{k+1}\leq V_k$ for every $k\geq0$.
\end{lemma}
\begin{proof}
For brevity, write $g_j:=\oF(\vz_j)+\vv_j$ for $j\geq1$.
When $k=0$, we have $A_0 = 0$ and the update in~\eqref{eq:def_v_+1} gives
$\vz_1-\vz_0=-a_0g_1$. Hence,
\begin{equation*}
  V_1
  =a_0\langle g_1,\vz_1-\vz_0\rangle
   +\frac{a_0^2}{2}\|g_1\|^2
  =-\frac{a_0^2}{2}\|g_1\|^2
  \leq0=V_0.
\end{equation*}

Now fix $k\geq1$. Since $\oF$ and $\oH$ are monotone, while
$\vv_k\in\oH(\vz_k)$ and $\vv_{k+1}\in\oH(\vz_{k+1})$, we have
\begin{equation}\label{eq:monotone_z_z_plus}
  \langle g_{k+1}-g_k,\vz_{k+1}-\vz_k\rangle\geq0.
\end{equation}
Moreover, from~\eqref{eq:def_v_+1} we can derive
\begin{align*}
  \vz_{k+1}-\vz_k
  &=\frac{a_k}{A_k+a_k}(\vz_0-\vz_k)-a_kg_{k+1}, \\
  \frac{A_k}{A_k+a_k}(\vz_{k+1}-\vz_k)
  &=\frac{a_k}{A_k+a_k}(\vz_0-\vz_{k+1})-a_kg_{k+1}.
\end{align*}
Therefore,
\begin{equation*}
  \langle g_{k+1},\vz_{k+1}-\vz_k\rangle
  =\frac{a_k}{A_k}\langle g_{k+1},\vz_0-\vz_{k+1}\rangle
   -\frac{a_k(A_k+a_k)}{A_k}\|g_{k+1}\|^2.
\end{equation*}
Similarly,
\begin{equation*}
  \langle g_k,\vz_{k+1}-\vz_k\rangle
  =\frac{a_k}{A_k+a_k}\langle g_k,\vz_0-\vz_k\rangle
   -a_k\langle g_{k+1},g_k\rangle.
\end{equation*}
Substituting these two identities into~\eqref{eq:monotone_z_z_plus} gives
\begin{equation*}
    (A_k+a_k)\langle g_{k+1},\vz_{k+1}-\vz_0\rangle
      +(A_k+a_k)^2\|g_{k+1}\|^2 \leq
      A_k\langle g_k,\vz_k-\vz_0\rangle
      +A_k(A_k+a_k)\langle g_{k+1},g_k\rangle.
\end{equation*}
By the Cauchy--Schwarz and Young inequalities, we have
\begin{equation*}
  A_k(A_k+a_k)\langle g_{k+1},g_k\rangle
  \leq\frac{A_k^2}{2}\|g_k\|^2
  +\frac{(A_k+a_k)^2}{2}\|g_{k+1}\|^2.
\end{equation*}
Since $A_{k+1}=A_k+a_k$, it follows that
\begin{equation*}
  A_{k+1}\langle g_{k+1},\vz_{k+1}-\vz_0\rangle
  +\frac{A_{k+1}^2}{2}\|g_{k+1}\|^2
  \leq
  A_k\langle g_k,\vz_k-\vz_0\rangle
  +\frac{A_k^2}{2}\|g_k\|^2.
\end{equation*}
By~\eqref{eq:Lyapunov_PPM}, this is precisely $V_{k+1}\leq V_k$.
\end{proof}

\begin{proof}[Proof of Proposition~\ref{prop:anchored_PPM}]
Fix $k\geq0$ and set
$g_{k+1}:=\oF(\vz_{k+1})+\vv_{k+1}$.
By Lemma~\ref{lem:Lyapunov_PPM},
$V_{k+1}\leq V_0=0$. Since $A_{k+1}>0$, dividing the definition
of $V_{k+1}$ by $A_{k+1}$ and rearranging gives
\begin{equation}\label{eq:ppm_residual_inner_product}
  \frac{A_{k+1}}{2}\|g_{k+1}\|^2
  \leq \langle g_{k+1},\vz_0-\vz_{k+1}\rangle.
\end{equation}

Because $\vz^*$ solves~\eqref{eq:monotone}, there exists
$\vv^*\in\oH(\vz^*)$ such that
$\oF(\vz^*)+\vv^*=0$. The monotonicity of $\oF$ and $\oH$ therefore
implies
\begin{equation*}
  \langle g_{k+1},\vz_{k+1}-\vz^*\rangle
  =\langle
    \oF(\vz_{k+1})+\vv_{k+1}-\oF(\vz^*)-\vv^*,
    \vz_{k+1}-\vz^*
  \rangle
  \geq0.
\end{equation*}
Consequently,
\begin{equation*}
  \langle g_{k+1},\vz_0-\vz_{k+1}\rangle
  \leq
  \langle g_{k+1},\vz_0-\vz_{k+1}\rangle
  +\langle g_{k+1},\vz_{k+1}-\vz^*\rangle =\langle g_{k+1},\vz_0-\vz^*\rangle.
\end{equation*}
Combining this inequality with~\eqref{eq:ppm_residual_inner_product}
and applying the Cauchy--Schwarz inequality yields
\begin{equation*}
  \frac{A_{k+1}}{2}\|g_{k+1}\|^2
  \leq \|g_{k+1}\|\,\|\vz_0-\vz^*\|.
\end{equation*}
If $g_{k+1}=0$, the desired bound is immediate. Otherwise, dividing by
$\|g_{k+1}\|$ gives
\begin{equation*}
  \|\oF(\vz_{k+1})+\vv_{k+1}\|
  =\|g_{k+1}\|
  \leq \frac{2\|\vz_0-\vz^*\|}{A_{k+1}}.
\end{equation*}
Finally, since $\vv_{k+1}\in\oH(\vz_{k+1})$, the definition
in~\eqref{eq:tangent_residual} gives
\begin{equation*}
  \res^{\mathrm{tan}}(\vz_{k+1})
  \leq\|\oF(\vz_{k+1})+\vv_{k+1}\|,
\end{equation*}
which completes the proof.
\end{proof}

\subsection{\arxivchanged{Proofs of Lemma~\ref{lem:Lyapunov_AEP} and Theorem~\ref{thm:aep}}}\label{appen:aep}

\begin{arxivchange}
We first prove the potential descent estimate in Lemma~\ref{lem:Lyapunov_AEP}.
\end{arxivchange}

\begin{proof}
Write $g_j:=\oF(\vz_j)+\vv_j$ for brevity. When $k=0$,
\eqref{eq:AEP_k+1/2} gives $\vz_{\frac{1}{2}}=\vz_0$, while
\eqref{eq:v_plus_AEP} gives
\begin{equation*}
  \vz_1-\vz_0=-a_0g_1+\vr_1.
\end{equation*}
Consequently, using $A_1=a_0$ and
$a_0g_1=\vr_1-(\vz_1-\vz_{\frac{1}{2}})$, we obtain
\begin{equation*}
  V_1=a_0\langle g_1,\vz_1-\vz_0\rangle
    +\frac{a_0^2}{2}\|g_1\|^2 =\frac{1}{2}\|\vr_1\|^2
    -\frac{1}{2}\|\vz_1-\vz_{\frac{1}{2}}\|^2 \leq
    -\frac{1-\rho^2}{2}\|\vz_1-\vz_{\frac{1}{2}}\|^2,
\end{equation*}
where the last inequality follows from the relative-error condition in~\eqref{eq:AEP_error_condition}.
Thus,~\eqref{eq:potential_diff} holds for $k=0$.

Now fix $k\geq1$. From~\eqref{eq:v_plus_AEP}, we can write
\begin{align}
  \vz_{k+1}-\vz_k
  &=\frac{a_k}{A_{k+1}}(\vz_0-\vz_k)
    -a_kg_{k+1}+\vr_{k+1},
  \label{eq:aep_iterate_difference}
  \\
  \frac{A_k}{A_{k+1}}(\vz_{k+1}-\vz_k)
  &=\frac{a_k}{A_{k+1}}(\vz_0-\vz_{k+1})
    -a_kg_{k+1}+\vr_{k+1}.
  \label{eq:aep_weighted_iterate_difference}
\end{align}
Taking the inner product of~\eqref{eq:aep_weighted_iterate_difference}
with $g_{k+1}$ yields
\begin{equation}\label{eq:aep_gkp1_inner_product}
  \langle g_{k+1},\vz_{k+1}-\vz_k\rangle
  =\frac{a_k}{A_k}\langle g_{k+1},\vz_0-\vz_{k+1}\rangle
   -\frac{a_kA_{k+1}}{A_k}
    \left\langle
      g_{k+1},g_{k+1}-\frac{\vr_{k+1}}{a_k}
    \right\rangle.
\end{equation}
Similarly, taking the inner product of~\eqref{eq:aep_iterate_difference}
with $g_k$ gives
\begin{equation}\label{eq:aep_gk_inner_product}
  \langle g_k,\vz_{k+1}-\vz_k\rangle
  =\frac{a_k}{A_{k+1}}\langle g_k,\vz_0-\vz_k\rangle
   -a_k\left\langle
     g_k,g_{k+1}-\frac{\vr_{k+1}}{a_k}
   \right\rangle.
\end{equation}
On the other hand, the monotonicity of $\oF$ and $\oH$, together with
$\vv_j\in\oH(\vz_j)$, implies
\begin{equation*}
  \langle g_{k+1}-g_k,\vz_{k+1}-\vz_k\rangle\geq0.
\end{equation*}
Substituting~\eqref{eq:aep_gkp1_inner_product} and
\eqref{eq:aep_gk_inner_product} into this inequality and multiplying by
$A_kA_{k+1}/a_k$ gives
\begin{equation}\label{eq:aep_pre_square_completion}
  A_{k+1}\langle g_{k+1},\vz_{k+1}-\vz_0\rangle
   +A_{k+1}^2
    \left\langle g_{k+1}-\frac{\vr_{k+1}}{a_k},g_{k+1}\right\rangle \leq
   A_k\langle g_k,\vz_k-\vz_0\rangle
   +A_kA_{k+1}
    \left\langle g_{k+1}-\frac{\vr_{k+1}}{a_k},g_k\right\rangle.
\end{equation}
Next, we complete the square in the last term on the left to obtain
\begin{equation}\label{eq:aep_left_square_completion}
  A_{k+1}^2
   \left\langle g_{k+1}-\frac{\vr_{k+1}}{a_k},g_{k+1}\right\rangle =\frac{A_{k+1}^2}{2}
     \left\|g_{k+1}-\frac{\vr_{k+1}}{a_k}\right\|^2
     +\frac{A_{k+1}^2}{2}\|g_{k+1}\|^2 -\frac{A_{k+1}^2}{2a_k^2}\|\vr_{k+1}\|^2.
\end{equation}
Similarly, the last term on the right can be written as
\begin{equation}\label{eq:aep_right_square_completion}
A_kA_{k+1}
   \left\langle g_{k+1}-\frac{\vr_{k+1}}{a_k},g_k\right\rangle =\frac{A_{k+1}^2}{2}
     \left\|g_{k+1}-\frac{\vr_{k+1}}{a_k}\right\|^2
     +\frac{A_k^2}{2}\|g_k\|^2 -\frac{1}{2}
     \left\|A_{k+1}\left(g_{k+1}-\frac{\vr_{k+1}}{a_k}\right)
       -A_kg_k\right\|^2.
\end{equation}
Furthermore, from~\eqref{eq:AEP_k+1/2} and~\eqref{eq:v_plus_AEP}, respectively, we have
\begin{align*}
  A_{k+1}\vz_{k+\frac{1}{2}}
  &=a_k\vz_0+A_k\vz_k-a_kA_kg_k, \\
  A_{k+1}\vz_{k+1}
  &=a_k\vz_0+A_k\vz_k
    -a_kA_{k+1}\left(g_{k+1}-\frac{\vr_{k+1}}{a_k}\right).
\end{align*}
Subtracting these identities yields
\begin{equation}\label{eq:aep_square_norm_identity}
  \left\|A_{k+1}\left(g_{k+1}-\frac{\vr_{k+1}}{a_k}\right)
    -A_kg_k\right\|
  =\frac{A_{k+1}}{a_k}
   \|\vz_{k+1}-\vz_{k+\frac{1}{2}}\|.
\end{equation}
Substituting
\eqref{eq:aep_left_square_completion}--\eqref{eq:aep_square_norm_identity}
into~\eqref{eq:aep_pre_square_completion} and canceling the common term
$\frac{A_{k+1}^2}{2}
\left\|g_{k+1}-\frac{\vr_{k+1}}{a_k}\right\|^2$ gives
\begin{equation*}
  V_{k+1}
  \leq V_k
  +\frac{A_{k+1}^2}{2a_k^2}
   \left(
     \|\vr_{k+1}\|^2
     -\|\vz_{k+1}-\vz_{k+\frac{1}{2}}\|^2
   \right).
\end{equation*}
Finally, the relative-error condition~\eqref{eq:AEP_error_condition} implies
\begin{equation*}
  \|\vr_{k+1}\|^2
  -\|\vz_{k+1}-\vz_{k+\frac{1}{2}}\|^2
  \leq
  -(1-\rho^2)\|\vz_{k+1}-\vz_{k+\frac{1}{2}}\|^2,
\end{equation*}
which proves~\eqref{eq:potential_diff}.
\end{proof}

\begin{proof}[Proof of Theorem~\ref{thm:aep}]
Because $\rho\in[0,1]$, Lemma~\ref{lem:Lyapunov_AEP} implies
$V_{k+1}\leq V_0=0$ for every $k\geq0$. Since
$\vv_{k+1}\in\oH(\vz_{k+1})$, the remainder follows from the same argument
as in the proof of Proposition~\ref{prop:anchored_PPM}.
\end{proof}

\subsection{Proof of Lemma~\ref{lem:stability}}\label{appen:stability}

Rearranging the inequality in~\eqref{eq:potential_diff} gives, for every
$i\geq0$,
\begin{equation*}
  \frac{1-\rho^2}{2}
  \frac{A_{i+1}^2}{a_i^2}
  \|\vz_{i+1}-\vz_{i+\frac{1}{2}}\|^2
  \leq V_i-V_{i+1}.
\end{equation*}
Summing from $i=0$ to $k$ and using $V_0=0$ yields
\begin{equation}\label{eq:stability_telescoping}
  \frac{1-\rho^2}{2}
  \sum_{i=0}^k
  \frac{A_{i+1}^2}{a_i^2}
  \|\vz_{i+1}-\vz_{i+\frac{1}{2}}\|^2
  \leq V_0-V_{k+1}=-V_{k+1}.
\end{equation}

It remains to bound $-V_{k+1}$. Set
$g_{k+1}:=\oF(\vz_{k+1})+\vv_{k+1}$. Since $\vz^*$
solves~\eqref{eq:monotone}, there exists $\vv^*\in\oH(\vz^*)$
such that $\oF(\vz^*)+\vv^*=0$. By the monotonicity of $\oF$
and $\oH$,
\begin{equation*}
  \langle g_{k+1},\vz_{k+1}-\vz^*\rangle
  =\langle
    \oF(\vz_{k+1})+\vv_{k+1}-\oF(\vz^*)-\vv^*,
    \vz_{k+1}-\vz^*
  \rangle
  \geq0.
\end{equation*}
Consequently,
\begin{equation}\label{eq:stability_potential_bound}
\begin{aligned}
  -V_{k+1}
  &=A_{k+1}\langle g_{k+1},\vz_0-\vz_{k+1}\rangle
    -\frac{A_{k+1}^2}{2}\|g_{k+1}\|^2 \\
  &\leq A_{k+1}\langle g_{k+1},\vz_0-\vz^*\rangle
    -\frac{A_{k+1}^2}{2}\|g_{k+1}\|^2 \\
  &\leq A_{k+1}\|g_{k+1}\|\,\|\vz_0-\vz^*\|
    -\frac{A_{k+1}^2}{2}\|g_{k+1}\|^2 \\
  &\leq\frac{1}{2}\|\vz_0-\vz^*\|^2,
\end{aligned}
\end{equation}
where the second inequality follows from the Cauchy--Schwarz inequality and
the last follows from Young's inequality. Combining
\eqref{eq:stability_telescoping} and~\eqref{eq:stability_potential_bound}
and using $\rho<1$ proves~\eqref{eq:aep_stability}.

\section{Proofs for the Line-Search Scheme}

\subsection{Proof of Lemma~\ref{lem:ls_residual}}\label{appen:ls_residual}
Fix a trial step size $a>0$. By~\eqref{eq:trial_v}, we have
\begin{equation*}
  \oP_{k,a}(\vz_{k+1}(a))+
  \vv_{k+1}(a)
  =
  \frac{1}{a}
  \left(
    \frac{a}{A_k+a}\vz_0
    +\frac{A_k}{A_k+a}\vz_k
    -\vz_{k+1}(a)
  \right).
\end{equation*}
On the other hand, evaluating~\eqref{eq:AEP_k+1/2} at the trial step
size $a$ gives
\begin{equation*}
  \frac{a}{A_k+a}\vz_0
  +\frac{A_k}{A_k+a}\vz_k
  =
  \vz_{k+\frac{1}{2}}(a)
  +\frac{aA_k}{A_k+a}\bigl(\oF(\vz_k)+\vv_k\bigr).
\end{equation*}
Substituting this identity into the preceding identity leads to
\begin{equation}\label{eq:trial_identity}
  \oP_{k,a}(\vz_{k+1}(a))+
  \vv_{k+1}(a)
  =
  \frac{A_k}{A_k+a}
  \bigl(\oF(\vz_k)+\vv_k\bigr)
  -\frac{\vz_{k+1}(a)-\vz_{k+\frac{1}{2}}(a)}{a}.
\end{equation}
By applying the triangle inequality and using
the Taylor remainder bound on $\|\oF(\vz_{k+1}(a))-\oP_{k,a}(\vz_{k+1}(a))\|$, we further have
\begin{equation*}
  \|\oF(\vz_{k+1}(a))+\vv_{k+1}(a)\|
  \leq
  \frac{A_k}{A_k+a}\|\oF(\vz_k)+\vv_k\|
  +\frac{\|\vz_{k+1}(a)-\vz_{k+\frac{1}{2}}(a)\|}{a}
  +\frac{L_2}{2}
  \|\vz_{k+1}(a)-\vz_{k+\frac{1}{2}}(a)\|^2.
\end{equation*}
Suppose that
$\phi_k(a)=L_2a
\|\vz_{k+1}(a)-\vz_{k+\frac{1}{2}}(a)\|\leq\sqrt{2}$. Then
$\|\vz_{k+1}(a)-\vz_{k+\frac{1}{2}}(a)\|
  \leq\frac{\sqrt{2}}{L_2a}$.
Using this inequality and $A_k/(A_k+a)\leq A_k/a$ in the preceding bound,
we obtain
\begin{equation}\label{eq:cap_residual_bound}
  \|\oF(\vz_{k+1}(a))+\vv_{k+1}(a)\|
  \leq
  \frac{A_k\|\oF(\vz_k)+\vv_k\|}{a}
  +\frac{1+\sqrt{2}}{L_2a^2}.
\end{equation}

We now set $a=\bar a_k$. The definition~\eqref{eq:ls_cap} gives the two
bounds
\begin{align*}
  \bar a_k
  &\geq
  \frac{2A_k\|\oF(\vz_k)+\vv_k\|}{\varepsilon},
  &
  \bar a_k^2
  &\geq
  \frac{2(1+\sqrt{2})}{L_2\varepsilon}.
\end{align*}
Consequently, each term on the right-hand side of
\eqref{eq:cap_residual_bound} is at most $\varepsilon/2$. Under the
hypothesis $\phi_k(\bar a_k)\leq\sqrt{2}$, that bound therefore gives
\[
  \|\oF(\vz_{k+1}(\bar a_k))
    +\vv_{k+1}(\bar a_k)\|
  \leq
  \varepsilon.
\]
Finally,
$\vv_{k+1}(\bar a_k)\in\oH(\vz_{k+1}(\bar a_k))$, and hence
$\res^{\mathrm{tan}}(\vz_{k+1}(\bar a_k))
  \leq \|\oF(\vz_{k+1}(\bar a_k))
    +\vv_{k+1}(\bar a_k)\| \leq
  \varepsilon$, as claimed.

\subsection{Proof of Lemma~\ref{lem:ls_bisection_complexity}}

Suppose that the bisection has performed $t$ unsuccessful steps. According to the invariants maintained by the bisection, its
endpoints still satisfy
\[
  \phi_k(a^-)<1,
  \qquad
  \phi_k(a^+)>\sqrt{2},
\]
and the current interval has length
$a^+-a^-=\bar a_k/2^t$. Therefore,
\[
  \sqrt{2}-1
  <
  \phi_k(a^+)-\phi_k(a^-)
  \leq
  \Lip(\phi_k;[0,\bar a_k])(a^+-a^-)
  =
  \frac{\bar a_k\Lip(\phi_k;[0,\bar a_k])}{2^t}.
\]
Solving the above inequality gives
\begin{equation*}
  t < \log_2
      \left(
        \frac{\bar a_k \Lip(\phi_k;[0,\bar a_k])}{\sqrt{2}-1}
      \right).
\end{equation*}
The bisection consequently cannot perform as many unsuccessful steps as the
ceiling in Lemma~\ref{lem:ls_bisection_complexity}; its next trial must lie in the
acceptance range. This completes the proof.

\section{Proofs of the Second-Order Complexity Bounds}
\label{appen:second_order_complexity}

We collect here the proofs and technical lemmas used in Section~\ref{subsec:second_order_complexity}.

\subsection{Outer Iteration Complexity}

\begin{proof}[Proof of Theorem~\ref{thm:second_order_complexity}]
  If an upper trial step size terminates the method during the
  first $T$ outer iterations, the conclusion follows from
  Lemma~\ref{lem:ls_residual}. Otherwise, the line search returns $T$ accepted
  steps. By~\eqref{eq:trial_error_merit} and the upper bound in
  \eqref{eq:ls_acceptance_window}, the relative-error condition in \eqref{eq:surrogate_error_condition} is satisfied and hence these $T$ accepted steps form an instance
  of AEP with $\rho=1/\sqrt{2}$. Moreover, the lower bound in
  \eqref{eq:ls_acceptance_window} gives, for every such step,
  \[
    L_2a_i
    \|\vz_{i+1}-\vz_{i+\frac{1}{2}}\|
    \geq1.
  \]
  Applying Lemma~\ref{lem:stability} with $\rho=1/\sqrt{2}$ therefore yields
  \begin{equation}\label{eq:appendix_second_order_sum}
    \sum_{i=0}^{k-1}
    \frac{A_{i+1}^2}{a_i^4}
    \leq
    L_2^2
    \sum_{i=0}^{k-1}
    \frac{A_{i+1}^2}{a_i^2}
    \|\vz_{i+1}-\vz_{i+\frac{1}{2}}\|^2
    \leq
    2L_2^2\|\vz_0-\vz^*\|^2,
    \qquad k=1,\ldots,T.
  \end{equation}
  As in Section~\ref{sec:AEP_Newton}, H\"older's inequality gives, for every
  $k=1,\ldots,T$,
  \[
    \sum_{i=0}^{k-1}A_{i+1}^{2/5}
    =
    \sum_{i=0}^{k-1}
    a_i^{4/5}
    \left(\frac{A_{i+1}^2}{a_i^4}\right)^{1/5}
    \leq
    A_k^{4/5}
    \left(
      \sum_{i=0}^{k-1}\frac{A_{i+1}^2}{a_i^4}
    \right)^{1/5}.
  \]
  Combining this inequality with~\eqref{eq:appendix_second_order_sum} yields
  \begin{equation}\label{eq:appendix_second_order_growth_recursion}
    A_k^{4/5}
    \geq
    \left(2L_2^2\|\vz_0-\vz^*\|^2\right)^{-1/5}
    \sum_{j=1}^k A_j^{2/5}.
  \end{equation}
  Set $B_k:=A_k^{2/5}$. The sequence $\{B_k\}_{k\geq1}$ is positive and
  nondecreasing, and~\eqref{eq:appendix_second_order_growth_recursion} becomes
  \[
    B_k^2
    \geq
    c\sum_{j=1}^k B_j,
    \qquad
    c:=\left(2L_2^2\|\vz_0-\vz^*\|^2\right)^{-1/5}.
  \]

  For completeness, we state the sequence-growth result used to solve this
  recursion from \cite{bubeck2019near}.
  \begin{lemma}[{\cite[Lemma~12]{bubeck2019near}}]
    \label{lem:growth}
    Let $\{B_k\}_{k\geq1}$ be a positive, nondecreasing sequence. If there
    exist constants $\alpha>1$ and $c>0$ such that $B_k^{\alpha}\geq c\sum_{j=1}^k B_j$ for every $k\geq1$, then
    \[
      B_k\geq
      \left(\frac{\alpha-1}{\alpha}ck\right)^{1/(\alpha-1)}.
    \]
  \end{lemma}

  Applying Lemma~\ref{lem:growth} to the preceding recursion with
  $\alpha=2$ gives $B_t\geq ct/2$. Consequently,
  \begin{equation}\label{eq:appendix_second_order_A_growth}
    A_t
    =B_t^{5/2}
    \geq
    \left(\frac{ct}{2}\right)^{5/2}
    =
    \frac{t^{5/2}}{8L_2\|\vz_0-\vz^*\|}.
  \end{equation}
  For
  $T=\left\lceil
    \left(16L_2\|\vz_0-\vz^*\|^2/\varepsilon\right)^{2/5}
  \right\rceil$, the last inequality implies
  $A_T\geq2\|\vz_0-\vz^*\|/\varepsilon$.
  Theorem~\ref{thm:aep} then gives
  \[
    \|\oF(\vz_{T})+\vv_{T}\|
    \leq
    \frac{2\|\vz_0-\vz^*\|}{A_T}
    \leq
    \varepsilon.
  \]
  Thus, the stopping test in Algorithm~\ref{alg:aep_newton} returns
  $\vz_T$, unless the method has already terminated earlier.
\end{proof}

\subsection{Line-Search Complexity}\label{appen:second_order_merit_regular}

We first derive the line-search bound from
Lemma~\ref{lem:second_order_merit_regular} and the generic bisection estimate
in Lemma~\ref{lem:ls_bisection_complexity}.

\begin{proof}[Proof of Theorem~\ref{thm:second_order_ls_complexity}]

  If the initial trial at $a=\bar a_k$ passes the termination test, the
  subroutine returns after one trial evaluation. We may therefore suppose
  that the bisection phase is executed. By
  Lemma~\ref{lem:ls_bisection_complexity}, it suffices to bound
  $\bar a_k\Lip(\phi_k;[0,\bar a_k])$. The upper bound in
  \eqref{eq:merit_lipschitz_bound} gives
  \begin{equation}\label{eq:appendix_merit_interval_product}
    \bar a_k\Lip(\phi_k;[0,\bar a_k])
    \leq
    29L_2\|\vz_0-\vz^*\|\bar a_k
    \bigl(1+L_2\|\vz_0-\vz^*\|\bar a_k\bigr)^2.
  \end{equation}
  Moreover, the cap bound in Lemma~\ref{lem:second_order_merit_regular} implies that
  \[
    L_2\|\vz_0-\vz^*\|\bar a_k
    \leq
    \max\left\{
      \frac{4L_2\|\vz_0-\vz^*\|^2}{\varepsilon},
      \sqrt{
        2(1+\sqrt{2})
        \frac{L_2\|\vz_0-\vz^*\|^2}{\varepsilon}
      }
    \right\}.
  \]
  Setting
  \[
    q:=\frac{L_2\|\vz_0-\vz^*\|^2}{\varepsilon},
  \]
  we obtain
  \[
    L_2\|\vz_0-\vz^*\|\bar a_k
    \leq
    \max\left\{4q,\sqrt{2(1+\sqrt{2})q}\right\}
    \leq4(1+q),
  \]
  where the last inequality uses $\sqrt q\leq(1+q)/2$ and
  $\sqrt{2(1+\sqrt{2})}<4$. Since
  $1+4(1+q)\leq5(1+q)$,~\eqref{eq:appendix_merit_interval_product}
  further implies
  \begin{equation}\label{eq:appendix_merit_product_q}
    \bar a_k\Lip(\phi_k;[0,\bar a_k])
    \leq
    29\cdot4(1+q)\bigl(1+4(1+q)\bigr)^2
    \leq
    2900(1+q)^3.
  \end{equation}

  Applying Lemma~\ref{lem:ls_bisection_complexity} and adding the initial
  evaluation at $a=\bar a_k$, the total number of trial evaluations is at
  most
  \[
    2+
    \left\lceil
      \log_2\left(
        \frac{2900(1+q)^3}{\sqrt{2}-1}
      \right)
    \right\rceil.
  \]
  This bound is
  \[
    \bigO\bigl(1+\log(1+q)\bigr)
    =
    \bigO\left(
      1+\log\left(
        1+\frac{L_2\|\vz_0-\vz^*\|^2}{\varepsilon}
      \right)
    \right).
  \]
  This proves Theorem~\ref{thm:second_order_ls_complexity}.

\end{proof}

We next prove Lemma~\ref{lem:second_order_merit_regular}. The key argument is to first reduce the Lipschitz modulus of the merit function to uniform and Lipschitz
bounds for the mapping $a \mapsto \vz_{k+1}(a)-\vz_{k+\frac{1}{2}}(a)$. Specifically, we define the trial displacement
\begin{equation}\label{eq:trial_displacement_map}
  \vd_k(a):=\vz_{k+1}(a)-\vz_{k+\frac{1}{2}}(a),
  \qquad a>0,
  \qquad \vd_k(0):=0.
\end{equation}
Moreover, we similarly define the Lipschitz modulus of $\vd_k$ over the interval $[0, \bar{a}_k]$ as
\begin{equation}\label{eq:merit_lipschitz_constant}
  \Lip(\vd_k;[0,\bar a_k]) := \sup_{0 \leq {a<b} \leq \bar a_k}
  \frac{\|\vd_k(a)-\vd_k(b)\|}{|a-b|}.
\end{equation}
The following elementary bound makes this reduction precise.
\begin{lemma}
  \label{lem:ls_sensitivity}
  Suppose that $\vd_k$ is Lipschitz continuous on $[0,\bar a_k]$. Then
  $\phi_k$ is Lipschitz continuous on the same interval, with
  \begin{equation}\label{eq:ls_sensitivity_reduction}
    \Lip(\phi_k;[0,\bar a_k])
    \leq L_2\left(
      \sup_{0\leq a\leq\bar a_k}\|\vd_k(a)\|
      +\bar a_k\Lip(\vd_k;[0,\bar a_k])
    \right).
  \end{equation}
\end{lemma}
\begin{proof}
  For distinct $a,b\in[0,\bar a_k]$, the definition of $\phi_k$ and the
  reverse triangle inequality give
  \begin{align*}
    |\phi_k(a)-\phi_k(b)|
    &\leq
    L_2|a-b|\|\vd_k(a)\|
    +L_2b\bigl|\|\vd_k(a)\|-\|\vd_k(b)\|\bigr| \\
    &\leq
    L_2|a-b|\|\vd_k(a)\|
    +L_2b\|\vd_k(a)-\vd_k(b)\|.
  \end{align*}
  Dividing by $|a-b|$, using $b\leq\bar a_k$, and taking the supremum
  over $a$ and $b$ proves~\eqref{eq:ls_sensitivity_reduction}.
\end{proof}

It therefore remains to control the size and variation of $\vd_k$. The next
two lemmas provide the required estimates and their proofs are deferred to
Appendix~\ref{appen:second_order_ls_technical}.

\begin{lemma}[Trial Displacement]
  \label{lem:trial_displacement}
  For $0<a\leq\bar a_0$,
  \[
    \|\vd_0(a)\|
    \leq
    2\|\vz_0-\vz^*\|
    +\frac{L_2\bar a_0}{2}\|\vz_0-\vz^*\|^2.
  \]
  For every $k\geq1$ and $0<a\leq\bar a_k$,
  \begin{equation}\label{eq:uniform_trial_displacement}
    \|\vd_k(a)\|
    \leq
    \left(
      3+\left(2\sqrt{2}+\frac92\right)
      L_2\|\vz_0-\vz^*\|\bar a_k
    \right)\|\vz_0-\vz^*\|.
  \end{equation}
\end{lemma}

\begin{lemma}[Lipschitz Continuity of the Trial Displacement]
  \label{lem:trial_map_lipschitz}
  For every $k\geq1$, with $\vd_k(0)=0$ as in
  \eqref{eq:trial_displacement_map},
  \begin{equation}\label{eq:trial_map_lipschitz}
    \Lip(\vd_k;[0,\bar a_k])
    \leq L_2\|\vz_0-\vz^*\|^2\left(
      16\sqrt{2}+18
      +\left(6\sqrt{2}+\frac{27}{2}\right)
        L_2\|\vz_0-\vz^*\|\bar a_k
    \right).
  \end{equation}
\end{lemma}

\begin{proof}[Proof of Lemma~\ref{lem:second_order_merit_regular}]

  We first consider the initial iteration $k=0$. Since $A_0=0$, we have
  $\vz_{1/2}(a)=\vz_0$ for every $a>0$. Thus, the affine model
  $\oP_{0,a}$ is independent of $a$; denote this common model by $\oP_0$.
  The trial subproblem then reduces to
  \begin{equation*}
    0\in a(\oP_0+\oH)(\vz_1(a))+\vz_1(a)-\vz_0.
  \end{equation*}
  For $0<a\leq b$, the resolvent comparison argument in
  \cite[proof of Lemma~4.3(b)]{monteiro2012iteration} gives
  \begin{equation*}
    \|\vz_1(a)-\vz_1(b)\|
    \leq\frac{b-a}{b}\|\vz_1(b)-\vz_0\|.
  \end{equation*}
  Since $\vd_0(t)=\vz_1(t)-\vz_0$, it follows that
  \begin{equation*}
    \|\vd_0(a)-\vd_0(b)\|
    \leq\frac{b-a}{b}\|\vd_0(b)\|.
  \end{equation*}
  Therefore, by the definition of $\phi_0$ and the reverse triangle
  inequality,
  \begin{equation*}
    |\phi_0(a)-\phi_0(b)|
    \leq L_2a\|\vd_0(a)-\vd_0(b)\|
      +L_2(b-a)\|\vd_0(b)\| \leq 2L_2(b-a)
      \sup_{0<t\leq\bar a_0}\|\vd_0(t)\|.
  \end{equation*}
  The same bound holds when $a=0$, directly from $\phi_0(0)=0$.
  Lemma~\ref{lem:trial_displacement} now yields
  \[
    \Lip(\phi_0;[0,\bar a_0])
    \leq
    L_2\|\vz_0-\vz^*\|
    \left(4+L_2\|\vz_0-\vz^*\|\bar a_0\right)
    \leq
    29L_2\|\vz_0-\vz^*\|
    \left(1+L_2\|\vz_0-\vz^*\|\bar a_0\right)^2.
  \]

  Now suppose that $k\geq1$.  For brevity, set
  \[
    x_k:=L_2\|\vz_0-\vz^*\|\bar a_k.
  \]
   Substituting
  \eqref{eq:uniform_trial_displacement} and
  \eqref{eq:trial_map_lipschitz} into
  \eqref{eq:ls_sensitivity_reduction} gives
  \[
    \Lip(\phi_k;[0,\bar a_k])
    \leq L_2\|\vz_0-\vz^*\|
    \left(
      3
      +\left(18\sqrt{2}+\frac{45}{2}\right)x_k
      +\left(6\sqrt{2}+\frac{27}{2}\right)x_k^2
    \right).
  \]
  Since $x_k\geq0$, $18\sqrt{2}+45/2<48$, and
  $6\sqrt{2}+27/2<22$, the last display is bounded by
  $29L_2\|\vz_0-\vz^*\|(1+x_k)^2$. This
  proves~\eqref{eq:merit_lipschitz_bound} for
  every $k\geq0$.

  It remains to bound the upper trial value $\bar a_k$. For $k\geq1$,
  Theorem~\ref{thm:aep} gives
  \[
    A_k\|\oF(\vz_k)+\vv_k\|\leq2\|\vz_0-\vz^*\|.
  \]
  The same inequality is trivial for $k=0$ because $A_0=0$. Therefore,
  \[
    \frac{2A_k\|\oF(\vz_k)+\vv_k\|}{\varepsilon}
    \leq\frac{4\|\vz_0-\vz^*\|}{\varepsilon}.
  \]
  Substituting this estimate into the definition in~\eqref{eq:ls_cap} proves
  the cap bound in Lemma~\ref{lem:second_order_merit_regular}.

\end{proof}

\section{Technical Lemmas for the Line Search}
\label{appen:second_order_ls_technical}

\subsection{Proof of Lemma~\ref{lem:trial_displacement}}

We first present the following lemma, which relates the displacement $\|\vd_k(a)\|$ to $\|\vz_{k+\frac{1}{2}}(a)-\vz_k\|$ and the current residual $\|\oF(\vz_k)+\vv_k\|$.
\begin{lemma}
\label{lem:trial_displacement_bound}
Recall that $\vd_k(a):=\vz_{k+1}(a)-\vz_{k+\frac{1}{2}}(a)$. For any $a>0$, it holds that
\begin{equation}\label{eq:trial_displacement_bound}
    \|\vd_k(a)\|\leq \|\vz_{k+\frac{1}{2}}(a)-\vz_k\|
       +\frac{a^2}{A_k+a}\|\oF(\vz_k)+\vv_k\|+\frac{L_2a}{2}
       \|\vz_{k+\frac{1}{2}}(a)-\vz_k\|^2.
\end{equation}
\end{lemma}
\begin{proof}
Replacing $a_k$ by $a$ in~\eqref{eq:AEP_k+1/2} gives
\begin{equation}\label{eq:zk_plus_half}
    \vz_{k+\frac{1}{2}}(a)
    =\frac{a}{A_k+a}\vz_0
       +\frac{A_k}{A_k+a}
          \bigl(\vz_k-a(\oF(\vz_k)+\vv_k)\bigr).
\end{equation}
The trial subproblem in~\eqref{eq:trial_subproblem} also gives
\begin{equation*}
    \vz_{k+1}(a)
    \in\frac{a}{A_k+a}\vz_0+\frac{A_k}{A_k+a}\vz_k
       -a\oP_{k,a}(\vz_{k+1}(a))-a\oH(\vz_{k+1}(a)).
\end{equation*}
Combining these two relations yields
\begin{equation*}
    \frac{\vz_{k+\frac{1}{2}}(a)-\vz_{k+1}(a)}{a}
       +\frac{A_k}{A_k+a}(\oF(\vz_k)+\vv_k)
    \in(\oP_{k,a}+\oH)(\vz_{k+1}(a)).
\end{equation*}
Since $\vv_k\in\oH(\vz_k)$, we also have
\begin{equation*}
    \oP_{k,a}(\vz_k)+\vv_k\in(\oP_{k,a}+\oH)(\vz_k).
\end{equation*}
Monotonicity of $\oP_{k,a}+\oH$, applied to these two graph points at $\vz_{k+1}(a)$ and $\vz_k$, leads to
\begin{align*}
    0
    &\leq\left\langle
       \vz_{k+\frac{1}{2}}(a)-\vz_{k+1}(a),
       \vz_{k+1}(a)-\vz_k
    \right\rangle\\
    &\phantom{{}\leq{}} -a\left\langle
       \frac{a}{A_k+a}(\oF(\vz_k)+\vv_k)
          +(\oP_{k,a}(\vz_k)-\oF(\vz_k)),
       \vz_{k+1}(a)-\vz_k
    \right\rangle.
\end{align*}
The three-point identity gives
\begin{align*}
    &\phantom{{}={}}\left\langle
       \vz_{k+\frac{1}{2}}(a)-\vz_{k+1}(a),
       \vz_{k+1}(a)-\vz_k
    \right\rangle\\
    &=\frac12\|\vz_{k+\frac{1}{2}}(a)-\vz_k\|^2
       -\frac12\|\vz_{k+1}(a)-\vz_{k+\frac{1}{2}}(a)\|^2
       -\frac12\|\vz_{k+1}(a)-\vz_k\|^2.
\end{align*}
By Cauchy--Schwarz and Young's inequality,
\begin{align*}
    &-a\left\langle
       \frac{a}{A_k+a}(\oF(\vz_k)+\vv_k)
          +(\oP_{k,a}(\vz_k)-\oF(\vz_k)),
       \vz_{k+1}(a)-\vz_k
    \right\rangle\\
    &\quad\leq
       a\left\|
          \frac{a}{A_k+a}(\oF(\vz_k)+\vv_k)
             +(\oP_{k,a}(\vz_k)-\oF(\vz_k))
       \right\|\|\vz_{k+1}(a)-\vz_k\|\\
    &\quad\leq
       \frac{a^2}{2}\left\|
          \frac{a}{A_k+a}(\oF(\vz_k)+\vv_k)
             +(\oP_{k,a}(\vz_k)-\oF(\vz_k))
       \right\|^2
       +\frac12\|\vz_{k+1}(a)-\vz_k\|^2.
\end{align*}
Combining the preceding three displays and canceling
$\frac12\|\vz_{k+1}(a)-\vz_k\|^2$ gives
\begin{align*}
    \|\vz_{k+1}(a)-\vz_{k+\frac{1}{2}}(a)\|^2
    &\leq\|\vz_{k+\frac{1}{2}}(a)-\vz_k\|^2\\
    &\quad+a^2\left\|
       \frac{a}{A_k+a}(\oF(\vz_k)+\vv_k)
          +(\oP_{k,a}(\vz_k)-\oF(\vz_k))
    \right\|^2.
\end{align*}
Taking square roots and using
$\sqrt{x^2+y^2}\leq x+y$ for $x,y\geq0$, we obtain
\begin{equation*}
    \|\vd_k(a)\| \leq\|\vz_{k+\frac{1}{2}}(a)-\vz_k\|+a\left\|
       \frac{a}{A_k+a}(\oF(\vz_k)+\vv_k)
          +(\oP_{k,a}(\vz_k)-\oF(\vz_k))
    \right\|.
\end{equation*}
Finally, the triangle inequality and the Taylor remainder bound imply
\begin{equation*}
    \left\|
       \frac{a}{A_k+a}(\oF(\vz_k)+\vv_k)
          +(\oP_{k,a}(\vz_k)-\oF(\vz_k))
    \right\| \leq\frac{a}{A_k+a}\|\oF(\vz_k)+\vv_k\| +\frac{L_2}{2}
       \|\vz_{k+\frac{1}{2}}(a)-\vz_k\|^2.
\end{equation*}
Substitution into the previous estimate proves
\eqref{eq:trial_displacement_bound}.
\end{proof}

We now establish the two uniform trajectory estimates used to bound
$\vd_k$.

\begin{lemma}[Uniform Trajectory Bounds]
\label{lem:bounded_trajectory}
Every outer iterate generated before termination satisfies
\begin{equation}\label{eq:bounded_trajectory}
    \|\vz_0-\vz_k-A_k(\oF(\vz_k)+\vv_k)\|
    \leq3\|\vz_0-\vz^*\|, \quad
    \|\vz_k-\vz_0\|\leq3\|\vz_0-\vz^*\|.
\end{equation}
\end{lemma}

\begin{proof}
For $k\geq1$, the update rule at accepted step $k-1$ gives
\begin{equation*}
    \vz_k-\vz_0=\frac{A_{k-1}}{A_k}(\vz_{k-1}-\vz_0)
       -a_{k-1}(\oF(\vz_k)+\vv_k) +a_{k-1}(\oF(\vz_k)-\oP_{k-1}(\vz_k)).
\end{equation*}
By Theorem~\ref{thm:aep}, we have
\begin{equation*}
    \|\oF(\vz_k)+\vv_k\|
    \leq\frac{2\|\vz_0-\vz^*\|}{A_k}.
\end{equation*}
The individual $(k-1)$th term in the stability estimate
\eqref{eq:aep_stability}, with $\rho=1/\sqrt{2}$, gives
\begin{equation*}
    \|\vz_k-\vz_{k-\frac12}\|
    \leq\frac{\sqrt{2}\|\vz_0-\vz^*\|a_{k-1}}{A_k}.
\end{equation*}
Moreover, \eqref{eq:trial_error_merit} and the upper acceptance
condition in~\eqref{eq:ls_acceptance_window} imply
\begin{equation*}
    a_{k-1}\|\oF(\vz_k)-\oP_{k-1}(\vz_k)\|
    \leq \frac{1}{\sqrt{2}}\|\vz_k-\vz_{k-\frac12}\|
    \leq \frac{\|\vz_0-\vz^*\|a_{k-1}}{A_k}.
\end{equation*}
Consequently,
\begin{equation*}
    \|\vz_k-\vz_0\|
    \leq
    \frac{A_{k-1}}{A_k}\|\vz_{k-1}-\vz_0\|
    +\frac{a_{k-1}}{A_k}\,3\|\vz_0-\vz^*\|.
\end{equation*}
Since $A_k=A_{k-1}+a_{k-1}$ and $\|\vz_0-\vz_0\|=0$, induction proves
the second inequality in~\eqref{eq:bounded_trajectory}. For the first,
the Lyapunov descent property gives $V_k\leq V_0=0$, while
\eqref{eq:Lyapunov} gives
\begin{equation*}
    2V_k=
    \|\vz_0-\vz_k-A_k(\oF(\vz_k)+\vv_k)\|^2
    -\|\vz_0-\vz_k\|^2.
\end{equation*}
Thus we obtain $\|\vz_0-\vz_k-A_k(\oF(\vz_k)+\vv_k)\|\leq \|\vz_0-\vz_k\|$, which is further bounded
by $3\|\vz_0-\vz^*\|$.
\end{proof}

\begin{proof}[Proof of Lemma~\ref{lem:trial_displacement}]
We distinguish the initial iteration from the subsequent ones.

Suppose first that $k=0$. Then $A_0=0$ and
$\vz_{1/2}(a)=\vz_0$. Let $\vv^*\in\oH(\vz^*)$ satisfy
$\oF(\vz^*)+\vv^*=0$. By the trial subproblem,
\begin{equation*}
    \frac{\vz_0-\vz_1(a)}{a}
    \in(\oP_{0,a}+\oH)(\vz_1(a)),
\end{equation*}
whereas
$\oP_{0,a}(\vz^*)+\vv^*\in(\oP_{0,a}+\oH)(\vz^*)$.
Monotonicity and the Cauchy--Schwarz inequality therefore yield
\begin{align*}
    \|\vz_1(a)-\vz^*\|^2
    &\leq\left\langle
       \vz_0-\vz^*
       -a\bigl(\oP_{0,a}(\vz^*)-\oF(\vz^*)\bigr),
       \vz_1(a)-\vz^*
    \right\rangle \\
    &\leq
       \left(\|\vz_0-\vz^*\|
       +a\|\oP_{0,a}(\vz^*)-\oF(\vz^*)\|\right)
       \|\vz_1(a)-\vz^*\|.
\end{align*}
It follows from the Taylor remainder bound that
\begin{equation*}
    \|\vz_1(a)-\vz^*\|
    \leq \|\vz_0-\vz^*\|
       +a\|\oP_{0,a}(\vz^*)-\oF(\vz^*)\|
    \leq \|\vz_0-\vz^*\|
       +\frac{L_2a}{2}\|\vz_0-\vz^*\|^2.
\end{equation*}
Consequently, for $0<a\leq\bar a_0$,
\begin{equation}\label{eq:trial_displacement_k0}
    \|\vd_0(a)\|
    =\|\vz_1(a)-\vz_0\|
    \leq2\|\vz_0-\vz^*\|
       +\frac{L_2\bar a_0}{2}\|\vz_0-\vz^*\|^2.
\end{equation}

Now let $k\geq1$. Because iteration $k-1$ was accepted, the lower
acceptance condition in~\eqref{eq:ls_acceptance_window} and the
individual $(k-1)$th term in~\eqref{eq:aep_stability} imply
\begin{equation}\label{eq:A_lower_bound}
    1
    \leq L_2a_{k-1}\|\vz_k-\vz_{k-\frac12}\|
    \leq \sqrt{2}L_2\|\vz_0-\vz^*\|
       \frac{a_{k-1}^2}{A_k}
    \leq \sqrt{2}L_2\|\vz_0-\vz^*\|A_k,
\end{equation}
where we used $a_{k-1} \leq A_k$ in the last inequality.
Together with Theorem~\ref{thm:aep}, this gives
\begin{equation}\label{eq:outer_residual_uniform}
    \|\oF(\vz_k)+\vv_k\|
    \leq\frac{2\|\vz_0-\vz^*\|}{A_k}
    \leq2\sqrt{2}L_2\|\vz_0-\vz^*\|^2.
\end{equation}
Furthermore, \eqref{eq:zk_plus_half} and
Lemma~\ref{lem:bounded_trajectory} show that
\begin{equation}\label{eq:trial_half_displacement}
    \|\vz_{k+\frac{1}{2}}(a)-\vz_k\|=\frac{a}{A_k+a}
      \|\vz_0-\vz_k-A_k(\oF(\vz_k)+\vv_k)\| \leq3\|\vz_0-\vz^*\|\frac{a}{A_k+a}.
\end{equation}
Substituting \eqref{eq:outer_residual_uniform} and
\eqref{eq:trial_half_displacement} into
\eqref{eq:trial_displacement_bound}, and using
\begin{equation*}
    \frac{a}{A_k+a}\leq 1, \qquad \frac{a^2}{A_k+a}\leq a,
    \qquad
    a\left(\frac{a}{A_k+a}\right)^2\leq a
    \leq\bar a_k,
\end{equation*}
we obtain
\begin{equation}\label{eq:trial_displacement_kpos}
    \|\vd_k(a)\|
    \leq3\|\vz_0-\vz^*\|
       +\left(2\sqrt{2}+\frac92\right)
          L_2\|\vz_0-\vz^*\|^2\bar a_k.
\end{equation}
This completes the proof.

\end{proof}

\subsection{Proof of Lemma~\ref{lem:trial_map_lipschitz}}

We first record a useful result used in the proof of Lemma~\ref{lem:trial_displacement}. Since iteration
$k-1$ was accepted, from \eqref{eq:A_lower_bound} we have
\begin{equation}\label{eq:A_inverse_upper_bound}
    \frac{1}{A_k}
    \leq\sqrt{2}L_2\|\vz_0-\vz^*\|,
    \qquad k\geq1.
\end{equation}
For $t>0$, define the trial weight and center
\begin{equation*}
    \alpha(t):=\frac{t}{A_k+t},
    \qquad
    \vc(t):=\vz_k+\alpha(t)(\vz_0-\vz_k).
\end{equation*}

By \eqref{eq:trial_displacement_map} and the triangle inequality, for
$a,b>0$,
\begin{equation}\label{eq:displacement_difference_reduction}
\begin{aligned}
    \|\vd_k(a)-\vd_k(b)\|
    &=\|\bigl(\vz_{k+1}(a)-\vz_{k+\frac{1}{2}}(a)\bigr)
       -\bigl(\vz_{k+1}(b)-\vz_{k+\frac{1}{2}}(b)\bigr)\| \\
    &\leq \|\vz_{k+1}(a)-\vz_{k+1}(b)\|
       +\|\vz_{k+\frac{1}{2}}(a)-\vz_{k+\frac{1}{2}}(b)\|.
\end{aligned}
\end{equation}
Moreover, \eqref{eq:zk_plus_half} and the first bound in
\eqref{eq:bounded_trajectory} give
\begin{equation}\label{eq:trial_predictor_sensitivity}
   \|\vz_{k+\frac{1}{2}}(a)-\vz_{k+\frac{1}{2}}(b)\|
   = |\alpha(a)-\alpha(b)|
    \bigl\|\vz_0-\vz_k-A_k(\oF(\vz_k)+\vv_k)\bigr\|
   \leq3\|\vz_0-\vz^*\||\alpha(a)-\alpha(b)|.
\end{equation}
We next control the first term in
\eqref{eq:displacement_difference_reduction}.

\begin{lemma}
\label{lem:trial_sensitivity}
For any $a,b>0$,
\begin{align}
    \|\vz_{k+1}(a)-\vz_{k+1}(b)\|
    &\leq
       |\alpha(a)-\alpha(b)|\|\vz_0-\vz_k\|
       +b\|\oP_{k,b}(\vz_{k+1}(a))-\oP_{k,a}(\vz_{k+1}(a))\|\notag\\
    &\quad+\frac{|a-b|}{a}
       \|\vc(a)-\vz_{k+1}(a)\|.
    \label{eq:pre_model_sensitivity}
\end{align}
\end{lemma}
\begin{proof}
The trial inclusions in \eqref{eq:trial_subproblem} imply, for $a,b>0$,
\begin{align}
    \frac{\vc(a)-\vz_{k+1}(a)}{a}
       +\oP_{k,b}(\vz_{k+1}(a))-\oP_{k,a}(\vz_{k+1}(a))
    &\in(\oP_{k,b}+\oH)(\vz_{k+1}(a)),\notag\\
    \frac{\vc(b)-\vz_{k+1}(b)}{b}
    &\in(\oP_{k,b}+\oH)(\vz_{k+1}(b)).
    \label{eq:trial_graph_points}
\end{align}
By monotonicity of $\oP_{k,b}+\oH$, we have
\begin{equation*}
    0\leq \Biggl\langle
       \frac{\vc(a)-\vz_{k+1}(a)}{a}
       +\oP_{k,b}(\vz_{k+1}(a))-\oP_{k,a}(\vz_{k+1}(a))
       -\frac{\vc(b)-\vz_{k+1}(b)}{b},
       \vz_{k+1}(a)-\vz_{k+1}(b)
    \Biggr\rangle.
\end{equation*}
Since $\vc(a)-\vc(b)=(\alpha(a)-\alpha(b))(\vz_0-\vz_k)$,
multiplying the monotonicity inequality by $b$ and rearranging gives
\begin{align}
    \|\vz_{k+1}(a)-\vz_{k+1}(b)\|^2
    &\leq\Biggl\langle
       (\alpha(a)-\alpha(b))(\vz_0-\vz_k)
       +b\bigl(\oP_{k,b}(\vz_{k+1}(a))-\oP_{k,a}(\vz_{k+1}(a))\bigr)\notag\\
    &\qquad
       +\frac{b-a}{a}(\vc(a)-\vz_{k+1}(a)),
       \vz_{k+1}(a)-\vz_{k+1}(b)
    \Biggr\rangle.
    \label{eq:two_trial_monotonicity}
\end{align}
If the two trial points coincide, the next estimate is immediate;
otherwise, Cauchy--Schwarz, cancellation of $\|\vz_{k+1}(a)-\vz_{k+1}(b)\|$, and the
triangle inequality yield \eqref{eq:pre_model_sensitivity}.
\end{proof}

In the next lemma, we control the second term on the right-hand side of \eqref{eq:pre_model_sensitivity}.
\begin{lemma}
\label{lem:trial_model_sensitivity}
For any $a,b>0$,
\begin{equation*}
    b\|\oP_{k,b}(\vz_{k+1}(a))-\oP_{k,a}(\vz_{k+1}(a))\|
    \leq3L_2\|\vz_0-\vz^*\|
       \left(\|\vd_k(a)\|+\frac32\|\vz_0-\vz^*\|\right)|a-b|.
\end{equation*}
\end{lemma}

\begin{proof}

For brevity, set
$\vy_a:=\vz_{k+\frac{1}{2}}(a)$ and
$\vy_b:=\vz_{k+\frac{1}{2}}(b)$. Directly expanding the two affine models gives
\begin{equation*}
    \oP_{k,b}(\vw)-\oP_{k,a}(\vw)
    =\oF(\vy_b)-\oF(\vy_a)
       -D\oF(\vy_b)(\vy_b-\vy_a)
       +\bigl(D\oF(\vy_b)-D\oF(\vy_a)\bigr)(\vw-\vy_a).
\end{equation*}
The first three terms on the right form a Taylor remainder, and the last is
controlled by the Lipschitz continuity of the Jacobian. Hence
\begin{equation}\label{eq:model_sensitivity}
    \|\oP_{k,b}(\vw)-\oP_{k,a}(\vw)\|
    \leq
       \frac{L_2}{2}\|\vy_a-\vy_b\|^2
       +L_2\|\vy_a-\vy_b\|\|\vw-\vy_a\|.
\end{equation}
Taking $\vw=\vz_{k+1}(a)$ in \eqref{eq:model_sensitivity} and using
$\|\vz_{k+1}(a)-\vy_a\|=\|\vd_k(a)\|$ gives
\begin{equation*}
    \|\oP_{k,b}(\vz_{k+1}(a))-\oP_{k,a}(\vz_{k+1}(a))\|
    \leq
       \frac{L_2}{2}\|\vy_a-\vy_b\|^2
       +L_2\|\vy_a-\vy_b\|\|\vd_k(a)\|.
\end{equation*}

The trial-weight difference satisfies
\begin{align}
    b|\alpha(a)-\alpha(b)|
    &= b\left|\frac{a}{A_k+a} - \frac{b}{A_k+b} \right| = |a-b|\frac{b}{A_k+b}\frac{A_k}{A_k+a}
      \leq|a-b|,
      \label{eq:scaled_trial_weight_difference}\\
    |\alpha(a)-\alpha(b)|&\leq1.
      \label{eq:unit_trial_weight_difference}
\end{align}
For the quadratic term, \eqref{eq:trial_predictor_sensitivity},
\eqref{eq:scaled_trial_weight_difference}, and
\eqref{eq:unit_trial_weight_difference} give
\begin{align*}
    \frac{bL_2}{2}
       \|\vz_{k+\frac{1}{2}}(a)-\vz_{k+\frac{1}{2}}(b)\|^2
    &\leq\frac92bL_2\|\vz_0-\vz^*\|^2
       |\alpha(a)-\alpha(b)|^2
       \\
    &\leq\frac92L_2\|\vz_0-\vz^*\|^2|a-b|,
\end{align*}
where the last inequality follows by multiplying
\eqref{eq:scaled_trial_weight_difference} and
\eqref{eq:unit_trial_weight_difference}. For the linear term,
\eqref{eq:trial_predictor_sensitivity} and
\eqref{eq:scaled_trial_weight_difference} give
\begin{equation*}
    bL_2\|\vz_{k+\frac{1}{2}}(a)-\vz_{k+\frac{1}{2}}(b)\|\|\vd_k(a)\|
    \leq3L_2\|\vz_0-\vz^*\|\|\vd_k(a)\||a-b|.
\end{equation*}
Adding the two bounds on linear and quadratic terms therefore yields
\begin{equation*}
    b\|\oP_{k,b}(\vz_{k+1}(a))-\oP_{k,a}(\vz_{k+1}(a))\|
    \leq3L_2\|\vz_0-\vz^*\|
       \left(\|\vd_k(a)\|+\frac32\|\vz_0-\vz^*\|\right)|a-b|.
\end{equation*}

\end{proof}

Next, the following lemma controls the last term on the right-hand side of \eqref{eq:pre_model_sensitivity}.
\begin{lemma}
\label{lem:trial_center_gap}
For $0<a\leq\bar a_k$,
\begin{equation}\label{eq:uniform_trial_center_gap}
    \|\vc(a)-\vz_{k+1}(a)\|
    \leq\left(10\sqrt{2}+\frac92\right)
       L_2\|\vz_0-\vz^*\|^2a.
\end{equation}
\end{lemma}

\begin{proof}

By \eqref{eq:zk_plus_half} and the definition of $\vd_k(a)$, we can write
\begin{equation*}
    \vz_{k+1}(a)=\vz_{k+\frac{1}{2}}(a)+\vd_k(a)=\vz_k+\frac{a}{A_k+a}
       \bigl(\vz_0-\vz_k-A_k(\oF(\vz_k)+\vv_k)\bigr)
       +\vd_k(a).
\end{equation*}
Subtracting this identity from $\vc(a) = \vz_k+\frac{a}{A_k+a}(\vz_0-\vz_k)$ gives
\begin{equation}\label{eq:trial_center_identity}
    \vc(a)-\vz_{k+1}(a)
    =\frac{aA_k}{A_k+a}(\oF(\vz_k)+\vv_k)-\vd_k(a).
\end{equation}
We first derive another bound on $\vd_k(a)$.
Applying \eqref{eq:trial_displacement_bound}, followed by
Lemma~\ref{lem:bounded_trajectory} and~\eqref{eq:A_inverse_upper_bound}, gives
\begin{align*}
    \|\vd_k(a)\|
    &\leq\frac{6a}{A_k+a}\|\vz_0-\vz^*\|
       +\frac{2a^2}{A_k(A_k+a)}\|\vz_0-\vz^*\|
       +\frac{9L_2a^3}{2(A_k+a)^2}\|\vz_0-\vz^*\|^2\\
    &\leq\frac{8a}{A_k}\|\vz_0-\vz^*\|
       +\frac92L_2a\|\vz_0-\vz^*\|^2\\
    &\leq\left(8\sqrt{2}+\frac92\right)
       L_2\|\vz_0-\vz^*\|^2a,
\end{align*}
where the second inequality uses $A_k+a\geq A_k$ and
$a/(A_k+a)\leq1$.

We now use this bound directly in \eqref{eq:trial_center_identity}.  By the
triangle inequality and \eqref{eq:A_inverse_upper_bound},
\begin{align*}
    \|\vc(a)-\vz_{k+1}(a)\|
    &\leq\frac{aA_k}{A_k+a}\|\oF(\vz_k)+\vv_k\|+\|\vd_k(a)\|\\
    &\leq\frac{2a}{A_k}\|\vz_0-\vz^*\|
       +\left(8\sqrt{2}+\frac92\right)
          L_2\|\vz_0-\vz^*\|^2a\\
    &\leq\left(10\sqrt{2}+\frac92\right)
       L_2\|\vz_0-\vz^*\|^2a.
\end{align*}
This proves \eqref{eq:uniform_trial_center_gap}.
\end{proof}

\begin{proof}[Proof of Lemma~\ref{lem:trial_map_lipschitz}]

Fix $a,b\in(0,\bar a_k]$. By~\eqref{eq:A_inverse_upper_bound}, we have
\begin{equation*}
    |\alpha(a)-\alpha(b)|
    =\frac{A_k|a-b|}{(A_k+a)(A_k+b)}\leq\frac{|a-b|}{A_k}\leq\sqrt{2}L_2\|\vz_0-\vz^*\||a-b|.
\end{equation*}
Hence, using Lemma~\ref{lem:bounded_trajectory}, the first term on the right-hand side of \eqref{eq:pre_model_sensitivity} in Lemma~\ref{lem:trial_sensitivity} can be bounded by
\begin{equation*}
  |\alpha(a)-\alpha(b)|\|\vz_0-\vz_k\| \leq 3\sqrt{2}L_2\|\vz_0-\vz^*\|^2|a-b|.
\end{equation*}
Together with Lemmas~\ref{lem:trial_model_sensitivity} and~\ref{lem:trial_center_gap}, we obtain from Lemma~\ref{lem:trial_sensitivity} that
\begin{align*}
    \|\vz_{k+1}(a)-\vz_{k+1}(b)\|
    &\leq3\sqrt{2}L_2\|\vz_0-\vz^*\|^2|a-b|\\
    &\quad+3L_2\|\vz_0-\vz^*\|
       \left(\|\vd_k(a)\|+\frac32\|\vz_0-\vz^*\|\right)|a-b| \\
    &\quad+\left(10\sqrt{2}+\frac92\right)
       L_2\|\vz_0-\vz^*\|^2|a-b|.
\end{align*}
For $a\in(0,\bar a_k]$, Lemma~\ref{lem:trial_displacement} gives
\begin{equation*}
    \|\vd_k(a)\|
    \leq\left(3+\left(2\sqrt{2}+\frac92\right)
       L_2\|\vz_0-\vz^*\|\bar a_k\right)\|\vz_0-\vz^*\|.
\end{equation*}
Substituting this estimate yields
\begin{align*}
    \|\vz_{k+1}(a)-\vz_{k+1}(b)\|
    &\leq L_2\|\vz_0-\vz^*\|^2\left(
       13\sqrt{2}+18
       +\left(6\sqrt{2}+\frac{27}{2}\right)
          L_2\|\vz_0-\vz^*\|\bar a_k
    \right)|a-b|.
\end{align*}

Finally,
note that
\begin{equation*}
    \|\vz_{k+\frac{1}{2}}(a)-\vz_{k+\frac{1}{2}}(b)\|
    \leq 3\|\vz_0-\vz^*\||\alpha(a)-\alpha(b)| \leq 3\sqrt{2}L_2\|\vz_0-\vz^*\|^2|a-b|.
\end{equation*}
Combining the last two estimates with
\eqref{eq:displacement_difference_reduction}, we obtain
\begin{align*}
    \|\vd_k(a)-\vd_k(b)\|
    &\leq L_2\|\vz_0-\vz^*\|^2\left(
       16\sqrt{2}+18
       +\left(6\sqrt{2}+\frac{27}{2}\right)
          L_2\|\vz_0-\vz^*\|\bar a_k
    \right)|a-b|.
\end{align*}
This proves the asserted bound on $(0,\bar a_k]$. Finally,
\eqref{eq:trial_displacement_bound} implies
$\vd_k(a)\to0$ as $a\downarrow0$.  Since $\vd_k(0)=0$, letting either
argument in the preceding estimate tend to zero extends the same bound
to $[0,\bar a_k]$ and proves \eqref{eq:trial_map_lipschitz}.

\end{proof}

\section{Proofs for the AEP Tensor Method}
\label{appen:tensor_complexity}

We give the complete AEP Tensor method and its bisection line search subroutine,
then prove the outer iteration and line-search bounds in
Theorem~\ref{thm:higher_order_complexity}.

\subsection{Regularized Model and Algorithm}

\begin{lemma}[{\cite[Lemma 7.1]{jiang2024generalized}}]
  \label{lem:regularized_taylor}

  Under Assumption~\ref{assum:Lips_higher_derivative}, define

  \begin{equation*}
    \oT_{\lambda}^{(p-1)}(\vz';\vz)
    :=
    \oT^{(p-1)}(\vz';\vz)
    +\frac{\lambda}{(p-1)!}
    \|\vz'-\vz\|^{p-1}(\vz'-\vz).
  \end{equation*}
  For every $\lambda\geq0$,
  \begin{equation}\label{eq:pth_taylor_remainder}
    \|\oF(\vz')-\oT_{\lambda}^{(p-1)}(\vz';\vz)\|
    \leq
    \frac{L_p+p\lambda}{p!}\|\vz'-\vz\|^p.
  \end{equation}
  Moreover, if $\lambda\geq L_p$, then
  $\oT_{\lambda}^{(p-1)}(\cdot;\vz)$ is maximal monotone and has
  domain $\reals^d$.
\end{lemma}

Algorithm~\ref{alg:aep_tensor} uses $\lambda=L_p$ and $M_p=(p+1)L_p$.
Its surrogate is therefore the regularized Taylor model in
\eqref{eq:surrogate-p}. For a trial step $a>0$, let
$\vz_{k+\frac{1}{2}}(a)$ and $\vz_{k+1}(a)$ denote the corresponding
intermediate point and the solution of the regularized Taylor-model
subproblem. The bisection acceptance window is
\begin{equation}\label{eq:tensor_acceptance_window}
  \frac{p!}{2}
  \leq (p+1)L_p a\|\vz_{k+1}(a)-\vz_{k+\frac{1}{2}}(a)\|^{p-1}
  \leq \frac{p!}{\sqrt{2}}.
\end{equation}
The line search in
Section~\ref{subsec:tensor_ls} either returns an $\varepsilon$-accurate point
or supplies a step satisfying~\eqref{eq:tensor_acceptance_window}.

\begin{algorithm}[!t]\small
  \caption{Anchored Extra-Proximal Tensor Method}
  \label{alg:aep_tensor}
  \begin{algorithmic}[1]
    \STATE \textbf{Input:} order $p\geq2$, initial point
    $\vz_0\in\dom\oH$, target accuracy $\varepsilon>0$, and $L_p>0$
    \STATE \textbf{Initialize:} choose $\vv_0\in\oH(\vz_0)$, set
    $A_0=0$, $\lambda=L_p$, and $M_p=(p+1)L_p$
    \FOR{$k=0,1,\ldots$}
      \IF{$\|\oF(\vz_k)+\vv_k\|\leq\varepsilon$}
        \STATE Return $\vz_k$
      \ENDIF
      \STATE
      Apply the bisection line search in Section~\ref{subsec:tensor_ls} to select a step size $a_k$

      \STATE Compute
      $\displaystyle
        \vz_{k+\frac{1}{2}}
        =
        \frac{a_k}{A_k+a_k}\vz_0
        +\frac{A_k}{A_k+a_k}
        \bigl(\vz_k-a_k(\oF(\vz_k)+\vv_k)\bigr)$
      \STATE Find $\vz_{k+1}$ such that
      \begin{equation*}
        0\in
        a_k\oT_{\lambda}^{(p-1)}
        (\vz_{k+1};\vz_{k+\frac{1}{2}})
        +a_k\oH(\vz_{k+1})
        +\vz_{k+1}
        -\frac{a_k}{A_k+a_k}\vz_0
        -\frac{A_k}{A_k+a_k}\vz_k
      \end{equation*}
      \STATE Update
      \begin{equation*}
        A_{k+1}=A_k+a_k,
        \qquad
        \vv_{k+1}
        =
        \frac{1}{a_k}
        \left(
          \frac{a_k}{A_k+a_k}\vz_0
          +\frac{A_k}{A_k+a_k}\vz_k
          -\vz_{k+1}
        \right)
        -\oT_{\lambda}^{(p-1)}
        (\vz_{k+1};\vz_{k+\frac{1}{2}})
      \end{equation*}
    \ENDFOR
  \end{algorithmic}
\end{algorithm}

For $\rho\in(0,1)$, Lemma~\ref{lem:regularized_taylor} with $\lambda=L_p$
ensures that the AEP relative-error condition~\eqref{eq:surrogate_error_condition}
is satisfied whenever
\begin{equation}\label{eq:aep_tensor_stepsize}
  a_k \leq \frac{p!\rho}{(p+1)L_p\|\vz_{k+1}-\vz_{k+\frac{1}{2}}\|^{p-1}},
\end{equation}
provided that $\vz_{k+1}\neq\vz_{k+\frac{1}{2}}$. With $\rho=1/\sqrt{2}$, each step accepted by bisection also
satisfies
\begin{equation}\label{eq:stepsize_lb_pth}
  a_k \geq \frac{p!\rho}{\sqrt{2}M_p\|\vz_{k+1}-\vz_{k+\frac{1}{2}}\|^{p-1}}.
\end{equation}
Together, these inequalities are equivalent to the acceptance window
\eqref{eq:tensor_acceptance_window}. We use them to prove the growth of
$A_k$ in Section~\ref{subsec:tensor_complexity_analysis}.

\subsection{Bisection Line Search}\label{subsec:tensor_ls}

We first specify the upper trial step size and the bisection procedure.
Throughout, we set
\[
  \rho=\frac{1}{\sqrt{2}},
  \qquad
  \lambda=L_p,
  \qquad
  M_p=L_p+p\lambda=(p+1)L_p.
\]
At an outer iteration $k$, let
$\vz_{k+\frac{1}{2}}(a)$ be defined by~\eqref{eq:AEP_k+1/2} with
$a_k=a$.  Set
\[
  \oP_{k,a}(\vz)
  :=
  \oT_{\lambda}^{(p-1)}
  (\vz;\vz_{k+\frac{1}{2}}(a)),
\]
and let $\vz_{k+1}(a)$ be the unique solution of
\begin{equation}\label{eq:tensor_trial_subproblem}
  0\in
  a\oP_{k,a}(\vz_{k+1}(a))
  +a\oH(\vz_{k+1}(a))
  +\vz_{k+1}(a)
  -\left(
    \frac{a}{A_k+a}\vz_0
    +\frac{A_k}{A_k+a}\vz_k
  \right).
\end{equation}
Define
\begin{equation}\label{eq:tensor_trial_v}
  \vv_{k+1}(a)
  :=
  \frac{1}{a}\left(
    \frac{a}{A_k+a}\vz_0
    +\frac{A_k}{A_k+a}\vz_k
    -\vz_{k+1}(a)
  \right)
  -\oP_{k,a}(\vz_{k+1}(a))
  \in\oH(\vz_{k+1}(a)).
\end{equation}

For $a>0$, define the merit function
\begin{equation}\label{eq:tensor_merit}
  \psi_k(a):=M_pa\|\vz_{k+1}(a)-\vz_{k+\frac{1}{2}}(a)\|^{p-1},
  \qquad \psi_k(0):=0.
\end{equation}

We start the line search from the computable upper trial value
\begin{equation}\label{eq:tensor_ls_cap}
  \bar a_k
  :=
  \max\left\{
    \frac{2A_k\|\oF(\vz_k)+\vv_k\|}{\varepsilon},
    \left(
      \frac{2(1+\rho)}{\varepsilon}
      \left(\frac{p!\rho}{M_p}\right)^{1/(p-1)}
    \right)^{(p-1)/p}
  \right\}.
\end{equation}
If $\psi_k(\bar a_k)\leq p!\rho$, the method returns
$\vz_{k+1}(\bar a_k)$.  Otherwise, it initializes
$a^-=0$ and $a^+=\bar a_k$.  At every bisection step it evaluates
$a=(a^-+a^+)/2$ and performs the update
\begin{equation}\label{eq:tensor_bisection_update}
  \begin{cases}
    a^-\leftarrow a,
      &\psi_k(a)<p!\rho/\sqrt{2},\\
    a^+\leftarrow a,
      &\psi_k(a)>p!\rho,\\
    \text{accept }a,
      &p!\rho/\sqrt{2}\leq\psi_k(a)\leq p!\rho.
  \end{cases}
\end{equation}
Thus, with $\rho=1/\sqrt{2}$, the accepted interval is precisely
\eqref{eq:tensor_acceptance_window}.

We have the following lemma, which is the higher-order counterpart of
Lemma~\ref{lem:ls_residual}.
\begin{lemma}
  \label{lem:tensor_ls_residual}
  If $\psi_k(\bar a_k)\leq p!\rho$, then $
    \res^{\mathrm{tan}}(\vz_{k+1}(\bar a_k))
    \leq\varepsilon$.
\end{lemma}

\begin{proof}
  Equations~\eqref{eq:AEP_k+1/2} and~\eqref{eq:tensor_trial_v} give
  \begin{equation}\label{eq:tensor_trial_identity}
    \oP_{k,a}(\vz_{k+1}(a))+\vv_{k+1}(a)
    =
    \frac{A_k}{A_k+a}
    \bigl(\oF(\vz_k)+\vv_k\bigr)
    -\frac{\vz_{k+1}(a)-\vz_{k+\frac{1}{2}}(a)}{a}.
  \end{equation}
  Write
  $d(a):=\|\vz_{k+1}(a)-\vz_{k+\frac{1}{2}}(a)\|$.
  By the Taylor remainder bound in~\eqref{eq:pth_taylor_remainder},
  \begin{align*}
    \|\oF(\vz_{k+1}(a))+\vv_{k+1}(a)\|
    &\leq
    \frac{A_k}{A_k+a}\|\oF(\vz_k)+\vv_k\|
    +\frac{d(a)}{a}
    +\frac{M_p}{p!}d(a)^p.
  \end{align*}
  If $\psi_k(a)=M_pad(a)^{p-1}\leq p!\rho$, then
  $M_pd(a)^p/p!\leq\rho d(a)/a$ and
  $d(a)\leq(p!\rho/(M_pa))^{1/(p-1)}$.  Consequently,
  \begin{equation}\label{eq:tensor_cap_residual_bound}
    \|\oF(\vz_{k+1}(a))+\vv_{k+1}(a)\|
    \leq
    \frac{A_k\|\oF(\vz_k)+\vv_k\|}{a}
    +(1+\rho)
    \left(\frac{p!\rho}{M_p}\right)^{1/(p-1)}
    a^{-p/(p-1)}.
  \end{equation}
  At $a=\bar a_k$, each term on the right is at most
  $\varepsilon/2$ by~\eqref{eq:tensor_ls_cap}.  Finally,
  $\vv_{k+1}(\bar a_k)\in\oH(\vz_{k+1}(\bar a_k))$, which proves the
  claim.
\end{proof}

\subsection{Complexity Analysis}\label{subsec:tensor_complexity_analysis}

We prove the two bounds in Theorem~\ref{thm:higher_order_complexity},
starting with the outer iteration bound. The line-search bound then follows
from a regularity estimate for the merit function.

\begin{proof}[Proof of Theorem~\ref{thm:higher_order_complexity}: outer iteration bound]

  If an upper trial step size terminates the method during the first $T$
  outer iterations, the conclusion follows from
  Lemma~\ref{lem:tensor_ls_residual}. Otherwise, the line search returns
  $T$ accepted steps.

  By~\eqref{eq:pth_taylor_remainder} and the upper bound
  in~\eqref{eq:tensor_acceptance_window}, the relative-error condition
  in~\eqref{eq:surrogate_error_condition} is satisfied, and hence these
  $T$ accepted steps form an instance of AEP with $\rho=1/\sqrt{2}$.

  Moreover, the lower bound in~\eqref{eq:tensor_acceptance_window} gives,
  for every such step,
  \[
    M_pa_i
    \|\vz_{i+1}-\vz_{i+\frac{1}{2}}\|^{p-1}
    \geq\frac{p!}{2}.
  \]
  Applying Lemma~\ref{lem:stability} with $\rho=1/\sqrt{2}$ therefore
  yields
  \begin{align}
    \label{eq:appendix_tensor_sum}
    \sum_{i=0}^{k-1}
    \frac{A_{i+1}^2}{a_i^{2p/(p-1)}}
    &\leq
    \left(\frac{2M_p}{p!}\right)^{2/(p-1)}
    \sum_{i=0}^{k-1}
    \frac{A_{i+1}^2}{a_i^2}
    \|\vz_{i+1}-\vz_{i+\frac{1}{2}}\|^2
    \notag\\
    &\leq
    2\left(\frac{2M_p}{p!}\right)^{2/(p-1)}
    \|\vz_0-\vz^*\|^2,
    \qquad k=1,\ldots,T.
  \end{align}

  H\"older's inequality gives, for every $k=1,\ldots,T$,

  \[
    A_k^{2p/(3p-1)}
    \left(
      \sum_{i=0}^{k-1}
      \frac{A_{i+1}^2}{a_i^{2p/(p-1)}}
    \right)^{(p-1)/(3p-1)}
    \geq
    \sum_{i=0}^{k-1}A_{i+1}^{2(p-1)/(3p-1)}.
  \]
  Combining this inequality with~\eqref{eq:appendix_tensor_sum} yields
  \begin{equation}\label{eq:appendix_tensor_growth_recursion}
    A_k^{2p/(3p-1)}
    \geq
    c\sum_{j=1}^k A_j^{2(p-1)/(3p-1)},
    \qquad
    c:=
    \left[
      2\left(\frac{2M_p}{p!}\right)^{2/(p-1)}
      \|\vz_0-\vz^*\|^2
    \right]^{-(p-1)/(3p-1)}.
  \end{equation}
  Set
  \[
    B_k:=A_k^{2(p-1)/(3p-1)}.
  \]
  The sequence $\{B_k\}_{k\geq1}$ is positive and nondecreasing, and
  \eqref{eq:appendix_tensor_growth_recursion} becomes
  \[
    B_k^{p/(p-1)}
    \geq
    c\sum_{j=1}^k B_j.
  \]
  Applying Lemma~\ref{lem:growth} with $\alpha=p/(p-1)$ gives
  $B_k\geq(ck/p)^{p-1}$. Consequently,
  \begin{equation}\label{eq:appendix_tensor_A_growth}
    A_k
    =B_k^{(3p-1)/(2(p-1))}
    \geq
    \frac{p!}{2^{(p+1)/2}p^{(3p-1)/2}}
    \frac{k^{(3p-1)/2}}
    {M_p\|\vz_0-\vz^*\|^{p-1}}.
  \end{equation}
  For
  \[
    T=
    \left\lceil
      \left(
        \frac{2^{(p+3)/2}p^{(3p-1)/2}M_p
        \|\vz_0-\vz^*\|^p}{p!\varepsilon}
      \right)^{2/(3p-1)}
    \right\rceil,
  \]
  the last display implies
  $A_T\geq2\|\vz_0-\vz^*\|/\varepsilon$.
  Theorem~\ref{thm:aep} then gives
  \[
    \res^{\mathrm{tan}}(\vz_T)
    \leq
    \frac{2\|\vz_0-\vz^*\|}{A_T}
    \leq
    \varepsilon.
  \]
  Thus, the stopping test in Algorithm~\ref{alg:aep_tensor} returns
  $\vz_T$, unless the method has already terminated earlier. Finally,
  $M_p=(p+1)L_p$, so this choice of $T$ has the order stated in
  Theorem~\ref{thm:higher_order_complexity}.
\end{proof}

The next technical estimate is the higher-order analogue of
Lemma~\ref{lem:second_order_merit_regular}.  We record it in the form needed
for the bisection analysis and defer its proof to
Section~\ref{appen:tensor_merit_regular}.

\begin{lemma}[Regularity of the Tensor Merit Function]
  \label{lem:tensor_merit_regular}
  There is a constant $C_p>0$ depending only on $p$ such that, at every
  outer iteration before termination, $\psi_k$ is continuous on
  $[0,\bar a_k]$ and
  \begin{equation}\label{eq:tensor_merit_lipschitz}
    \Lip(\psi_k;[0,\bar a_k])
    \leq
    C_pM_p\|\vz_0-\vz^*\|^{p-1}
    \left(
      1+M_p\|\vz_0-\vz^*\|^{p-1}\bar a_k
    \right)^{2p-2}.
  \end{equation}
  Moreover, if the bisection phase is entered, then
  \begin{equation}\label{eq:tensor_cap_upper_bound}
    M_p\|\vz_0-\vz^*\|^{p-1}\bar a_k
    \leq
    C_p\max\left\{
      \frac{L_p\|\vz_0-\vz^*\|^p}{\varepsilon},
      \left(
        \frac{L_p\|\vz_0-\vz^*\|^p}{\varepsilon}
      \right)^{(p-1)/p}
    \right\}.
  \end{equation}
\end{lemma}

\begin{proof}[Proof of Theorem~\ref{thm:higher_order_complexity}: line-search bound]
  If the upper trial value is accepted, the claim is immediate.  Otherwise, the
  bisection endpoints satisfy
  \[
    \psi_k(a^-)<\frac{p!\rho}{\sqrt{2}},
    \qquad
    \psi_k(a^+)>p!\rho.
  \]
  After $t$ unsuccessful bisection trials, the interval has length
  $\bar a_k/2^t$.  Hence, by
  Lemma~\ref{lem:tensor_merit_regular},
  \[
    p!\rho\left(1-\frac{1}{\sqrt{2}}\right)
    <
    \psi_k(a^+)-\psi_k(a^-)
    \leq
    \Lip(\psi_k;[0,\bar a_k])\frac{\bar a_k}{2^t}.
  \]
  Thus the number of unsuccessful trials is at most
  \begin{equation*}
    \left\lceil
      \log_2\left(
        \frac{\bar a_k\Lip(\psi_k;[0,\bar a_k])}
        {p!\rho(1-1/\sqrt{2})}
      \right)
    \right\rceil.
  \end{equation*}
  To see explicitly that the polynomial exponent in
  \eqref{eq:tensor_merit_lipschitz} does not affect the logarithmic
  complexity, set
  \[
    Q:=\frac{L_p\|\vz_0-\vz^*\|^p}{\varepsilon},
    \qquad
    x_k:=M_p\|\vz_0-\vz^*\|^{p-1}\bar a_k.
  \]
  Equations~\eqref{eq:tensor_merit_lipschitz} and
  \eqref{eq:tensor_cap_upper_bound} imply
  \[
    \bar a_k\Lip(\psi_k;[0,\bar a_k])
    \leq C_px_k(1+x_k)^{2p-2}
    \leq C_p(1+Q)^{2p-1}.
  \]

  Including the initial capped trial and the final successful bisection
  trial, the total number of trial evaluations is

  \[
    \bigO_p\left(
      1+\log\left(
        1+\frac{L_p\|\vz_0-\vz^*\|^p}{\varepsilon}
      \right)
    \right).
  \]
\end{proof}

\section{Regularity of the Tensor Merit Function}
\label{appen:tensor_merit_regular}

Throughout this section, $C_p>0$ denotes a constant that depends only on
$p$ and whose actual value may change from one occurrence to the next.

We follow the same three-step argument used for the second-order merit
function in Appendix~\ref{appen:second_order_merit_regular}: reduce the merit
function estimate to bounds on the trial displacement, establish a uniform
bound on that displacement, and then control its Lipschitz modulus with respect to
the trial step size. Define
\begin{equation}\label{eq:tensor_trial_displacement_map}
  \vd_k(a):=\vz_{k+1}(a)-\vz_{k+\frac{1}{2}}(a),
  \qquad a>0,
  \qquad \vd_k(0):=0.
\end{equation}
As in the second-order case, write
\[
  \Lip(\vd_k;[0,\bar a_k])
  :=
  \sup_{0\leq a<b\leq\bar a_k}
  \frac{\|\vd_k(a)-\vd_k(b)\|}{|a-b|}.
\]
The following higher-order counterpart of Lemma~\ref{lem:ls_sensitivity}
reduces the Lipschitz modulus of $\psi_k$ to uniform and Lipschitz bounds
for this trial-displacement map.

\begin{lemma}[Tensor Merit Sensitivity]
  \label{lem:tensor_ls_sensitivity}
  Suppose that $\vd_k$ is Lipschitz continuous on $[0,\bar a_k]$.
  Then $\psi_k$ is Lipschitz continuous on the same interval, with
  \begin{equation}
    \label{eq:tensor_ls_sensitivity_reduction}
    \Lip(\psi_k;[0,\bar a_k])
    \leq
    M_p\Bigl(
      \sup_{0\leq a\leq\bar a_k}\|\vd_k(a)\|
    \Bigr)^{p-1}
    +(p-1)M_p\bar a_k
    \Bigl(
      \sup_{0\leq a\leq\bar a_k}\|\vd_k(a)\|
    \Bigr)^{p-2}
    \Lip(\vd_k;[0,\bar a_k]).
  \end{equation}
\end{lemma}

\begin{proof}
  For distinct $a,b\in[0,\bar a_k]$, the definition of $\psi_k$ gives
  \begin{align*}
    |\psi_k(a)-\psi_k(b)|
    &\leq
    M_p|a-b|\|\vd_k(a)\|^{p-1}
    +M_pb\left|
      \|\vd_k(a)\|^{p-1}-\|\vd_k(b)\|^{p-1}
    \right|.
  \end{align*}
  For all $x,y\geq0$, the mean-value theorem yields
  \[
    |x^{p-1}-y^{p-1}|
    \leq
    (p-1)\max\{x,y\}^{p-2}|x-y|.
  \]
  Applying this inequality together with the reverse triangle inequality,
  the uniform bound on the trial displacement, and $b\leq\bar a_k$, we
  obtain
  \begin{align*}
    |\psi_k(a)-\psi_k(b)|
    &\leq
    M_p\left(
      \sup_{0\leq t\leq\bar a_k}\|\vd_k(t)\|
    \right)^{p-1}|a-b|
    \\
    &\quad
    +(p-1)M_p\bar a_k
    \left(
      \sup_{0\leq t\leq\bar a_k}\|\vd_k(t)\|
    \right)^{p-2}
    \Lip(\vd_k;[0,\bar a_k])|a-b|.
  \end{align*}
  Dividing by $|a-b|$ and taking the supremum over $a$ and $b$
  proves~\eqref{eq:tensor_ls_sensitivity_reduction}.
\end{proof}

The next two lemmas control the size and variation of $\vd_k$. Their proofs
are deferred to Subsections~\ref{appen:tensor_trial_displacement}
and~\ref{appen:tensor_trial_map_lipschitz}.
\begin{lemma}[Trial Displacement]
  \label{lem:tensor_trial_displacement}
  For every $k\geq0$ and $0<a\leq\bar a_k$,
  \begin{equation}\label{eq:tensor_uniform_trial_displacement}
    \|\vd_k(a)\|
    \leq C_p\left(
      1+M_p\|\vz_0-\vz^*\|^{p-1}\bar a_k
    \right)\|\vz_0-\vz^*\|,
  \end{equation}
  where $C_p>0$ depends only on $p$.
\end{lemma}

\begin{lemma}[Lipschitz Continuity of the Trial Displacement]
  \label{lem:tensor_trial_map_lipschitz}
  For every $k\geq1$, with $\vd_k(0)=0$ as in
  \eqref{eq:tensor_trial_displacement_map},
  \begin{equation}\label{eq:tensor_trial_map_lipschitz}
    \Lip(\vd_k;[0,\bar a_k])
    \leq C_p M_p\|\vz_0-\vz^*\|^p\left(
      1
      + M_p\|\vz_0-\vz^*\|^{p-1}\bar a_k
    \right)^{p-1}.
  \end{equation}
\end{lemma}

\begin{proof}[Proof of Lemma~\ref{lem:tensor_merit_regular}]
  Set
  \[
    R:=\|\vz_0-\vz^*\|,
    \qquad
    x_k:=M_pR^{p-1}\bar a_k.
  \]
  We first consider $k=0$. In this case,
  $\vz_{1/2}(a)=\vz_0$, so the regularized Taylor model is independent of
  $a$. For $0<a\leq b$, monotonicity of this fixed model plus $\oH$, applied
  to the two trial inclusions, gives
  \[
    \|\vd_0(a)-\vd_0(b)\|
    \leq
    \frac{b-a}{b}\|\vd_0(b)\|.
  \]
  The mean-value theorem and the definition of $\psi_0$ therefore imply
  \begin{align*}
    |\psi_0(a)-\psi_0(b)|
    &\leq
    M_pa(p-1)
    \Bigl(
      \sup_{0<t\leq\bar a_0}\|\vd_0(t)\|
    \Bigr)^{p-2}
    \|\vd_0(a)-\vd_0(b)\|
    +M_p(b-a)
    \Bigl(
      \sup_{0<t\leq\bar a_0}\|\vd_0(t)\|
    \Bigr)^{p-1}
    \\
    &\leq
    pM_p(b-a)
    \Bigl(
      \sup_{0<t\leq\bar a_0}\|\vd_0(t)\|
    \Bigr)^{p-1}.
  \end{align*}
  The same conclusion holds when $a=0$ directly from $\psi_0(0)=0$.
  Lemma~\ref{lem:tensor_trial_displacement} now gives
  \[
    \Lip(\psi_0;[0,\bar a_0])
    \leq
    C_pM_pR^{p-1}(1+x_0)^{p-1}
    \leq
    C_pM_pR^{p-1}(1+x_0)^{2p-2}.
  \]

  Now suppose that $k\geq1$. Lemmas
  \ref{lem:tensor_trial_displacement} and
  \ref{lem:tensor_trial_map_lipschitz} give
  \begin{align*}
    \sup_{0\leq a\leq\bar a_k}\|\vd_k(a)\|
    &\leq C_pR(1+x_k),
    \\
    \Lip(\vd_k;[0,\bar a_k])
    &\leq C_pM_pR^p(1+x_k)^{p-1}.
  \end{align*}
  Substitution into~\eqref{eq:tensor_ls_sensitivity_reduction} yields
  \begin{align*}
    \Lip(\psi_k;[0,\bar a_k])
    &\leq
    C_pM_pR^{p-1}
    \left((1+x_k)^{p-1}+x_k(1+x_k)^{2p-3}\right)
    \\
    &\leq
    C_pM_pR^{p-1}(1+x_k)^{2p-2},
  \end{align*}
  which proves~\eqref{eq:tensor_merit_lipschitz} and the asserted
  continuity.

  It remains to bound the upper trial value $\bar a_k$. Theorem~\ref{thm:aep} gives
  $A_k\|\oF(\vz_k)+\vv_k\|\leq2R$ for $k\geq1$, and the same inequality
  is trivial for $k=0$. Hence the first term in
  \eqref{eq:tensor_ls_cap} is at most $4R/\varepsilon$. Multiplying the two
  terms in the maximum defining $\bar a_k$ by $M_pR^{p-1}$, using
  $M_p=(p+1)L_p$, and absorbing factors depending only on $p$ into $C_p$
  gives
  \[
    M_pR^{p-1}\bar a_k
    \leq
    C_p\max\left\{
      \frac{L_pR^p}{\varepsilon},
      \left(\frac{L_pR^p}{\varepsilon}\right)^{(p-1)/p}
    \right\}.
  \]
  This is~\eqref{eq:tensor_cap_upper_bound}.
\end{proof}

\subsection{Proof of Lemma~\ref{lem:tensor_trial_displacement}}
\label{appen:tensor_trial_displacement}
  Set $\vg_k:=\oF(\vz_k)+\vv_k$.
  We first record the higher-order analogue of
  Lemma~\ref{lem:trial_displacement_bound}. Note that in the proof of Lemma~\ref{lem:trial_displacement_bound}, except in the final step where we apply the Taylor remainder bound, the derivation only depends on the monotonicity of $\oP_{k,a}+\oH$. Hence, we also have
  \begin{equation}
    \begin{aligned}
          \|\vd_k(a)\|
    &\leq \|\vz_{k+\frac12}(a)-\vz_k\|
    +\frac{a^2}{A_k+a}\|\vg_k\|
    +a\|\oP_{k,a}(\vz_k)-\oF(\vz_k) \| \\
    & \leq
    \|\vz_{k+\frac12}(a)-\vz_k\|
    +\frac{a^2}{A_k+a}\|\vg_k\|
    +\frac{M_pa}{p!}
      \|\vz_{k+\frac12}(a)-\vz_k\|^p.
    \end{aligned} \label{eq:tensor_trial_displacement_bound}
  \end{equation}

  We shall also use the trajectory bounds in
  Lemma~\ref{lem:bounded_trajectory}. To justify their use here, observe
  that every accepted tensor step satisfies
  \[
    a_i\|\oF(\vz_{i+1})-\oP_i(\vz_{i+1})\|
    \leq
    \frac{M_pa_i}{p!}
      \|\vz_{i+1}-\vz_{i+\frac12}\|^p
    \leq
    \frac{1}{\sqrt{2}}
      \|\vz_{i+1}-\vz_{i+\frac12}\|,
  \]
  where the last inequality follows from the upper bound in
  \eqref{eq:tensor_acceptance_window}. Thus the accepted iterates form an
  AEP trajectory with $\rho=1/\sqrt{2}$, and the same proof as that of
  Lemma~\ref{lem:bounded_trajectory} gives
  \begin{equation}\label{eq:tensor_bounded_trajectory}
    \|\vz_0-\vz_k-A_k\vg_k\|\leq3\|\vz_0-\vz^*\|,
    \qquad
    \|\vz_k-\vz_0\|\leq3\|\vz_0-\vz^*\|.
  \end{equation}

  We now distinguish the initial iteration from the subsequent ones. If
  $k=0$, then $A_0=0$ and $\vz_{\frac12}(a)=\vz_0$. Choose
  $\vv^*\in\oH(\vz^*)$ such that
  $\oF(\vz^*)+\vv^*=0$. Comparing the trial inclusion at $\vz_1(a)$
  with the graph point
  $\oP_{0,a}(\vz^*)+\vv^*\in
  (\oP_{0,a}+\oH)(\vz^*)$ and using monotonicity gives
  \[
    \|\vz_1(a)-\vz^*\|
    \leq
    \|\vz_0-\vz^*\|+a\|\oP_{0,a}(\vz^*)-\oF(\vz^*)\|
    \leq
    \|\vz_0-\vz^*\|+\frac{M_pa}{p!}\|\vz_0-\vz^*\|^p.
  \]
  Consequently, for $0<a\leq\bar a_0$,
  \begin{equation*}
    \|\vd_0(a)\|
    \leq
    2\|\vz_0-\vz^*\|
    +\frac{M_p\bar a_0}{p!}\|\vz_0-\vz^*\|^p.
  \end{equation*}

  Now let $k\geq1$.
  Since iteration $k-1$ was accepted, its lower acceptance condition and
  the individual $(k-1)$th term in~\eqref{eq:aep_stability} imply
  \begin{align*}
    \frac{p!}{2}
    &\leq
    M_pa_{k-1}\|\vz_k-\vz_{k-\frac12}\|^{p-1}
    \leq
    2^{(p-1)/2}M_p\|\vz_0-\vz^*\|^{p-1}
    \frac{a_{k-1}^p}{A_k^{p-1}}
    \leq
    2^{(p-1)/2}M_p\|\vz_0-\vz^*\|^{p-1}A_k.
  \end{align*}
  It follows that
  \begin{equation}\label{eq:tensor_A_lower_bound}
    \frac{1}{A_k}
    \leq
    \frac{2^{(p+1)/2}}{p!}M_p\|\vz_0-\vz^*\|^{p-1}.
  \end{equation}
  Together with Theorem~\ref{thm:aep}, this yields
  \begin{equation}\label{eq:tensor_outer_residual_uniform}
    \|\vg_k\|
    \leq
    \frac{2\|\vz_0-\vz^*\|}{A_k}
    \leq
    \frac{2^{(p+3)/2}}{p!}M_p\|\vz_0-\vz^*\|^p.
  \end{equation}
  Moreover,~\eqref{eq:AEP_k+1/2} and
  \eqref{eq:tensor_bounded_trajectory} give
  \[
    \|\vz_{k+\frac12}(a)-\vz_k\|
    =
    \frac{a}{A_k+a}\|\vz_0-\vz_k-A_k\vg_k\|
    \leq
    3\|\vz_0-\vz^*\|\frac{a}{A_k+a}.
  \]
  Substituting the last two bounds into
  \eqref{eq:tensor_trial_displacement_bound}, and using
  $a^2/(A_k+a)\leq a$ and $a/(A_k+a)\leq1$, gives, for
  $0<a\leq\bar a_k$,
  \begin{equation*}
    \|\vd_k(a)\|
    \leq
    \left[
      3+\frac{2^{(p+3)/2}+3^p}{p!}
        M_p\|\vz_0-\vz^*\|^{p-1}\bar a_k
    \right]\|\vz_0-\vz^*\|.
  \end{equation*}
  The two cases imply~\eqref{eq:tensor_uniform_trial_displacement}, for
  example with
  \[
    C_p:=3+\frac{2^{(p+3)/2}+3^p}{p!}.
  \]

\subsection{Proof of Lemma~\ref{lem:tensor_trial_map_lipschitz}}
\label{appen:tensor_trial_map_lipschitz}

We first present the following lemma, which bounds the difference of the regularized Taylor model when changing its center.

\begin{lemma}
  \label{lem:tensor_model_center_sensitivity}
  For $\vy,\vy',\vw\in\reals^d$, there is a constant $C_p>0$, depending
  only on $p$, such that
  \begin{equation}\label{eq:tensor_model_center_sensitivity}
    \left\|
      \oT_{\lambda}^{(p-1)}(\vw;\vy)
      -\oT_{\lambda}^{(p-1)}(\vw;\vy')
    \right\|
    \leq
    C_pM_p\|\vy-\vy'\|
    \left(
      \|\vw-\vy\|+\|\vy-\vy'\|
    \right)^{p-1}.
  \end{equation}
\end{lemma}

\begin{proof}

  Set
  \[
    \Delta:=\vy'-\vy,
    \qquad
    \vr_t:=\vw-\vy-t\Delta,
    \qquad 0\leq t\leq1.
  \]
   In the following, we will show that, at almost every $t$,
  \begin{equation}\label{eq:tensor_taylor_center_derivative}
    \frac{d}{dt}\oT^{(p-1)}(\vw;\vy+t\Delta)
    =
    \frac{1}{(p-1)!}\oG_t[\vr_t^{p-1}].
  \end{equation}
  The zeroth Taylor term satisfies
  \[
    \frac{d}{dt}\oF(\vy+t\Delta)
    =D\oF(\vy+t\Delta)[\Delta].
  \]
  Let $\vr_t^j$ denote $j$ copies of
  $\vr_t$.
  For $j=1,\ldots,p-2$, by the chain rule, differentiating
  the $j$th Taylor term gives
  \begin{align*}
    \frac{d}{dt}
    \left(
      \frac{1}{j!}D^j\oF(\vy+t\Delta)[\vr_t^j]
    \right)
    &=
    \frac{1}{j!}D^{j+1}\oF(\vy+t\Delta)[\Delta,\vr_t^j]
    \\
    &\quad
    +\frac{1}{j!}\sum_{\ell=1}^j
      D^j\oF(\vy+t\Delta)
      [\vr_t^{\ell-1},-\Delta,\vr_t^{j-\ell}]
    \\
    &=
    \frac{1}{j!}D^{j+1}\oF(\vy+t\Delta)[\Delta,\vr_t^j]
    -\frac{1}{(j-1)!}
      D^j\oF(\vy+t\Delta)[\Delta,\vr_t^{j-1}].
  \end{align*}
  The last equality uses the fact that \(D^j\oF(\vy+t\Delta)\) is a symmetric \(j\)-linear map. Therefore, it is invariant under permutations of its \(j\) arguments.
 For the highest-order Taylor term, note that the map $t\mapsto D^{p-1}\oF(\vy+t\Delta)$ is
  $L_p\|\Delta\|$-Lipschitz and hence absolutely continuous. Therefore,
  its derivative
  \[
    \oG_t:=\frac{d}{dt}D^{p-1}\oF(\vy+t\Delta)
  \]
  exists for almost every $t$ and satisfies
  $\|\oG_t\|_{\op}\leq L_p\|\Delta\|$. Hence,
  \begin{align*}
    &\frac{d}{dt}
    \left(
      \frac{1}{(p-1)!}
      D^{p-1}\oF(\vy+t\Delta)[\vr_t^{p-1}]
    \right)
    \\
    &\qquad=
    \frac{1}{(p-1)!}\oG_t[\vr_t^{p-1}]
    -\frac{1}{(p-2)!}
      D^{p-1}\oF(\vy+t\Delta)[\Delta,\vr_t^{p-2}].
  \end{align*}
  Summing these identities, for each $j=1,\ldots,p-1$ the negative
  term involving $D^j\oF(\vy+t\Delta)$ cancels the positive term obtained from
  differentiating the $(j-1)$th Taylor coefficient. Consequently, we have~\eqref{eq:tensor_taylor_center_derivative}.
  Integrating from $0$ to $1$ and using
  $\|\vr_t\|\leq\|\vw-\vy\|+\|\Delta\|$ yields
  \begin{align}
    \|\oT^{(p-1)}(\vw;\vy')
      -\oT^{(p-1)}(\vw;\vy)\|
    &\leq
    \frac{L_p}{(p-1)!}\|\Delta\|
    \int_0^1\|\vr_t\|^{p-1}\,dt
    \notag\\
    &\leq
    \frac{L_p}{(p-1)!}\|\vy'-\vy\|
    \left(
      \|\vw-\vy\|+\|\vy'-\vy\|
    \right)^{p-1}.
    \label{eq:tensor_unregularized_center_sensitivity}
  \end{align}
  Moreover, the map $\vq\mapsto\|\vq\|^{p-1}\vq$ satisfies
  \[
    \bigl\|\|\vq\|^{p-1}\vq
      -\|\vq'\|^{p-1}\vq'\bigr\|
    \leq
    p\bigl(\|\vq\|+\|\vq'\|\bigr)^{p-1}
    \|\vq-\vq'\|.
  \]
  Apply the latter inequality with $\vq=\vw-\vy$ and
  $\vq'=\vw-\vy'$. Since
  \[
    \|\vq-\vq'\|=\|\vy-\vy'\|,
    \qquad
    \|\vq\|+\|\vq'\|
    \leq2\left(\|\vw-\vy\|+\|\vy-\vy'\|\right),
  \]
  multiplying by $\lambda/(p-1)!$ and combining with
  \eqref{eq:tensor_unregularized_center_sensitivity} proves
  \eqref{eq:tensor_model_center_sensitivity}, because
  $\lambda=L_p\leq M_p$.

\end{proof}

\begin{proof}[Proof of Lemma~\ref{lem:tensor_trial_map_lipschitz}]
  For $t>0$, define the trial weight and center
  \[
    \alpha(t):=\frac{t}{A_k+t},
    \qquad
    \vc(t):=\vz_k+\alpha(t)(\vz_0-\vz_k).
  \]
  Equation~\eqref{eq:AEP_k+1/2} gives
  \begin{equation}\label{eq:tensor_trial_predictor_formula}
    \vz_{k+\frac12}(t)
    =
    \vz_k+\alpha(t)
    \bigl(
      \vz_0-\vz_k-A_k(\oF(\vz_k)+\vv_k)
    \bigr).
  \end{equation}
  Hence~\eqref{eq:tensor_bounded_trajectory} implies
  \begin{equation}\label{eq:tensor_trial_predictor_sensitivity}
    \|\vz_{k+\frac12}(a)-\vz_{k+\frac12}(b)\|
    \leq
    3\|\vz_0-\vz^*\||\alpha(a)-\alpha(b)|.
  \end{equation}
  We shall use the two elementary estimates
  \begin{equation}\label{eq:tensor_trial_weight_sensitivity}
    |\alpha(a)-\alpha(b)|
    \leq
    \frac{|a-b|}{A_k}
    \leq
    \frac{2^{(p+1)/2}}{p!}
    M_p\|\vz_0-\vz^*\|^{p-1}|a-b|,
    \qquad
    b|\alpha(a)-\alpha(b)|\leq|a-b|,
  \end{equation}
  where the second inequality in the first bound follows from
  \eqref{eq:tensor_A_lower_bound}.

  As in the second-order argument, we first derive an alternative bound on the
  trial displacement $\vd_k$. Combining
  \eqref{eq:tensor_trial_displacement_bound},
  \eqref{eq:tensor_A_lower_bound},
  \eqref{eq:tensor_outer_residual_uniform}, and
  \[
    \|\vz_{k+\frac12}(a)-\vz_k\|
    \leq3\|\vz_0-\vz^*\|\frac{a}{A_k+a}
  \]
  gives
  \begin{equation}\label{eq:tensor_trial_displacement_parameter_bound}
    \|\vd_k(a)\|
    \leq C_pM_p\|\vz_0-\vz^*\|^pa.
  \end{equation}
  Moreover, since
  \[
    \vc(a)-\vz_{k+\frac12}(a)
    =\frac{aA_k}{A_k+a}(\oF(\vz_k)+\vv_k),
  \]
  the last estimate and Theorem~\ref{thm:aep}, together with
  \eqref{eq:tensor_A_lower_bound}, imply
  \begin{equation}\label{eq:tensor_trial_center_gap}
    \|\vc(a)-\vz_{k+1}(a)\|
    \leq
    \|\vc(a)-\vz_{k+\frac12}(a)\|+\|\vd_k(a)\|
    \leq C_pM_p\|\vz_0-\vz^*\|^pa.
  \end{equation}

  We now compare two trial points. The proof of
  Lemma~\ref{lem:trial_sensitivity} uses only monotonicity of $\oP_{k,a}+\oH$ and therefore applies here, giving
  \begin{align}
    \|\vz_{k+1}(a)-\vz_{k+1}(b)\|
    &\leq
    |\alpha(a)-\alpha(b)|\|\vz_0-\vz_k\|
    +b\|\oP_{k,b}(\vz_{k+1}(a))
      -\oP_{k,a}(\vz_{k+1}(a))\|
    \notag\\
    &\quad
    +\frac{|a-b|}{a}\|\vc(a)-\vz_{k+1}(a)\|.
    \label{eq:tensor_two_trial_sensitivity}
  \end{align}
  Apply Lemma~\ref{lem:tensor_model_center_sensitivity} with
  $\vy=\vz_{k+\frac12}(a)$,
  $\vy'=\vz_{k+\frac12}(b)$, and $\vw=\vz_{k+1}(a)$. By
  the definition of $\vd_k(a)$, this gives
  \begin{align*}
    &b\|\oP_{k,b}(\vz_{k+1}(a))
      -\oP_{k,a}(\vz_{k+1}(a))\|
    \\
    &\qquad\leq
    C_pM_pb
    \|\vz_{k+\frac12}(a)-\vz_{k+\frac12}(b)\|
    \Bigl(
      \|\vd_k(a)\|
      +\|\vz_{k+\frac12}(a)-\vz_{k+\frac12}(b)\|
    \Bigr)^{p-1}.
  \end{align*}
  The predictor estimate~\eqref{eq:tensor_trial_predictor_sensitivity}
  implies
  \[
    \|\vz_{k+\frac12}(a)-\vz_{k+\frac12}(b)\|
    \leq
    3\|\vz_0-\vz^*\||\alpha(a)-\alpha(b)|.
  \]
  Moreover, since $|\alpha(a)-\alpha(b)|\leq1$,
  Lemma~\ref{lem:tensor_trial_displacement} gives
  \begin{align*}
    \|\vd_k(a)\|
      +\|\vz_{k+\frac12}(a)-\vz_{k+\frac12}(b)\|\leq
    C_p\|\vz_0-\vz^*\|
    \left(
      1+M_p\|\vz_0-\vz^*\|^{p-1}\bar a_k
    \right).
  \end{align*}
  Combining the last three displays yields
  \begin{align*}
    &b\|\oP_{k,b}(\vz_{k+1}(a))
      -\oP_{k,a}(\vz_{k+1}(a))\|
    \\
    &\qquad\leq
    C_pM_p\|\vz_0-\vz^*\|^p
    \left(
      1+M_p\|\vz_0-\vz^*\|^{p-1}\bar a_k
    \right)^{p-1}
    b|\alpha(a)-\alpha(b)|
    \\
    &\qquad\leq
    C_pM_p\|\vz_0-\vz^*\|^p
    \left(
      1+M_p\|\vz_0-\vz^*\|^{p-1}\bar a_k
    \right)^{p-1}
    |a-b|.
  \end{align*}
  The final inequality uses
  $b|\alpha(a)-\alpha(b)|\leq|a-b|$ from
  \eqref{eq:tensor_trial_weight_sensitivity}.
  Substituting this estimate,
  \eqref{eq:tensor_trial_weight_sensitivity},
  \eqref{eq:tensor_bounded_trajectory}, and
  \eqref{eq:tensor_trial_center_gap} into
  \eqref{eq:tensor_two_trial_sensitivity} gives
  \[
    \|\vz_{k+1}(a)-\vz_{k+1}(b)\|
    \leq
    C_pM_p\|\vz_0-\vz^*\|^p
    \left(
      1+M_p\|\vz_0-\vz^*\|^{p-1}\bar a_k
    \right)^{p-1}
    |a-b|.
  \]
  The predictor estimate
  in \eqref{eq:tensor_trial_predictor_sensitivity} and the first bound in
  \eqref{eq:tensor_trial_weight_sensitivity} similarly give
  \[
    \|\vz_{k+\frac12}(a)-\vz_{k+\frac12}(b)\|
    \leq C_pM_p\|\vz_0-\vz^*\|^p|a-b|.
  \]
  Hence
  \[
    \|\vd_k(a)-\vd_k(b)\|
    \leq
    C_pM_p\|\vz_0-\vz^*\|^p
    \left(
      1+M_p\|\vz_0-\vz^*\|^{p-1}\bar a_k
    \right)^{p-1}
    |a-b|,
    \qquad a,b>0.
  \]
  Finally,~\eqref{eq:tensor_trial_displacement_parameter_bound} implies
  $\vd_k(a)\to0$ as $a\downarrow0$. Since $\vd_k(0)=0$, the same bound
  extends to $[0,\bar a_k]$ and proves
  \eqref{eq:tensor_trial_map_lipschitz}.

\end{proof}

\section{Proofs for the Complexity Lower Bounds}
\label{sec:proof_lower_bound}

\subsection{Proof of Lemma~\ref{lem:phi-a-b}}

Let $\rho\colon \reals \rightarrow [0,1]$ be an even $C^\infty$ function such that
\begin{equation*}
    \rho(u) = 0 \quad \text{if }|u|\leq \frac{1}{2}, \quad \rho(u) = 1 \quad \text{if } |u| \geq 1.
\end{equation*}
Define the constant $B_r := \|\rho^{(r)}\|_{\infty} < \infty$, which is finite since $\rho^{(r)}$ is supported on the compact set $[-1,1]$. Now, we set $\varphi_a(s) = a \int_{0}^{s/a} \rho(u)\, du$.
As $\rho$ is constantly zero near the origin, so is $\varphi_a$, and clearly $\varphi_a^{(k)}(0) = 0$ for $0\le k\le r$, which proves (i).
Also, since $\rho$ is even, $\varphi_a$ is an odd function such that
\begin{equation}
\label{eq:phi-derivatives}
    \varphi^{(k+1)}_a(s) = a^{-k} \rho^{(k)}(s/a) , \quad k \ge 0 .
\end{equation}
Taking $k=0$ in \eqref{eq:phi-derivatives} proves (ii) because $0\le \rho \le 1$, and taking $k=r$ in \eqref{eq:phi-derivatives} proves (iii) because $\Lip(\varphi_a^{(r)}) \leq \|\varphi^{(r+1)}_a\|_{\infty} =B_r /a^r$.
Next, because $\rho(u) = 1$ for $|u| \geq 1$, for $s\geq a$, we have
\begin{align*}
    \varphi_a(s) & = a \int_{0}^{s/a} \rho(u)\, du = a \int_{0}^{1} \rho(u)\, du + a \int_{1}^{s/a} \rho(u)\, du \\
    & = s-a + a \int_{\frac{1}{2}}^{1} \rho(u)\, du
\end{align*}
so we have $\varphi_a (s) = s - b$ with $b = a - a \int_{\frac{1}{2}}^{1} \rho(u)\, du$, and $\frac{a}{2} \le b \le a$ because $0 \le \rho(u) \le 1$ for $\frac{1}{2} \le u \le 1$.
Because $\varphi_a$ is odd, this proves (iv).
Finally, (v) follows because $\varphi_a(s) = s-b$ for $s \geq a$ and $0 \leq \varphi_a'(s) \leq 1$ for all $s \in \reals$.

\subsection{Proof of Lemma~\ref{lem:G_operator}}

For (a), note that
\begin{equation*}
    D\oG(\vx) = -\varphi_a' (\inprod{\vv_{N+1}}{\vx}) \vv_1 \vv_{N+1}^\top + \sum_{i=1}^N \varphi_a' (\inprod{\vv_i}{\vx}) \vv_{i+1} \vv_i^\top .
\end{equation*}
Because $\vv_1, \dots, \vv_{N+1}$ are orthonormal, by Lemma~\ref{lem:phi-a-b}(ii) we have $\|D\oG(\vx)\|_{\op} \le 1$, i.e., $\oG$ is nonexpansive.
Similarly,
\begin{align*}
    D^{r}\oG(\vx) [\vh_1, \dots, \vh_r] & = -\varphi_a^{(r)} (\inprod{\vv_{N+1}}{\vx}) \left( \prod_{j=1}^r \inprod{\vv_{N+1}}{\vh_j} \right) \vv_1 \\
    & \quad + \sum_{i=1}^N \varphi_a^{(r)} (\inprod{\vv_i}{\vx}) \left( \prod_{j=1}^r \inprod{\vv_i}{\vh_j} \right) \vv_{i+1}
\end{align*}
so Lemma~\ref{lem:phi-a-b}(iii) implies that $D^{r}\oG$ is $\frac{B_r}{a^r}$-Lipschitz continuous.

To prove (b), observe that any fixed point of $\oG$ is of the form $\vx = \sum_{i=1}^{N+1} t_i \vv_i$ because it should belong to the image of $\oG$.
Additionally, it must satisfy
\begin{equation}
\label{eq:worst-case-fixed-point-condition}
    t_{i+1} = \varphi_a (t_i) \,\, (i=1,\dots,N), \quad t_1 + \varphi_a (t_{N+1}) = c .
\end{equation}
Because $s_i \ge M \ge a$ for all $i=1,\dots,N+1$, by (iv), $\varphi_a (s_i) = s_i - b = s_{i+1}$ and $s_1 + \varphi_a (s_{N+1}) = M + Nb + M - b = c$, satisfying \eqref{eq:worst-case-fixed-point-condition}.
These are, in fact, the unique coefficients satisfying \eqref{eq:worst-case-fixed-point-condition}, because this condition implies $t_1 + \varphi_a^{\circ (N+1)}(t_1) = c$, where $\varphi_a^{\circ(N+1)}$ is the $(N+1)$-fold composition of $\varphi_a$ with itself.
Because the map $t \mapsto t + \varphi_a^{\circ (N+1)}(t)$ is strictly increasing, this uniquely determines $t_1$ and hence the other coefficients.
Together, the above arguments show that $\vx^*$ is the unique fixed point.

Finally, for (c), let $z_i = \langle \vv_i, \vx \rangle$ for $i=1,\dots,N+1$ and suppose $z_{N+1} = 0$. Then
\begin{equation*}
    \oG(\vx) = (c-\varphi_a(z_{N+1})) \vv_1 + \sum_{i=1}^N \varphi_a(z_i)\vv_{i+1} = c \vv_1 + \sum_{i=1}^N \varphi_a(z_i)\vv_{i+1}
\end{equation*}
and hence,
\begin{align*}
    \|\vx - \oG(\vx)\| & \geq \sqrt{(c-z_1)^2 + \sum_{i=1}^N (\varphi_a(z_i) - z_{i+1})^2} \\
    & \geq \frac{c-z_1 + \sum_{i=1}^N (\varphi_a(z_i) - z_{i+1})}{\sqrt{N+1}} \\
    & \geq \frac{c-z_1 + \sum_{i=1}^N ((z_i-b) - z_{i+1})}{\sqrt{N+1}} = \frac{2M - b}{\sqrt{N+1}},
\end{align*}
where the second line uses the Cauchy--Schwarz inequality, and the last line follows from Lemma~\ref{lem:phi-a-b}(v).

\subsection{Proof of \cref{thm:higher_order_lower_bound}}
\label{subsec:proof_monotone_lower_bound}

Set $r=p-1$, $a = B_r^{1/r}$ and $M = a(N+1)$.
With this choice of parameters, the operator $\oG$ in \cref{lem:G_operator} is nonexpansive and $D^r\oG$ is $1$-Lipschitz, regardless of how $\vv_1, \dots, \vv_{N+1}$ are chosen.
Moreover, it has a unique fixed point $\vx^* = \sum_{i=1}^{N+1} s_i \vv_i$ with $s_i = M + (N+1-i)b \le (a+b)(N+1)$, and it satisfies
\begin{equation}
\label{eqn:worst-case-operator-fixed-point-norm-bound}
    \norm{\vx^*} = \sqrt{\sum_{i=1}^{N+1} s_i^2} \le \sqrt{(N+1)^3 (a+b)^2} \le 2a(N+1)^{3/2} := R_0
\end{equation}
where we used $b\le a$.
Now, let $\sigma = \frac{R}{R_0}$.
Given a deterministic $p$th-order method for monotone inclusions, we will progressively reveal the directions $\vv_1, \dots, \vv_{N+1}$ according to the method's output, following the resisting-oracle argument of \cite{NemirovskiYudin1983_problem}.
In the end, once all $\vv_i$ are determined, our worst-case operator $\oF$ will take the form
\begin{equation}
\label{eqn:worst-case-F}
    \oF(\vz) = L_p \sigma^p (\oI - \oG) \left(\frac{\vz - \vz_0}{\sigma} \right) ,
\end{equation}
which is monotone because $\oG$ is nonexpansive.
Furthermore, writing $\vx = \frac{\vz - \vz_0}{\sigma}$, the derivative oracles output
\begin{align*}
    \oF(\vz) &= L_p\sigma^p(\vx-\oG(\vx)),\\
    D\oF(\vz) &= L_p\sigma^{p-1}(\mI-D\oG(\vx)),\\
    D^j\oF(\vz) &= -L_p\sigma^{p-j}D^j\oG(\vx) ,
       \quad j=2,\dots,p-1.
\end{align*}
Hence, for arbitrary $\vz'$ and $\vx' = \frac{\vz' - \vz_0}{\sigma}$, we have
\begin{align*}
    \norm{D^{p-1} \oF(\vz) - D^{p-1} \oF(\vz')}_\mathrm{op}
    & = L_p \sigma \norm{D^{p-1} \oG(\vx) - D^{p-1} \oG(\vx')}_\mathrm{op} \\
    & \le L_p \sigma \norm{\vx - \vx'} = L_p \norm{\vz - \vz'}
\end{align*}
by \cref{lem:G_operator}(a), showing that $D^{p-1} \oF$ is $L_p$-Lipschitz.

Now, we initialize the given deterministic algorithm at $\vz_0$ and run it on \eqref{eqn:worst-case-F}, while selecting $\vv_{i+1}$ after the $i$th output of the algorithm is determined, so that
\begin{equation}
\label{eqn:vi-perpendicular-to-history}
    \vv_{i+1} \perp \operatorname{span}\{\vv_1,\dots,\vv_i,\vx_0,\dots,\vx_i\} := \mathcal{V}_i
\end{equation}
where $\vx_t = \frac{\vz_t - \vz_0}{\sigma}$ for $t = 0, \dots, N$.
This is possible because of the following two reasons: First, $\vx_0 = \frac{\vz_0 - \vz_0}{\sigma} = 0$ so $\dim \mathcal{V}_i \le 2i \le 2N < n$, leaving at least one direction to choose in $\mathcal{V}_i^\perp$.
Second, given that \eqref{eqn:vi-perpendicular-to-history} holds, for $t > i$, the coefficients $-\varphi_a(\inprod{\vv_{N+1}}{\vx_i})$ and $\varphi_a(\inprod{\vv_{t}}{\vx_i})$ appearing in the definition of $\oG(\vx_i)$ (expression obtained by substituting $\vx = \vx_i$ in \eqref{eq:def_G}) both vanish, and so do their derivatives up to order $p-1$ by \cref{lem:phi-a-b}(i).
Therefore, the newly chosen direction $\vv_{i+1}$ has no effect on the previous $D^j\oF$-oracle responses at $\vz_0,\dots,\vz_{i-1}$ for $j=0,\dots,p-1$, and does not alter the algorithm's history.

Now our construction implies $\inprod{\vv_{N+1}}{\vx_N} = 0$, so by \cref{lem:G_operator}(c),
\begin{equation*}
    \norm{\vx_N - \oG(\vx_N)} \ge \frac{2M-b}{\sqrt{N+1}} \ge \frac{M}{\sqrt{N+1}} = \frac{R_0}{2(N+1)}
\end{equation*}
where we use $b\le a\le M=a(N+1)$ and $R_0 = 2a(N+1)^{3/2}$.
Consequently, using $a = B_r^{1/r}$ with $r=p-1$, we obtain
\begin{align*}
    \norm{\oF(\vz_N)} & = \frac{L_p R^p }{R_0^p} \norm{\vx_N - \oG(\vx_N)} \\
    & \geq \frac{L_p R^p}{2R_0^{p-1}(N+1)} = \frac{1}{2^p B_{p-1}} \frac{L_p R^p}{(N+1)^{(3p-1)/2}} .
\end{align*}
Now note that $\vx^*$ is the unique fixed point of $\oG$, so $\vz^* := \vz_0+\sigma\vx^*$ is the unique zero of $\oF$, and we have $\norm{\vz_0-\vz^*} = \sigma\norm{\vx_0 - \vx^*} = \sigma \norm{\vx^*} \le \sigma R_0 = R$ by \eqref{eqn:worst-case-operator-fixed-point-norm-bound}.
This completes the proof.

\newpage

\printbibliography

\end{document}